\documentclass{mcom-l}

\usepackage{amsfonts}
\usepackage{amssymb}
\usepackage{graphicx}
\usepackage{epstopdf}
\usepackage{algpseudocode}
\usepackage{tensor}
\usepackage{verbatim}
\usepackage{mathtools}
\usepackage{hyperref}
\usepackage{cleveref}
\usepackage{tikz}
\usetikzlibrary{calc, shapes, angles, quotes, patterns, decorations.pathreplacing, 
	cd}
\usepackage{subcaption}
\usepackage{bm}
\usepackage{multirow}
\usepackage{dsfont}

\usepackage{amsopn}

\crefname{equation}{}{}
\Crefname{equation}{Equation}{Equations}
\crefname{lemma}{Lemma}{Lemmas}
\crefname{theorem}{Theorem}{Theorems}
\crefname{corollary}{Corollary}{Corollaries}
\crefname{figure}{Figure}{Figures}
\Crefname{figure}{Figure}{Figures}
\crefname{table}{Table}{Tables}
\Crefname{table}{Table}{Tables}

\DeclareMathOperator{\vcurl}{\mathbf{curl}}
\DeclareMathOperator{\grad}{\mathbf{grad}}
\DeclareMathOperator{\dive}{div}
\DeclareMathOperator{\spann}{span}

\newcommand{\bdd}[1]{ \boldsymbol{#1} }
\newcommand{\unitvec}[1]{\bdd{#1}}
\newcommand{\vertex}[1]{\bdd{{#1}}}
\newcommand{\tensorbd}[1]{\mathbf{{#1}}}
\newcommand{\symgrad}[1]{\tensorbd{D}{#1}}

\newcommand{\iseminormsup}[4]{\tensor*[_{#1}]{ \left|#2 \right| }{_{#3}^{#4} } }
\newcommand{\iseminorm}[3]{\iseminormsup{#1}{#2}{#3}{}}

\newcommand{\reftri}{\hat{T}}

\renewcommand{\d}[1]{\,\mathrm{d}{#1}}

\makeatletter
\let\orgdescriptionlabel\descriptionlabel
\renewcommand*{\descriptionlabel}[1]{%
	\let\orglabel\label
	\let\label\@gobble
	\phantomsection
	\edef\@currentlabel{#1}%
	\let\label\orglabel
	\orgdescriptionlabel{#1}%
}
\makeatother

\newcommand{\newinf}{\mathop{\inf\vphantom{\sup}}}

\newcommand{\vertiii}[1]{%
	{\left\vert\kern-0.25ex\left\vert\kern-0.25ex\left\vert #1 
		\right\vert\kern-0.25ex\right\vert\kern-0.25ex\right\vert}}

\definecolor{fourth}{RGB}{230, 159, 0}
\definecolor{fifth}{RGB}{86, 180, 233}
\definecolor{sixth}{RGB}{0, 158, 115}
\definecolor{pver}{RGB}{204, 121, 167}

\newtheorem{theorem}{Theorem}[section]
\newtheorem{lemma}[theorem]{Lemma}
\newtheorem{corollary}[theorem]{Corollary}

\theoremstyle{definition}

\theoremstyle{remark}
\newtheorem{remark}[theorem]{Remark}

\numberwithin{equation}{section}

\begin{document}
	

\title[Scott-Vogelius Stability]{Stability of high-order Scott-Vogelius elements 
for 2D non-Newtonian incompressible flow}


\author{Charles Parker}
\address{Ridgway Scott Foundation, 2544 S Forest Ln. 
	Cedarville, MI 49719}
\curraddr{}
\email{charles\_parker@alumni.brown.edu}
\thanks{This material is based upon work supported by 
	the National Science Foundation under Award No.DMS-2201487 
	and by the Mary Wheeler Fellowship from the Ridgway Scott Foundation.}

\author{Endre S\"{u}li}
\address{Mathematical Institute, 
	University of Oxford, Andrew Wiles Building, Woodstock Road, Oxford OX2 6GG, 
	UK}
\curraddr{}
\email{suli@maths.ox.ac.uk}
\thanks{}


\date{}

\dedicatory{}

\begin{abstract}
	We consider the stability of high-order Scott-Vogelius elements for 2D  
	non-Newtonian incompressible flow problems. For elements of degree 4 or 
	higher, we construct a right-inverse of the divergence operator that is 
	stable uniformly in the polynomial degree $N$ from $L^p$ to 
	$\boldsymbol{W}^{1,p}$, show that the associated inf-sup constant is 
	bounded below by a constant that decays at worst like $N^{-3\left| 
		\frac{1}{2} - \frac{1}{p}\right|}$, and 
	construct local Fortin operators with stability constants explicit in the
	polynomial degree. We demonstrate these results with several numerical 
	examples suggesting that the $p$-version method can offer superior 
	convergence rates over the $h$-version method even in the non-Newtonian 
	setting.
\end{abstract}

\maketitle


\section{Introduction}

Non-Newtonian fluids, particularly those with implicit constitutive laws 
\cite{Rajagopal03,Rajagopal06,Rajagopal08}, arise in a variety of engineering, 
biological, and industrial 
applications including slurries (Bingham) in drilling engineering, 
synovial fluid (shear-thinning) in animal joints, and adhesives (viscoelastic) 
in manufacturing. The mathematical analysis 
\cite{BLM25,BMM23,BPSS22,SuliDebiec23,SuliDebiec25,Gazca24,Gazca22,Kaltenbach23b,%
	Kaltenbach23a} 
of these models for incompressible fluids and the numerical analysis of 
corresponding finite element methods
\cite{Barrenechea23,Berselli24,Diening25,FarSuli22,Gazca23,Gazca25,%
	HeidSuliBing22,Heid22,Jessberger24,Kaltenbach23i,Kaltenbach23ii,%
	Kaltenbach23iii}
have received much attention recently. These finite element methods roughly 
fall into two categories: those that produce exactly (pointwise) 
divergence-free velocity approximations and those that do not. While many of 
the standard classical methods fall into the latter category, exactly 
divergence-free methods, typically in the Newtonian setting, have also been an 
area of active research; see e.g. \cite{JohnLinkeMerdonNeilReb17,Neilan20} for 
reviews.

Exactly divergence-free finite element discretizations offer a number of 
benefits over standard finite elements. 
By construction, the numerical approximation more accurately respects 
the underlying physics in the sense that mass is conserved pointwise. 
Moreover, for Newtonian fluids, these schemes are pressure robust 
\cite{JohnLinkeMerdonNeilReb17}, meaning that the velocity error is 
independent of the pressure error and is quasi-optimal.
Properly quantifying pressure robustness in the non-Newtonian setting
and showing that pressure robust schemes for Newtonian fluids remain 
pressure robust in non-Newtonian setting is an ongoing area of research. 
Recent results \cite{Diening25} show that some particular divergence-free 
finite element schemes based on the lowest-order Crouzeix-Raviart element are 
also pressure robust for the $p$-Stokes equation.

Divergence-free elements offer an additional benefit in the non-Newtonian 
setting related to the stabilization of convective terms. We 
briefly summarize the discussion in \cite[\S 3.2 \& 3.3]{Diening13}. For 
concreteness, consider the following $p$-Navier-Stokes problem for a power-law 
fluid:
\begin{subequations}
	\label{eq:p-navier-stokes}
	\begin{align}
		-\dive \tensorbd{S}(\symgrad{\bdd{u}}) + \dive(\bdd{u} \otimes 
		\bdd{u}) + \grad q &= \bdd{f} \qquad & 
		&\text{in } \Omega, \\
		-\dive \bdd{u} &= 0 \qquad & &\text{in } \Omega, \\
		\bdd{u} &= \bdd{0} \qquad & &\text{on } \partial \Omega,
	\end{align}
\end{subequations}
where $\Omega \subset \mathbb{R}^d$, $d \in \{2,3\}$, is a bounded, Lipschitz, 
polyhedral domain, $\symgrad{\bdd{v}}$ is the symmetric part of the 
gradient of $\bdd{v}$, and for fixed $p \in (1,\infty)$, $\tensorbd{S}$ is the 
shear stress defined by
\begin{align}
	\label{eq:extra-stress-tensor}
	\tensorbd{S}(\tensorbd{A}) := \nu |\tensorbd{A}|^{p-2} 
	\tensorbd{A}, \qquad \forall \tensorbd{A} \in \mathbb{R}^{d \times d},
\end{align}
where $\nu > 0$ is constant.
For an appropriate choice of Sobolev spaces, problem \cref{eq:p-navier-stokes} 
has a weak solution provided that $p > \frac{2d}{d+2}$ \cite{BGMWG09}. 
Typically, the weak form of the convective term $\dive(\bdd{u} \otimes 
\bdd{u})$ 
is rewritten as
\begin{align}
	\label{eq:convective-weak}
	\frac{1}{2} \left[ (\bdd{u} \otimes \bdd{v}, \grad \bdd{u} ) - ( \bdd{u} 
	\otimes \bdd{u}, \grad \bdd{v} ) \right],
\end{align}
where $(\cdot,\cdot)$ denotes the $L^2(\Omega)$-inner product, to ensure 
that discretizations inherit the skew-symmetry of the convective term (see 
e.g. \cite{Temam77}). 
For finite elements that do not deliver divergence-free velocities, the 
trilinear form \cref{eq:convective-weak} is only bounded in the proper norms 
uniformly in the discretization parameters provided that $p > \frac{2d}{d+1}$. 
Consequently, existing results for weak convergence for these methods require 
$p > \frac{2d}{d+1}$, while for divergence-free elements, 
one can prove weak convergence over the full range $p > \frac{2d}{d+2}$ for 
which the continuous problem is well-posed \cite[Theorem 19]{Diening13}. See, 
however, \cite{Barrenechea23}, where weak convergence of a stabilized finite 
element method has been shown for the full range of $p > \frac{2d}{d+2}$ for 
which weak solutions are known to exist. The 
existence of fluids for which $p \in \left(  \frac{2d}{d+2}, 
\frac{2d}{d+1}\right]$, e.g. shear-thinning nanofluids, which play a key role 
in the heat transfer in mirco-structures \cite{Gonzalez21}, shows that
the construction of methods that converge in the full range $p > 
\frac{2d}{d+2}$ is of practical importance.	

We focus on high-order 2D Scott-Vogelius elements \cite{ScottVog85}, which 
deliver exactly exactly divergence-free velocities. In 
particular, the velocity space consists of piecewise polynomials of degree $N 
\geq 1$ and the pressure space is taken to be the divergence of this space:
\begin{align}
	\label{eq:sv-def-nobc}
	\bdd{V}^N := \{ \bdd{v} \in C(\bar{\Omega})^2 : \bdd{v}|_{K} \in 
	\mathcal{P}_N(K)^2 \ \forall K \in \mathcal{T} \} 
	\quad \text{and} \quad 
	Q^{N-1} := \dive \bdd{V}^N,
\end{align}
where $\mathcal{T}$ is a triangulation of $\Omega \subset \mathbb{R}^2$ and 
$\mathcal{P}_N(K)$ denotes the space of polynomials of degree at most $N$. A 
local characterization of $Q^{N-1}$ is only available on 
general meshes if $N \geq 4$ \cite{ScottVog85,Vogelius83le}, which is the case 
we consider here. For meshes with 
special structure, local characterizations are available for lower degrees 
\cite{Qin92} e.g. on barycentrically refined meshes for $N \geq 2$  
\cite[Theorems 4.6.1 \& 6.4.1]{Qin92}. 

The stability of the discretization hinges on the behavior of the inf-sup 
constant
\begin{align}
	\label{eq:inf-sup-sv-nobc-def}
	\beta_{N, p} := \newinf_{ q \in Q^{N-1} } \sup_{ \bdd{v} \in 
		\bdd{V}^N } \frac{ (\dive \bdd{v}, q) }{ \|\bdd{v}\|_{1,p,\Omega} 
		\|q\|_{p',\Omega}  },
\end{align}
where $p' := p/(p-1)$ is the H\"{o}lder conjugate of $p$ and $\| \cdot 
\|_{1,p,\Omega}$ and $\|\cdot\|_{p', \Omega}$ denote the standard norms on the 
Sobolev space $W^{1,p}(\Omega)$ and Lebesgue space $L^{p'}(\Omega)$. 
For the power-law fluid in \cref{eq:p-navier-stokes}, $p$ is the same as the 
index in the shear stress \cref{eq:extra-stress-tensor}; for a general 
non-Newtonian 	fluid, the value of $p$ depends on the constitutive law. In 
the Newtonian setting 
($p=2$), $\beta_{N, 2}$ is bounded below uniformly in the mesh size and 
polynomial degree provided that $N \geq 4$ \cite{AinCP21LE}. Under the same 
condition, the Scott-Vogelius elements are also stable with respect to the 
mesh size in the non-Newtonian setting: $\beta_{N, p}$ is bounded below 
uniformly in the mesh size for $p \in (1, \infty)$ \cite[Lemma A.4 \& Lemma 
A.11]{Tscherpel18}. However, the behavior of $\beta_{N, p}$ with respect to the
polynomial degree $N$ for $p \neq 2$ has not previously been addressed for any 
finite element pair. 

In this work, we establish three novel results regarding the stability of 
high-order Scott-Vogelius elements on general meshes in the non-Newtonian 
setting. First, we show that for $N \geq 4$, there exists a right-inverse of 
the divergence operator $\mathcal{L}_{\dive, N} : Q^{N-1} \to \bdd{V}^N$ that 
is stable uniformly in the polynomial degree from $L^p(\Omega)$ to 
$\bdd{W}^{1, 
	p}(\Omega)$. Second, we show that the inf-sup constant 
$\beta_{N, p}$ is bounded below by a constant the decays like $N^{-3\left| 
	\frac{1}{2} - \frac{1}{p} \right|}$. Third, we construct local Fortin 
operators, which are ubiquitous in the numerical analysis of finite element 
methods for non-Newtonian flows, with stability constants explicit in the
polynomial degree. 
We believe these results to be the first with stability constants explicit 
in the 
polynomial degree for any finite element method for non-Newtonian flow, and 
they are foundational in understanding the convergence behavior of $p$- and 
$hp$-finite element methods for non-Newtonian flows.\footnote{Note that the 
	``$p$" in ``$p$- and $hp$-finite element methods" and ``$p$-version" 
	refers to the polynomial 
	degree of the spaces, which is ``$N$" in our notation, and clashes with our 
	notation for ``$L^p$" spaces. In the remainder of the manuscript, we 
	continue 
	to use the standard nomenclature ``$p$- and $hp$-finite element methods" and 
	``$p$-version" without confusion.} We also observe in some numerical 
examples in \cref{sec:p-stokes-numerics} that the $p$-version method can 
offer superior convergence rates over the $h$-version method for some 
smooth and nonsmooth solutions. 

The remainder of the paper is organized as follows. In \cref{sec:sv-review}, 
we review existing results for high-order Scott-Vogelius elements. The main 
contributions are summarized in \cref{sec:main-results} and applied to 
the $p$-Stokes system in \cref{sec:p-stokes-app}, where numerical results 
highlight several gaps in the existing theory. Each of the final three 
sections develops the theory for each main result: the right-inverse of the 
divergence operator $\mathcal{L}_{\dive, N}$ is constructed in 
\cref{sec:invert-div}, a lower bound for the inf-sup constant $\beta_{N, p}$ 
is proved in \cref{sec:lp-stability-l2-projection}, and the local Fortin 
operators are constructed in \cref{sec:local-fortin}.

\section{Scott-Vogelius FEM discretization}
\label{sec:sv-review}

We begin with a review of the Scott-Vogelius element \cite{ScottVog85} and its 
stability properties. Let $\Omega \subset \mathbb{R}^2$ be a polygonal domain 
with $\Gamma_0 \subset \partial \Omega$ consisting of a nonzero finite number 
of open 
segments of $\partial \Omega$ and $\Gamma_1 := \partial \Omega \setminus 
\Gamma_0$. For $p \in [1,\infty]$ and $s \in (0, \infty]$,
let $W^{s, p}(\Omega)$ denote the usual Sobolev spaces with norm 
$\|\cdot\|_{s,p,\Omega}$ and let $p':= p/(p-1)$ denote the H\"{o}lder 
conjugate of $p$. For $s = 0$, in which $W^{0,p}(\Omega) = L^p(\Omega)$, we 
also denote the norm by $\|\cdot\|_{p, \Omega}$. Spaces of vector fields are 
denoted by $\bdd{W}^{s,p}(\Omega)$, while spaces of tensors are denoted by 
$\tensorbd{W}^{s,p}(\Omega)$. The velocity space with zero trace on $\Gamma_0$ 
and the corresponding pressure space are given by
\begin{align}
	\bdd{W}^{1,p}_{\Gamma}(\Omega) &:= \{ \bdd{v} \in 
	\bdd{W}^{1,p}(\Omega) : 
	\bdd{v}|_{\Gamma_0} = \bdd{0} \} \\
	L^{p'}_{\Gamma}(\Omega) &:= \begin{cases}
		\{ q \in L^{p'}(\Omega) : (q, 1) = 0 \} & \text{ if 
		} |\Gamma_1| = 0, \\
		L^{p'}(\Omega) & \text{otherwise}.
	\end{cases}
\end{align}	
Note that the pressure space is mean-free in the case of enclosed flow 
($|\Gamma_1| = |\partial \Omega|$).

Let $\mathcal{T}$ denote a shape-regular triangulation of $\Omega$ satisfying 
the usual assumptions \cite[p. 38]{Cia02}. Additionally, we 
assume that the interior of every boundary edge of $\mathcal{T}$ lies either 
on $\Gamma_0$ or $\Gamma_1$. For $N \geq 1$, the Scott-Vogelius velocity and 
pressure space in the absence of boundary conditions are given by 
\cref{eq:sv-def-nobc}. The corresponding spaces with boundary conditions are
\begin{align}
	\label{eq:sv-def-bc}
	\bdd{V}^N_{\Gamma} &:= \bdd{V}^N \cap \bdd{W}^{1, 1}_{\Gamma}(\Omega)
	\quad \text{and} \quad
	Q^{N-1}_{\Gamma} := \dive \bdd{V}^N_{\Gamma}.
\end{align}

\subsection{Characterization of the pressure space}

The space $Q^{N-1}_{\Gamma}$ has a more explicit characterization on general 
meshes if the degree is sufficiently high ($N \geq 4$) 
\cite{AinCP21LE,ScottVog85}. We first note that for any $N \geq 1$, 
$Q^{N-1}_{\Gamma}$ is a subspace of (discontinuous) piecewise polynomials; i.e.
\begin{align}
	\label{eq:q-subspace-dg}
	Q^{N-1}_{\Gamma} \subseteq DG^{N-1}(\mathcal{T}) \cap L^1_{\Gamma}(\Omega),
\end{align}
where, given a collection of elements $\mathcal{U} \subseteq \mathcal{T}$, the 
space $ DG^N(\mathcal{U})$ is defined as follows.
\begin{align}
	\label{eq:dg-definition}
	DG^N(\mathcal{U}) &:= \{ q \in L^1(\Omega) : q|_{K} \in \mathcal{P}_N(K) \ 
	\forall K \in \mathcal{U} \}.
\end{align}
The inclusions in \cref{eq:q-subspace-dg} can be strict provided the mesh 
contains so-called ``singular" vertices. To describe these vertices, we 
require some additional notation.

\begin{figure}[htb]
	\centering
	\begin{subfigure}[b]{0.45\linewidth}
		\centering
		\begin{tikzpicture}[scale=0.525]
			
			\coordinate (a) at (0, 0);
			\coordinate (a0) at (4, 0.2);
			\coordinate (a1) at (3.1, 2.9);
			\coordinate (a2) at (-2.1, 3.9);
			\coordinate (a3) at (-4.2, -2);
			\coordinate (a4) at (2.9, -3.2);				
			
			\filldraw (a) circle (2pt) node[align=center,below]{$\vertex{a}$}
			-- (a0) circle (2pt) 	
			-- (a1) circle (2pt) 
			-- (a);
			\filldraw (a) circle (2pt) node[align=center,below]{}	
			-- (a1) circle (2pt) 
			-- (a2) circle (2pt) 
			-- (a);
			\filldraw (a) circle (2pt) node[align=center,below]{}	
			-- (a2) circle (2pt) 
			-- (a3) circle (2pt) 
			-- (a);
			\filldraw (a) circle (2pt) node[align=center,below]{}	
			-- (a4) circle (2pt) 
			-- (a0) circle (2pt) 
			-- (a);
			\draw[dashed] (a3) -- (a4);
			\draw (2.8, 1.4) node(K0){$K_1$};
			\draw (1.3, 2.7) node(K1){$K_2$};
			\draw (-2, 1.5) node(K1){$K_3$};
			\draw (-0.1, -0.8) node(Kdots){$\ldots$};
			\draw (2.5, -1.8) node(Km){$K_m$};
			\pic["$\theta_1$"{anchor=west}, draw, angle radius=0.8cm, angle 
			eccentricity=1] {angle=a0--a--a1};
			\pic["$\theta_2$"{anchor=south}, draw, angle radius=1cm, angle 
			eccentricity=1] {angle=a1--a--a2};
			\pic["$\theta_3$"{anchor=east}, draw, angle radius=0.8cm, angle 
			eccentricity=1] {angle=a2--a--a3};
			\pic["$\theta_m$"{anchor=west}, draw, angle radius=1cm, angle 
			eccentricity=1] {angle=a4--a--a0};
			
		\end{tikzpicture}
		\caption{}
		\label{fig:internal schema}
	\end{subfigure}
	\hfill
	\begin{subfigure}[b]{0.5\linewidth}
		\centering
		\begin{tikzpicture}[scale=0.525]
			\coordinate(a) at (0, 0);
			\coordinate(a0) at (6, -0.2);
			\coordinate(a1) at (5, 2);
			\coordinate(a2) at (3.4, 4);
			\coordinate(a3) at (1, 4.6);
			\coordinate(a4) at (-1.9, 4);
			
			\filldraw (a) circle (2pt) node[align=center,below]{$\vertex{a}$}
			-- (a0) circle (2pt) node[align=center,below]{}	
			-- (a1) circle (2pt) node[align=center,above]{}
			-- (a);
			\filldraw (a) circle (2pt) node[align=center,below]{}	
			-- (a1) circle (2pt) node[align=center,above]{}
			-- (a2) circle (2pt) node[align=center,below]{}
			-- (a);
			\filldraw (a) circle (2pt) node[align=center,below]{}	
			-- (a2) circle (2pt) node[align=center,above]{};
			\filldraw (a3) circle (2pt) node[align=center,below]{}
			-- (a);
			\filldraw (a) circle (2pt) node[align=center,below]{}	
			-- (a3) circle (2pt) node[align=center,above]{}
			-- (a4) circle (2pt) node[align=center,below]{}
			-- (a);
			
			\draw[dashed] (a2) -- (a3);

			\coordinate (a12) at ($(a)!2/3!(a1)$);
			\coordinate (a121) at ($(a)!2/3-1/sqrt(29)!(a1)$);
			\coordinate (a1205) at ($(a)!2/3-0.5/sqrt(29)!(a1)$);
			
			\coordinate (a22) at ($(a)!2/3!(a2)$);
			\coordinate (a221) at ($(a)!2/3-1/sqrt(27.56)!(a2)$);
			\coordinate (a2205) at ($(a)!2/3-0.5/sqrt(27.56)!(a2)$);
			
			\coordinate (a32) at ($(a)!2/3!(a3)$);
			\coordinate (a321) at ($(a)!2/3-1/sqrt(22.16)!(a3)$);
			\coordinate (a3205) at ($(a)!2/3-0.5/sqrt(22.16)!(a3)$);
			
			\coordinate (a42) at ($(a)!2/3!(a4)$);
			\coordinate (a421) at ($(a)!2/3-1/sqrt(17)!(a4)$);
			\coordinate (a4205) at ($(a)!2/3-0.5/sqrt(17)!(a4)$);

			\pic["$\theta_1$"{anchor=west}, draw, angle radius=1cm, angle 
			eccentricity=1] {angle=a0--a--a1};
			\pic["$\theta_2$"{anchor=west}, draw, angle radius=1.3cm, angle 
			eccentricity=1] {angle=a1--a--a2};
			\pic["$\theta_m$"{anchor=south}, draw, angle radius=1cm, angle 
			eccentricity=1] {angle=a3--a--a4};
			
			\draw (4.75, 1) node(K0){$K_1$};
			\draw (3.75, 2.6) node(K1){$K_2$};
			\draw (1.9, 3.75) node(Kdots){$\ldots$};
			\draw (-0.75, 3.5) node(Km){$K_m$};
			
			\filldraw[thick] (0, 0) -- ($(a0) + ($(a)!1/sqrt(17)!(a0)$) $);
			\filldraw[thick] (0, 0) -- ($(a4) + ($(a)!1/sqrt(17)!(a4)$) $);
			
			\draw ($(a0) + ($(a)!0.5/sqrt(17)!(a0)$) $) 
			node[align=center,below](G0){$\partial \Omega$};
			\draw ($(a4) + ($(a)!0.5/sqrt(17)!(a4)$) $) 
			node[align=center,left](Gm1){$\partial \Omega$};
		\end{tikzpicture}		
		\caption{}
		\label{fig:boundary schema}
	\end{subfigure}
	\caption{Notation for mesh around (a) an internal vertex $\vertex{a}$ and 
		(b) a boundary vertex $\vertex{a}$, each abutting $m = 
		|\mathcal{T}_{\vertex{a}}|$ elements.}
	\label{fig:patch-schema}
\end{figure}
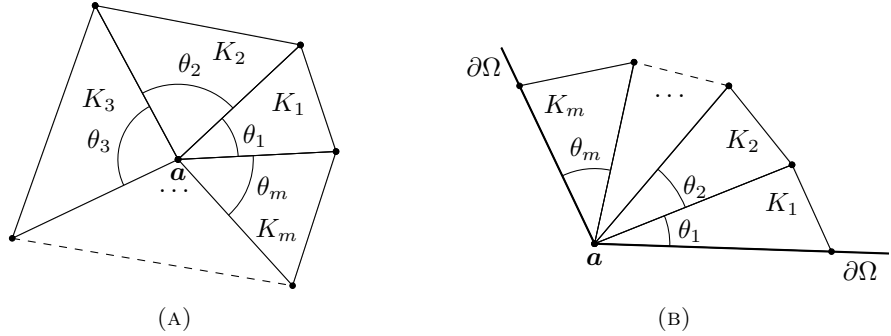

Let $\mathcal{V}$ denote the set of vertices of the mesh, partitioned into 
interior vertices $\mathcal{V}_I$ and boundary vertices $\mathcal{V}_B$. We 
partition $\mathcal{V}_B$ further into $\mathcal{V}_{B, 0}$ consisting of 
boundary vertices laying on the interior of $\Gamma_0$ and $\mathcal{V}_{B, 1} 
:= \mathcal{V}_B \setminus \mathcal{V}_{B, 0}$. Given a vertex $\vertex{a} \in 
\mathcal{V}$, we let $\mathcal{T}_{\vertex{a}}$ denote the set of elements 
abutting $\vertex{a}$ labeled as in \cref{fig:patch-schema}, and we define 
$\omega_{\vertex{a}} = \mathrm{int}(\cup_{K \in \mathcal{T}_{\vertex{a}}} 
\bar{K})$. A vertex $\vertex{a} \in \mathcal{V}$ is \textit{singular} if all 
element edges meeting at $\vertex{a}$ lie on two lines and $\vertex{a} \in 
\mathcal{V}_I \cup \mathcal{V}_{B, 0}$, and we denote by $\mathcal{V}_S$ the 
set of singular vertices; i.e.
\begin{align}
	\mathcal{V}_S := \{ \vertex{a} \in \mathcal{V}_I \cup \mathcal{V}_{B, 0} : 
	\theta_{i} + \theta_{i+1} = \pi, \ 1 \leq i \leq 
	|\mathcal{T}_{\vertex{a}}|-1 \}.
\end{align}
It is straightforward to see that if $\vertex{a} \in \mathcal{V}_S$, then 
$\vertex{a}$ abuts exactly four elements if $\vertex{a} \in \mathcal{V}_I$, 
otherwise $\vertex{a}$ abuts $1$, $2$, or $3$ elements if $\vertex{a} \in 
\mathcal{V}_{B, 0}$ (see \cref{fig:singular-vertices}). The space 
$Q^{N-1}_{\Gamma}$ has ``weak continuity" at singular vertices: For $q \in 
Q^{N-1}_{\Gamma}$, there holds
\begin{align}
	\label{eq:alternating-sum-condition}
	\mathcal{A}_{\vertex{a}}(q) := \sum_{i=1}^{|\mathcal{T}_{\vertex{a}}|} 
	(-1)^i q|_{K_i}(\vertex{a}) =  0 \qquad \forall \vertex{a} \in  
	\mathcal{V}_S.
\end{align}
Thanks to \cite[Theorem 4.1]{AinCP21LE}, condition 
\cref{eq:alternating-sum-condition} is the only compatibility condition 
satisfied by functions in $Q_{\Gamma}^{N-1}$ if $N \geq 4$.
\begin{lemma}
	For $N \geq 4$, there holds
	\begin{align}
		\label{eq:q-characterization}
		Q^{N-1}_{\Gamma} = \{ q \in DG^{N-1}(\mathcal{T}) \cap 
		L^1_{\Gamma}(\Omega) : \mathcal{A}_{\vertex{a}}(q) = 0 \ \forall 
		\vertex{a} \in \mathcal{V}_S \}.
	\end{align}
\end{lemma}
\begin{remark}
	Note that for $N \geq 4$, we have the identity $Q^{N-1}_{\Gamma} = Q^{N-1} 
	\cap L^1_{\Gamma}(\Omega)$ if and only if there are no singular vertices 
	on $\Gamma_{0}$ (i.e. $\mathcal{V}_S \cap \mathcal{V}_{B, 0} = \emptyset$).
\end{remark}

\begin{figure}[htb]
	\centering
	\begin{subfigure}{0.48\linewidth}
		\centering
		\begin{tikzpicture}[scale=0.525]
			
			\coordinate (a) at (0, 0);
			\coordinate (a0) at (4, 0);
			\coordinate (a1) at (1, 2);
			\coordinate (a2) at (-3, 0);
			\coordinate (a3) at (-1.5, -3);
			
			\filldraw (a) circle (3pt) node[align=center,below]{}
			-- (a0) circle (3pt) 	
			-- (a1) circle (3pt) 
			-- (a);
			\filldraw (a) circle (3pt) node[align=center,below]{}	
			-- (a1) circle (3pt) 
			-- (a2) circle (3pt) 
			-- (a);
			\filldraw (a) circle (3pt) node[align=center,below]{}	
			-- (a2) circle (3pt) 
			-- (a3) circle (3pt) 
			-- (a);
			\filldraw (a) circle (3pt) node[align=center,below]{}	
			-- (a3) circle (3pt) 
			-- (a0) circle (3pt) 
			-- (a);
			
			\draw ($(a)+(0.3, -0.3) $) node{$\vertex{a}$};
		\end{tikzpicture}
		\caption{}
		\label{fig:singular-vertex-interior}
	\end{subfigure}
	\hfill
	\begin{subfigure}{0.48\linewidth}
		\centering
		\begin{tikzpicture}[scale=0.525]
			
			\coordinate (a) at (3, 0);
			\coordinate (a0) at (0, 0);
			\coordinate (a1) at (0, 3);
			
			\coordinate (b0) at (-1.5, 0);
			\coordinate (b1) at (-1, 4);
			
			\filldraw (a) circle (3pt) node[align=center,below]{}
			-- (a0) circle (3pt) 	
			-- (a1) circle (3pt) 
			-- (a);
			
			\filldraw[thick] (a) -- (b0);
			\filldraw[thick] (a) -- (b1);
			
			\draw ($($(a0)!0.5!(b0)$) + (0, -0.4)$) node{$\Gamma_0$};
			\draw ($($(a1)!0.5!(b1)$) + (0.5, 0.2)$) node{$\Gamma_0$};
			
			\draw ($(a)+(0.3, -0.3) $) node{$\vertex{a}$};
			
		\end{tikzpicture}
		\caption{}
		\label{fig:singular-vertex-boundary-1}
	\end{subfigure} \\
	\begin{subfigure}{0.48\linewidth}
		\centering
		
		\begin{tikzpicture}[scale=0.525]
			
			\coordinate (a) at (0, 0);
			\coordinate (a0) at (4, 0);
			\coordinate (a1) at (1, 2);
			\coordinate (a2) at (-3, 0);
			
			\coordinate (b0) at (5, 0);
			\coordinate (b2) at (-4, 0);
			
			\filldraw (a) circle (3pt) node[align=center,below]{}
			-- (a0) circle (3pt) 	
			-- (a1) circle (3pt) 
			-- (a);
			\filldraw (a) circle (3pt) node[align=center,below]{}	
			-- (a1) circle (3pt) 
			-- (a2) circle (3pt) 
			-- (a);
			
			\filldraw[thick] (a) -- (b0);
			\filldraw[thick] (a) -- (b2);
			
			\draw ($(a)+(0, -0.4) $) node{$\vertex{a}$};
			
			\draw ($($(a0)!0.5!(b0)$) + (0, -0.4)$) node{$\Gamma_0$};
			\draw ($($(a2)!0.5!(b2)$) + (0, -0.4)$) node{$\Gamma_0$};
		\end{tikzpicture}
		
		\caption{}
		\label{fig:singular-vertex-boundary-2}
	\end{subfigure}
	\hfill
	\begin{subfigure}{0.48\linewidth}
		\centering
		\begin{tikzpicture}[scale=0.525]
			
			\coordinate (a) at (0, 0);
			\coordinate (a0) at (4, 0);
			\coordinate (a1) at (1, 2);
			\coordinate (a2) at (-3, 0);
			\coordinate (a3) at (-1, -2);
			
			\coordinate (b0) at (5, 0);
			\coordinate (b3) at (-1.5, -3);
			
			\filldraw (a) circle (3pt) node[align=center,below]{}
			-- (a0) circle (3pt) 	
			-- (a1) circle (3pt) 
			-- (a);
			\filldraw (a) circle (3pt) node[align=center,below]{}	
			-- (a1) circle (3pt) 
			-- (a2) circle (3pt) 
			-- (a);
			\filldraw (a) circle (3pt) node[align=center,below]{}	
			-- (a2) circle (3pt) 
			-- (a3) circle (3pt) 
			-- (a);
			
			\filldraw[thick] (a) -- (b0);
			\filldraw[thick] (a) -- (b3);

			\draw ($(a)+(0.3, -0.3) $) node{$\vertex{a}$};
			
			\draw ($($(a0)!0.5!(b0)$) + (0, -0.4)$) node{$\Gamma_0$};
			\draw ($($(a3)!0.5!(b3)$) + (0.5, -0.3)$) node{$\Gamma_0$};
		\end{tikzpicture}
		\caption{}
		\label{fig:singular-vertex-boundary-3}
	\end{subfigure}
	\caption{Possible configurations for a singular vertex $\vertex{a} \in 
		\mathcal{V}_S$}
	\label{fig:singular-vertices}
\end{figure}
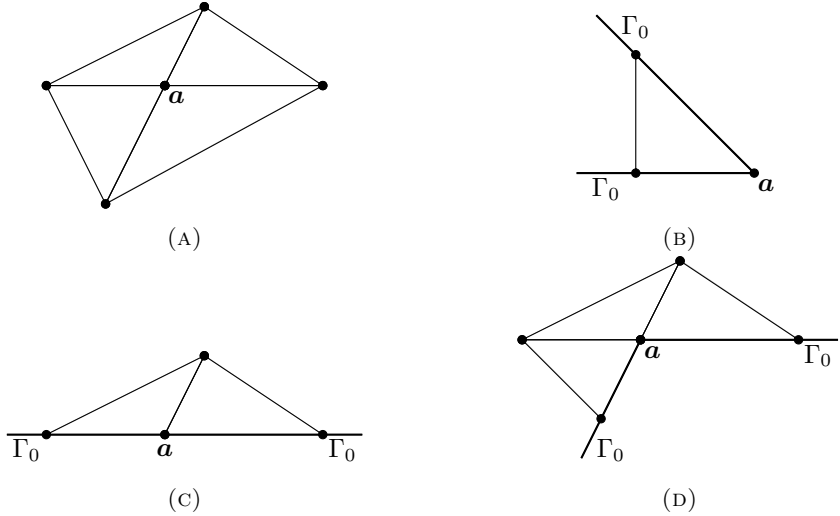

\subsection{Existing stability results}

The inf-sup constant with boundary conditions is the following analogue of 
\cref{eq:inf-sup-sv-nobc-def}:
\begin{align}
	\label{eq:inf-sup-sv-def}
	\beta_{N, p, \Gamma} := \newinf_{ q \in Q^{N-1}_{\Gamma} } \sup_{ \bdd{v} 
		\in 
		\bdd{V}^N_{\Gamma} } \frac{ (\dive \bdd{v}, q) }{ 
		\|\bdd{v}\|_{1,p,\Omega} 
		\|q\|_{p',\Omega}  }, \qquad p \in (1, \infty),
\end{align}
where we implicitly assume that $0$ is always excluded from the infimums and 
supremums. Since $\dive : \bdd{V}^N_{\Gamma} \to Q^{N-1}_{\Gamma}$ is 
surjective by construction, $\beta_{N, p, \Gamma} > 0$ for any $N \geq 1$ and 
$p \in (1, \infty)$ but may decay to 0 as the mesh is refined or as the 
polynomial degree is raised.

In the Newtonian case ($p=2$), sharp lower and upper bounds for the inf-sup 
constant $\beta_{N, 2}$ in \cref{eq:inf-sup-sv-def} explicit in the polynomial 
degree 
are known for $N \geq 4$. Related to the algebraic constraint in 
$Q^{N-1}_{\Gamma}$ 
\cref{eq:q-characterization}, we define the following geometric quantity at a 
vertex $\vertex{a}$:
\begin{align*}
	\xi_{\vertex{a}} := \sum_{i=1}^{|\mathcal{T}_{\vertex{a}}|-1} 
	|\sin(\theta_i + \theta_{i+1})| + \begin{cases}
		|\sin(\theta_{|\mathcal{T}_{\vertex{a}}|} + \theta_1)| & \text{if } 
		\vertex{a} \in \mathcal{V}_I, \\
		0 & \text{otherwise}.
	\end{cases} 
\end{align*}
Note that for $\vertex{a} \in \mathcal{V}_I \cup \mathcal{V}_{B, 0}$, 
$\xi_{\vertex{a}} = 0$ if and only if $\vertex{a} \in \mathcal{V}_S$; i.e.  
$\mathcal{V}_S = \{ \vertex{a} \in \mathcal{V}_I \cup \mathcal{V}_{B, 0} : 
\xi_{\vertex{a}} = 0 \}$. For a mesh $\mathcal{T}$, we define
\begin{align}
	\label{eq:xi-t-def}
	\xi_{\mathcal{T}} := \min_{ \vertex{a} \in (\mathcal{V}_I \cup 
		\mathcal{V}_{B, 0}) \setminus \mathcal{V}_{S} } \xi_{\vertex{a}}.
\end{align}
Then, \cite[Theorem 5.1]{AinCP21LE} gives the following lower bound on 
$\beta_{N, 2}$, while the upper bound follows from \cite[Theorem 6]{Grable24}.
\begin{lemma}
	\label{lem:inf-sup-p2}
	Let $N \geq 4$. Then, there exist constants $\beta_0, \beta_1, \tilde{\xi} 
	> 0$, depending only on $\Omega$, $\Gamma_0$, and shape regularity, such 
	that
	\begin{subequations}
		\begin{alignat}{2}
			\label{eq:inf-sup-p2-lower}
			\beta_{N, 2} &\geq \xi_{\mathcal{T}} \beta_0, \qquad & & \\
			\label{eq:inf-sup-p2-upper}
			\beta_{N, 2} &\leq \xi_{\mathcal{T}} \beta_1 \qquad & &\text{if } 
			\xi_{\mathcal{T}} < \tilde{\xi} \text{ and } |\Gamma_1| = 0.
		\end{alignat}
	\end{subequations}
\end{lemma}
In particular, the factor of $\xi_{\mathcal{T}}$ in \cref{eq:inf-sup-p2-lower} 
is sharp. For $p \in (1, \infty)$, one also has $\beta_{N, p} \geq \beta_0(N, 
p) \xi_{\mathcal{T}}$ \cite[Lemma A.4 \& Lemma 
A.11]{Tscherpel18}, where $\beta_0$ depends on $N$, $p$, shape regularity, 
$\Omega$, and $\Gamma_0$. One of the goals of this work is to quantify the 
dependence of this constant on the polynomial degree $N$.

\section{Main results}
\label{sec:main-results}

We now turn to the three main results: a discrete right-inverse of the 
divergence operator bounded uniformly in polynomial degree, a lower bound for 
the inf-sup constant $\beta_{N, p, \Gamma}$ \cref{eq:inf-sup-sv-def} that 
decays at 
worst like $N^{-3/2}$, and a local Fortin operator with continuity explicit in 
polynomial degree. For constants appearing in estimates, we use the notation 
$C(r, s, \ldots)$ to denote that, unless specified otherwise, the constant 
depends only on $r, s, \ldots$.

\subsection{Right-inverse of the divergence operator}

Our first result is the existence of a right-inverse of the divergence 
operator for each $N \geq 4$ mapping from $Q^{N-1}_{\Gamma}$ to 
$\bdd{V}^N_{\Gamma}$ that is $W^{1,p}(\Omega)$-stable uniformly in polynomial 
degree.	
\begin{theorem}
	\label{thm:invert-div}
	For $N \geq 4$, there exists a linear operator $\mathcal{L}_{\dive, N} : 
	Q_{\Gamma}^{N-1} \to \bdd{V}_{\Gamma}^{N}$ satisfying the following for 
	all $q \in Q_{\Gamma}^{N-1}$:
	\begin{align}
		\label{eq:invert-div-fem}
		\dive \mathcal{L}_{\dive, N} q = q \quad \text{and} \quad \|  
		\mathcal{L}_{\dive, N} q \|_{1, p, \Omega} \leq C(p) 
		\xi_{\mathcal{T}}^{-1} \|q\|_{p, \Omega} \qquad \forall p \in 
		(1,\infty),
	\end{align}
	where $C(p)$ also depends on $\Omega$, $\Gamma_0$, and shape regularity.
\end{theorem}
The proof of \cref{thm:invert-div} appears in \cref{sec:proof-invert-div}. In 
the case $p=2$, \cref{thm:invert-div} is equivalent to the inf-sup stability 
\cref{eq:inf-sup-p2-lower}. As we shall see shortly, the case $p\neq 2$ 
requires additional results.

\subsection{Inf-sup stability}

An immediate consequence of \cref{thm:invert-div} is the following lower 
bound on the inf-sup constant $\beta_{N, p, \Gamma}$ \cref{eq:inf-sup-sv-def}.
\begin{align}
	\label{eq:inf-sup-by-discrete-norm}
	\beta_{N, p, \Gamma} = \newinf_{ q \in Q^{N-1}_{\Gamma} } \sup_{ \bdd{v} 
		\in 
		\bdd{V}^N_{\Gamma} } \frac{ (\dive \bdd{v}, q) }{ 
		\|\bdd{v}\|_{1,p,\Omega} \|q\|_{p',\Omega}  } \geq 
	\frac{\xi_{\mathcal{T}}}{C(p)}  
	\newinf_{ q \in Q^{N-1}_{\Gamma} } \sup_{r \in 
		Q_{\Gamma}^{N-1}} \frac{ (q, r) }{ \|r\|_{p, \Omega} \|q\|_{p', 
			\Omega} },
\end{align}
where $C(p)$ is the same constant as in \cref{eq:invert-div-fem}. The final 
term of the RHS involves a discrete analogue of the dual $L^p$ norm, which we 
denote by
\begin{align}
	\label{eq:discrete-lpprime-norm}
	\vertiii{q}_{p', Q_{\Gamma}^{N-1}} := \sup_{r \in 
		Q_{\Gamma}^{N-1}} \frac{ (q, r) }{ \|r\|_{p, \Omega} } \qquad \forall 
	q \in Q_{\Gamma}^{N-1}.
\end{align}
Of course, if the supremum above is taken over all of $L^{p}(\Omega)$, then 
the above quantity is the $L^{p'}$-norm of $q$. Since $\dive 
\bdd{V}^N_{\Gamma} = Q^{N-1}_{\Gamma}$ and $\|\dive \bdd{v}\|_{1, p} \leq C 
\|\bdd{v}\|_{1, p}$, this discrete dual norm also appears in an upper-bound 
for $\beta_{N, p, \Gamma}$, as
\begin{align*}
	\beta_{N, p, \Gamma} \leq \frac{1}{C} \newinf_{ q \in Q^{N-1}_{\Gamma} } 
	\sup_{ 
		\bdd{v} \in \bdd{V}^N_{\Gamma} } \frac{ (\dive \bdd{v}, q) }{ 
		\|\dive \bdd{v}\|_{1,p,\Omega} \|q\|_{p',\Omega}  } 
	= \frac{1}{C} \newinf_{ q \in Q^{N-1}_{\Gamma} } 
	\frac{\vertiii{q}_{p', Q_{\Gamma}^{N-1}}}{\|q\|_{p', \Omega}}.
\end{align*}

In the case $p = 2$, one has $\vertiii{\cdot}_{2, Q^{N-1}_{\Gamma}} = 
\|\cdot\|_{2, \Omega}$, and so $\beta_{N, 2} \geq C \xi_{\mathcal{T}}$, as 
previously stated in \cref{eq:inf-sup-p2-lower}. For $p \neq 2$, one 
additionally needs to find the equivalence constants between $\|\cdot\|_{p', 
	\Omega}$ and  $\vertiii{\cdot}_{p', Q^{N-1}_{\Gamma}}$ on $Q^{N-1}_{\Gamma}$. 
Suppose that the $L^2$-projection operator $\mathbb{P}_{N-1}^{Q} : 
L^1(\Omega) \to Q_{\Gamma}^{N-1}$ is bounded in $L^p(\Omega)$; that is 
\begin{alignat*}{2}
	(\mathbb{P}_{N-1}^{Q} q, r) &= (q, r) \qquad & &\forall r \in 
	Q_{\Gamma}^{N-1}, \ \forall q \in L^1(\Omega), \\
	\|\mathbb{P}_{N-1}^{Q} r\|_{p, \Omega} &\leq \Delta(N, p) \|r\|_{p, 
		\Omega} \qquad & &\forall r \in L^{p}(\Omega),
\end{alignat*}
for some constant $\Delta(N, p)$. Then, we have
\begin{align}
	\label{eq:discrete-norm-by-projection}
	\vertiii{q}_{p', Q_{\Gamma}^{N-1}} \geq \frac{ (\mathbb{P}_{N-1}^{Q}(q 
		|q|^{p'-2}), q ) }{ \| \mathbb{P}_{N-1}^{Q}(q |q|^{p'-2}) \|_{p, \Omega} }
	\geq \frac{1}{\Delta(N, p)} \frac{ \|q\|_{p', \Omega}^{p'} }{ 
		\|q|q|^{p'-2}\|_{p, \Omega} } = \frac{1}{\Delta(N, p)} \|q\|_{p', 
		\Omega}.
\end{align}
Consequently, one route to proving inf-sup stability is to 
establish the $L^p$-stability of the $L^2$-projection onto $Q_{\Gamma}^{N-1}$, 
which is the route we pursue.

To this end, we place the following two mild restrictions on the mesh 
$\mathcal{T}$.		

\vspace{0.5em}

\begin{description}
	\item[(M1)\label{cond:mesh-disjoint-sing}] No two singular vertices share 
	an edge; i.e. if $\vertex{a}, \bdd{b} \in \mathcal{V}_S$ with $\vertex{a} 
	\neq \bdd{b}$, then $\mathcal{T}_{\vertex{a}} \cap \mathcal{T}_{\bdd{b}} = 
	\emptyset$.
	
	\vspace{0.5em}
	
	\item[(M2)\label{cond:mesh-no-odd-sing}] If $|\Gamma| = |\Gamma_0|$, then 
	no singular boundary vertex $\vertex{a} \in \mathcal{V}_{B, 0} \cap 
	\mathcal{V}_S$ coincides with a corner of $\Omega$; i.e. if $\vertex{a} 
	\in \mathcal{V}_{S} \cap \mathcal{V}_{B, 0}$, then 
	$|\mathcal{T}_{\vertex{a}}| = 2$.
\end{description}
\vspace{0.5em}
As we will see in \cref{sec:lp-stability-l2-projection} below, the first 
condition \ref{cond:mesh-disjoint-sing} ensures that 
$\mathbb{P}_{N-1}^{Q}$ can be decomposed into local projections that are 
defined on an individual element or a single vertex patch, while 
\ref{cond:mesh-no-odd-sing} ensures that $\mathcal{A}_{\vertex{a}}(1) = 0$ 
for all $\vertex{a} \in \mathcal{V}_{S}$, which is helpful in dealing with the 
zero-mean constraint in $Q_{\Gamma}^{N-1}$ when $|\partial \Omega| = 
|\Gamma_0|$. Under 
these conditions, the operator norm of $\mathbb{P}_{N-1}^{Q}$ grows at most 
like $N^{3/2}$:	
\begin{theorem}
	\label{thm:inf-sup-stability}
	If the mesh satisfies 
	\ref{cond:mesh-disjoint-sing}--\ref{cond:mesh-no-odd-sing} and $N \geq 4$, 
	then
	\begin{align}
		\label{eq:lp-norm-l2-projection}
		\|\mathbb{P}_{N-1}^{Q} r\|_{p, \Omega} &\leq C N^{3\left| \frac{1}{2} 
			- \frac{1}{p} \right|}  \|r\|_{p, 
			\Omega} \qquad \forall r \in L^{p}(\Omega), \ \forall p \in 
		[1,\infty],
	\end{align}
	where $C$ depends on shape regularity and if $|\Gamma| = |\Gamma_0|$, also 
	on $|\Omega|$. Consequently, 
	\begin{align}
		\label{eq:inf-sup-stability}
		\beta_{N, p, \Gamma} \geq C(p) N^{-3\left| \frac{1}{2} - \frac{1}{p} 
			\right|} 
		\xi_{\mathcal{T}} > 0 \qquad \forall p \in (1,\infty), 
	\end{align}
	where $C(p)$ also depends on  $\Omega$, $\Gamma_0$, and shape regularity 
	and $\beta_{N, p, \Gamma}$ is defined in \cref{eq:inf-sup-sv-def}.
\end{theorem}
The proof of \cref{thm:inf-sup-stability} appears in \cref{sec:proof-inf-sup}.

In comparison with the case $p=2$ \cref{eq:inf-sup-p2-lower}, the lower bound 
in 
\cref{eq:inf-sup-stability} in the case $p \neq 2$ converges to 0 at an 
algebraic rate as the polynomial degree is increased. We will see in 
\cref{sec:p-stokes-error-estimates} below that this potential instability 
leads 
to error estimates that are possibly suboptimal with respect to the polynomial 
degree $N$. However, the numerical experiments in \cref{sec:p-stokes-numerics} 
suggest that our lower bound in \cref{eq:inf-sup-stability} is pessimistic.

\begin{remark}
	The quantity $\xi_{\mathcal{T}}$ appearing in the stability estimate 
	\cref{eq:inf-sup-stability} has little practical importance. 
	In particular, given 
	an initial mesh $\mathcal{T}$ with vertices $\vertex{a} \in \mathcal{V}_I 
	\cup \mathcal{V}_{B, 0}$ such that $\xi_{\vertex{a}} \ll 1$, one may 
	perform a barycentric refinement of one or more elements in 
	$\mathcal{T}_{\vertex{a}}$. Denoting by $\tilde{\mathcal{T}}$ the modified 
	mesh, we have that $\xi_{\tilde{\mathcal{T}}}$ is uniformly bounded below 
	by a constant depending only on the shape regularity of $\mathcal{T}$ and 
	$\Omega$ and \ref{cond:mesh-disjoint-sing}--\ref{cond:mesh-no-odd-sing} 
	are automatically satisfied. We refer to \cite[Remark 2 \& Lemma 
	4.6]{AinCP21LE} for further discussion.
\end{remark}

\subsection{Fortin operators}

Local Fortin operators are one of the key tools in most of the existing 
analysis of finite element discretizations of non-Newtonain flows 
(see e.g. \cite{Belenki12,Diening25,Diening13,SuliTsch20}). As shown in 
\cite[Appendix A.2.2.2]{Tscherpel18}, one can 
modify the construction of a right-inverse of the divergence operator for the 
Scott-Vogelius elements to obtain a local Fortin operator. We proceed 
similarly here by modifying the proofs of 
\cref{thm:invert-div,thm:inf-sup-stability} to construct two local Fortin 
operators: one with continuity properties explicit in the polynomial degree 
that 
is a projection only in the lowest order ($N=4$) case and another that is a 
projection for any degree but with continuity constants implicit in the 
polynomial degree.

For these operators, we will assume that the mesh contains no singular 
vertices.

\vspace{0.5em}

\begin{description}
	\item[(M3)\label{cond:mesh-no-sing}] The mesh contains no singular 
	vertices; i.e. $\mathcal{V}_S = \emptyset$.
\end{description}

\vspace{0.5em}

\noindent Moreover, let $\mathcal{T}_K := \cup_{\vertex{a} \in \mathcal{V}_K} 
\mathcal{T}_{\vertex{a}}$, and $\omega_{K} := 
\mathrm{int}(\cup_{L \in \mathcal{T}_{K}} \bar{L})$. Then, the existence of 
the first Fortin operator is as follows.

\begin{theorem}
	\label{thm:fortin-nonprojection}
	Suppose that the mesh satisfies \ref{cond:mesh-no-sing} and let $N \geq 
	4$. There exists a linear operator $\mathcal{L}_{F, N} : 
	\bdd{W}^{1,1}(\Omega) \to \bdd{V}^N$ satisfying the following for $\bdd{v} 
	\in \bdd{W}^{1,1}(\Omega)$:
	\begin{subequations}
		\begin{alignat}{2}
			\label{eq:fortin-nonprojection-l2-projection}
			\int_{\Omega} (\dive \mathcal{L}_{F, N} \bdd{v}) q \d{\bdd{x}} &= 
			\int_{\Omega} (\dive \bdd{v}) q \d{\bdd{x}} \qquad & &\forall q 
			\in DG^{N-1}(\mathcal{T}), \\
			\mathcal{L}_{F, N} \bdd{v}|_{\Gamma_0} &= \bdd{0} \qquad & & 
			\text{if } \bdd{v} \in \bdd{W}^{1,1}_{\Gamma}(\Omega), \\
			\mathcal{L}_{F, N} \bdd{v} &= \bdd{v} \qquad & &\text{if } \bdd{v} 
			\in \bdd{V}^4.
		\end{alignat}	
	\end{subequations}
	Moreover, for all $p \in [1,\infty]$, $\epsilon \in (0, 1)$, $\bdd{v} \in 
	\bdd{W}^{1,p}(\Omega)$, and $K \in \mathcal{T}$, there holds
	\begin{subequations}
		\begin{align}
			\label{eq:fortin-nonprojection-lp-local}
			\| \mathcal{L}_{F, N} \bdd{v} \|_{p, K} &\leq \begin{cases}
				C(p) \left(   \|  \bdd{v} \|_{p, \omega_K} + h_K 
				\xi_{\mathcal{T}}^{-1} N^{3\left|\frac{1}{2} - 
					\frac{1}{p}\right|} |  \bdd{v} |_{1, p, \omega_K} \right) & 
				\text{if } p \in (1,\infty), \\
				C(\epsilon) \left(   \|  \bdd{v} \|_{p, \omega_K} + h_K 
				\xi_{\mathcal{T}}^{-1} N^{\frac{3}{2}  + \epsilon} |  \bdd{v} 
				|_{1, p, 
					\omega_K} \right) & \text{if } p \in \{1,\infty\},
			\end{cases} \\
			\label{eq:fortin-nonprojection-w1p-local}
			| \mathcal{L}_{F, N} \bdd{v} |_{1, p, K} &\leq \begin{cases}
				C(p)  \xi_{\mathcal{T}}^{-1} N^{3\left|\frac{1}{2} - 
					\frac{1}{p}\right|} | \bdd{v} |_{1, p, \omega_K} & \text{if } 
					p 
				\in (1,\infty), \\
				C(\epsilon)  \xi_{\mathcal{T}}^{-1} N^{\frac{3}{2} + \epsilon} | 
				\bdd{v} 
				|_{1, p, \omega_K} & \text{if } p \in \{1,\infty\},
			\end{cases}
		\end{align}
	\end{subequations}
	where $C(p)$ and $C(\epsilon)$ also depend on $\Omega$, $\Gamma_0$, and 
	shape regularity.
\end{theorem}
Note that $\mathcal{L}_{F, N}$ is only a projection on $\bdd{V}^4$, as opposed 
to all of $\bdd{V}^{N}$. Some of the existing theory (e.g. \cite[Remark 
2.13]{Belenki12}) only 
assumes that the Fortin operator is a projection on $\mathcal{P}_1(\Omega)$, 
so $\mathcal{L}_{N, F}$ is suitable to apply these results. Here, 
\ref{cond:mesh-no-sing} is crucial to ensure that $\|\mathcal{L}_{F, N} 
\bdd{v}\|_{1, p, K}$ depends only on $\bdd{v}$ in one patch of elements around 
$K$. In contrast, the estimates for the Fortin operator in \cite[Lemma 
A.11]{Tscherpel18} depend on multiple layers of patches, as only 
\ref{cond:mesh-no-odd-sing} is assumed.

By modifying one step in the construction of $\mathcal{L}_{F, N}$, we obtain a 
local Fortin operator that is also a projection operator on all of 
$\bdd{V}^N$, which is the more common assumption on the Fortin operator, 
at the expense of implicit dependence on the polynomial degree in the 
estimates.
\begin{theorem}
	\label{thm:fortin-projection}
	Suppose that the mesh satisfies \ref{cond:mesh-no-sing} and let $N \geq 
	4$. There exists a linear operator $\mathcal{I}_{N, F} : 
	\bdd{W}^{1,1}(\Omega) \to \bdd{V}^N$ satisfying the following for $\bdd{v} 
	\in \bdd{W}^{1,1}(\Omega)$:
	\begin{subequations}
		\begin{alignat}{2}
			\label{eq:fortin-projection-l2-projection}
			\int_{\Omega} (\dive \mathcal{I}_{N, F} \bdd{v}) q \d{\bdd{x}} &= 
			\int_{\Omega} (\dive \bdd{v}) q \d{\bdd{x}} \qquad & &\forall q 
			\in DG^{N-1}(\mathcal{T}), \\
			\mathcal{I}_{N, F} \bdd{v}|_{\Gamma_0} &= \bdd{0} \qquad & & 
			\text{if } \bdd{v} \in \bdd{W}^{1,1}_{\Gamma}(\Omega), \\
			\mathcal{I}_{N, F} \bdd{v} &= \bdd{v} \qquad & &\text{if } \bdd{v} 
			\in \bdd{V}^N.
		\end{alignat}	
	\end{subequations}
	Moreover, for all $p \in [1,\infty]$ and $\bdd{v} \in 
	\bdd{W}^{1,p}(\Omega)$, there holds
	\begin{subequations}
		\begin{alignat}{2}
			\label{eq:fortin-projection-lp-local}
			\| \mathcal{I}_{N, F} \bdd{v} \|_{p, K} &\leq C(N, p) \left( \|  
			\bdd{v} \|_{p, \omega_K} + h_K \xi_{\mathcal{T}}^{-1} |  \bdd{v} 
			|_{1, p, \omega_K} \right) \qquad & &\forall K \in \mathcal{T}, \\
			\label{eq:fortin-projection-w1p-local}
			| \mathcal{I}_{N, F} \bdd{v} |_{1, p, K} &\leq C(N, p)  
			\xi_{\mathcal{T}}^{-1}  | \bdd{v} |_{1, p, \omega_K}  \qquad & 
			&\forall K \in \mathcal{T},
		\end{alignat}
	\end{subequations}
	where $C(N, p)$ also depends on $\Omega$, $\Gamma_0$, and shape 
	regularity. 
\end{theorem}
In the case that $\Gamma_0 = \partial \Omega$, \cref{thm:fortin-projection}, 
we recover \cite[Lemma A.11]{Tscherpel18}, albeit with a different underlying 
construction. In the 
next section, we apply the results for inf-sup 
stability (\cref{thm:inf-sup-stability}) and the existence of local Fortin 
operators to obtain a priori error estimates for the $p$-Stokes problem.

	\section{Application to the $p$-Stokes system}
\label{sec:p-stokes-app}

To demonstrate how our results may be applied to non-Newtonian flows, we 
consider the $p$-Stokes problem with $p \in (1,\infty)$ in weak form: Find 
$(\bdd{u}, q) \in \bdd{W}^{1, p}_{\Gamma}(\Omega) \times 
L^{p'}_{\Gamma}(\Omega)$ such that
\begin{subequations}
	\label{eq:p-stokes-weak}
	\begin{alignat}{2}
		\label{eq:p-stokes-weak-1}
		(\nu \tensorbd{S}(\symgrad{\bdd{u}}), \symgrad{\bdd{v}}) - (\dive 
		\bdd{v}, 
		q) &= L(\bdd{v}) \qquad 
		& &\forall \bdd{v} \in \bdd{W}^{1, p}_{\Gamma}(\Omega), \\
		\label{eq:p-stokes-weak-2}
		-(\dive \bdd{u}, r) &= 0 \qquad & &\forall r \in 
		L^{p'}_{\Gamma}(\Omega).
	\end{alignat}
\end{subequations} 
where $\nu > 0$ is constant and $L$ belongs to the dual space of 
$\bdd{W}^{1,p}_{\Gamma}(\Omega)$. We 
recall that $\symgrad{\bdd{v}}$ is the symmetric part of the gradient of 
$\bdd{v}$ and the shear stress $\tensorbd{S}$ is defined in 
\cref{eq:extra-stress-tensor}. The Scott-Vogelius discretization of 
\cref{eq:p-stokes-weak} then reads: Find 
$(\bdd{u}_N, q_{N-1}) \in \bdd{V}^N_{\Gamma} \times Q^{N-1}_{\Gamma}$, $N \geq 
4$, such that
\begin{subequations}
	\label{eq:p-stokes-fem}
	\begin{alignat}{2}
		\label{eq:p-stokes-fem-1}
		(\nu \tensorbd{S}(\symgrad{\bdd{u}_N}), \symgrad{\bdd{v}}) - (\dive 
		\bdd{v}, q_{N-1}) &= L(\bdd{v})  \qquad & &\forall \bdd{v} \in 
		\bdd{V}^{N}_{\Gamma}, \\
		\label{eq:p-stokes-fem-2}
		-(\dive \bdd{u}_N, r) &= 0 \qquad & &\forall r \in 
		Q^{N-1}_{\Gamma}.
	\end{alignat}
\end{subequations} 
The proof of well-posedness of problems 
\cref{eq:p-stokes-weak,eq:p-stokes-fem} is 
standard (see e.g. \cite[\S 2]{BarrettLiu94}). We note that Korn's inequality 
gives the following stability estimate for $\bdd{u}$ and $\bdd{u}_N$:
\begin{align}
	\label{eq:p-stokes-u-un-bounded}
	\nu \max \left\{ \| \bdd{u} \|_{1, p, \Omega}, \| \bdd{u}_N \|_{1, p, 
		\Omega} \right\}^{p-1} 
	\leq  
	C(\Omega, \Gamma_0) \sup_{ \bdd{v} \in \bdd{W}^{1, p}(\Omega) } \frac{ 
		|L(\bdd{v})| }{ \|\bdd{v}\|_{1, p, \Omega} }.  
\end{align}

\subsection{Pressure-robust error estimates}
\label{sec:p-stokes-error-estimates}

We now proceed similarly as in \cite[\S 4]{Diening25} to obtain pressure 
robust error estimates. To this end, we introduce the auxiliary tensor
\begin{align}
	\label{eq:f-tensor-def}
	\tensorbd{F}(\tensorbd{A}) := \sqrt{|\tensorbd{A}|^{p-2}} \tensorbd{A} 
	\qquad \forall \tensorbd{A} \in \mathbb{R}^{2 \times 2},
\end{align}
which is used to define the ``natural velocity distance" \cite{Belenki12} for 
the $p$-Stokes problem $\| \tensorbd{F}(\symgrad{\bdd{u}}) - 
\tensorbd{F}(\symgrad{\bdd{v}}) \|_{2, \Omega}$. 

\begin{lemma}
	\label{lem:p-stokes-quasi-opt}
	Let $\bdd{u} \in \bdd{W}^{1, p}_{\Gamma}(\Omega)$ and $q \in 
	L_{\Gamma}^{p'}(\Omega)$ denote the solution \cref{eq:p-stokes-weak} and 
	$\bdd{u}_N \in \bdd{V}^N_{\Gamma}$ and $q_{N-1} \in Q^{N-1}_{\Gamma}$ the 
	solution to \cref{eq:p-stokes-fem} with $N \geq 4$. Then, there holds
	\begin{subequations}
		\begin{align}
			\label{eq:p-stokes-u-quasi-opt}
			\| \tensorbd{F}(\symgrad{\bdd{u}}) - 
			\tensorbd{F}(\symgrad{\bdd{u}_N}) 
			\|_{2, \Omega} 
			&\leq C(p) 
			\inf_{ \substack{ \bdd{v} \in \bdd{V}^N_{\Gamma} \\ 
					\dive \bdd{v} = 0 } }  
			\| \tensorbd{F}(\symgrad{\bdd{u}}) - \tensorbd{F}(\tensorbd{D} 
			\bdd{v}) \|_{2, \Omega}
		\end{align}
		and
		\begin{align}
			\label{eq:p-stokes-q-quasi-opt-s}
			\|q - q_{N-1}\|_{p', \Omega} &\leq \frac{C}{\beta_{N, p, \Gamma}} 
			\left( 
			\nu \| 
			\tensorbd{S}(\symgrad{\bdd{u}}) - \tensorbd{S}(\tensorbd{D} 
			\bdd{u}_N)\|_{p', \Omega} + \inf_{r \in Q^{N-1}_{\Gamma}} \|q - 
			r\|_{p', \Omega} \right), 
		\end{align}
	\end{subequations}
	where $\beta_{N, p, \Gamma}$ is defined in \cref{eq:inf-sup-sv-def}. 
	Moreover, there holds
	\begin{multline}
		\label{eq:p-stokes-q-quasi-opt-f}
		\|q - q_{N-1}\|_{p', \Omega} \\
		\leq \frac{C(p, L)}{\beta_{N, p, \Gamma}} \left( \nu \| 
		\tensorbd{F}(\symgrad{\bdd{u}}) - \tensorbd{F}(\tensorbd{D} 
		\bdd{u}_N)\|_{2, \Omega}^{\min \left\{ 1, \frac{2}{p'} \right\}} + 
		\inf_{r \in Q^{N-1}_{\Gamma}} \|q - 
		r\|_{p', \Omega} \right).
	\end{multline}
\end{lemma}
\begin{proof}
	
	\noindent \textbf{Step 1: \cref{eq:p-stokes-u-quasi-opt}. }
	Note that choosing $\bdd{w} \in \bdd{V}^{N}_{\Gamma}$ with $\dive \bdd{w} 
	= 0$ 
	in \cref{eq:p-stokes-weak-1,eq:p-stokes-fem-1} gives
	\begin{align}
		\label{eq:proof:s-diff-oroth-div-free}
		(\tensorbd{S}(\symgrad{\bdd{u}}) - \tensorbd{S}(\tensorbd{D} 
		\bdd{u}_N), \tensorbd{D} \bdd{w} ) = 0.
	\end{align}
	Now let $\bdd{v} \in \bdd{V}^{N}_{\Gamma}$ with $\dive \bdd{v} = 
	0$. Applying \cite[Lemma 2.5]{Belenki12}, we obtain
	\begin{align*}
		\alpha(p) \| \tensorbd{F}(\symgrad{\bdd{u}}) - 
		\tensorbd{F}(\symgrad{\bdd{u}_N})  \|_{2, \Omega}^2 
		&\leq (\tensorbd{S}(\symgrad{\bdd{u}}) - \tensorbd{S}(\tensorbd{D} 
		\bdd{u}_N), \symgrad{\bdd{u}} - \symgrad{\bdd{u}_N}) \\
		&= (\tensorbd{S}(\symgrad{\bdd{u}}) - \tensorbd{S}(\tensorbd{D} 
		\bdd{u}_N), \symgrad{\bdd{u}} - \tensorbd{D} \bdd{v}),
	\end{align*} 
	where we used \cref{eq:proof:s-diff-oroth-div-free} and that $\dive 
	\bdd{u}_N = 0$ by \cref{eq:p-stokes-fem-2}. Thanks to 
	\cite[Lemma 2.7]{Belenki12}, there exists $C(p)$ such that
	\begin{multline*}
		|(\tensorbd{S}(\symgrad{\bdd{u}}) - \tensorbd{S}(\tensorbd{D} 
		\bdd{u}_N), \symgrad{\bdd{u}} - \tensorbd{D} \bdd{v})| \\
		\leq 
		\frac{\alpha(p)}{2}  \| 
		\tensorbd{F}(\symgrad{\bdd{u}}) - \tensorbd{F}(\tensorbd{D} 
		\bdd{u}_N)\|_{2, \Omega}^2 + C(p) \| 
		\tensorbd{F}(\symgrad{\bdd{u}}) - \tensorbd{F}(\tensorbd{D} 
		\bdd{v})\|_{2, \Omega}^2,
	\end{multline*}
	and inequality \cref{eq:p-stokes-u-quasi-opt} now follows from taking the 
	infimum over all such $\bdd{v}$.
	
	\noindent \textbf{Step 2: \cref{eq:p-stokes-q-quasi-opt-s}. } Let $r \in 
	Q^{N-1}_{\Gamma}$. The inf-sup condition \cref{eq:inf-sup-sv-def} and 
	\cref{eq:p-stokes-weak-1,eq:p-stokes-fem-1} give
	\begin{align*}
		\beta_{N, p, \Gamma} \|q_{N-1} - r\|_{p', \Omega} 
		&\leq \sup_{\bdd{v} \in \bdd{V}^N_{\Gamma}} 
		\frac{ -(\dive \bdd{v}, q_{N-1} - r) }{\|\bdd{v}\|_{1, p}} \\
		&= \sup_{\bdd{v} \in \bdd{V}^N_{\Gamma}} 
		\frac{ \nu ( \tensorbd{S}(\symgrad{\bdd{u}}) - 
			\tensorbd{S}(\tensorbd{D} 
			\bdd{u}_N), \tensorbd{D} \bdd{v} ) - (\dive \bdd{v}, q - r) 
		}{\|\bdd{v}\|_{1, p}} \\
		&\leq C \left(  \nu \| 
		\tensorbd{S}(\symgrad{\bdd{u}}) - \tensorbd{S}(\tensorbd{D} 
		\bdd{u}_N)\|_{p', \Omega} + \|q - r\|_{p', \Omega} \right),
	\end{align*}
	and inequality \cref{eq:p-stokes-q-quasi-opt-s} follows from the triangle 
	inequality and taking the infimum over all $r \in Q^{N-1}_{\Gamma}$.
	
	\noindent \textbf{Step 3: \cref{eq:p-stokes-q-quasi-opt-f}. } Lemma 4.6 
	from \cite{Belenki12} shows that if $p \in (2,\infty)$, then
	\begin{align*}
		\| \tensorbd{S}(\symgrad{\bdd{u}}) - \tensorbd{S}(\tensorbd{D} 
		\bdd{u}_N)\|_{p', \Omega}
		\leq C(p) \| |\symgrad{\bdd{u}}| + |\symgrad{\bdd{u}_N}| 
		\|_{p, \Omega}^{\frac{p}{2}-1}  \| \tensorbd{F}(\symgrad{\bdd{u}}) 
		- 
		\tensorbd{F}(\symgrad{\bdd{u}_N})\|_{2, \Omega},
	\end{align*}
	while for $p \in (1, 2)$, there holds
	\begin{align}
		\label{eq:proof:p-stokes-extra-stress-by-natural-dist}
		\| \tensorbd{S}(\symgrad{\bdd{u}}) - \tensorbd{S}(\tensorbd{D} 
		\bdd{u}_N)\|_{p', \Omega}
		\leq C(p) \| \tensorbd{F}(\symgrad{\bdd{u}}) - 
		\tensorbd{F}(\symgrad{\bdd{u}_N})\|_{2, \Omega}^{\frac{2}{p'}},
	\end{align}
	and so inequality \cref{eq:p-stokes-q-quasi-opt-f} follows from 
	\cref{eq:p-stokes-u-un-bounded}.		
\end{proof}

\Cref{lem:p-stokes-quasi-opt} shows that when using the ``natural velocity 
distance," one obtains a quasi-optimal best approximation of the velocity 
\cref{eq:p-stokes-u-quasi-opt}, which reduces to the pressure-robust error 
estimates of \cite{JohnLinkeMerdonNeilReb17} in the linear case ($p=2$). Note 
that the pressure error estimate \cref{eq:p-stokes-q-quasi-opt-s} also 
involves the 
stress error, while only the estimates in the natural velocity distance 
are 
quasi-optimal. In the linear case, $\tensorbd{F}(\symgrad{\bdd{u}})$ and 
$\tensorbd{S}(\symgrad{\bdd{u}})$ are 
identical, and so one obtains optimal rates of convergence of the pressure. In 
the nonlinear case, $\tensorbd{F}(\symgrad{\bdd{u}}) \neq 
\tensorbd{S}(\symgrad{\bdd{u}})$, and so bounding the 
additional stress error by the natural velocity distance results in the 
$\min\{1, 
2/p'\}$ exponent in \cref{eq:p-stokes-q-quasi-opt-f}. 

To obtain rates of convergence, we apply the approximation results from 
\cite{Belenki12}. In particular, when the mesh contains no singular vertices 
($\mathcal{V}_S = \emptyset$), the Fortin operators in 
\cref{thm:fortin-nonprojection,thm:fortin-projection} satisfy 
\cite[Assumptions 2.9 \& 2.11]{Belenki12}, and so applying \cite[Theorem 
5.1]{Belenki12} and standard approximation results show that the method 
\cref{eq:p-stokes-fem} converges linearly under suitable regularity.	
\begin{corollary}
	\label{cor:p-stokes-linear-convergence}
	Suppose that $\mathcal{V}_S = \emptyset$ and let $\bdd{u}$, $q$, 
	$\bdd{u}_N$, and $q_{N-1}$ be as in \cref{lem:p-stokes-quasi-opt}.
	If $\tensorbd{F}(\symgrad{\bdd{u}}) \in \tensorbd{W}^{1,2}(\Omega)$, then 
	there holds
	\begin{align}
		\label{eq:p-stokes-linear-natural-norm-conv}
		\| \tensorbd{F}(\symgrad{\bdd{u}}) - \tensorbd{F}(\symgrad{\bdd{u}_N}) 
		\|_{2, \Omega} &\leq C(p) h \|\grad 
		\tensorbd{F}(\symgrad{\bdd{u}})\|_{2, \Omega}.
	\end{align}
	If moreover, $q \in W^{1, p'}(\Omega)$, then there holds
	\begin{align}
		\label{eq:p-stokes-linear-pressure-conv}
		\|q - q_{N-1}\|_{p',\Omega} \leq  \frac{C(p, L) 
		}{\beta_{N, p, \Gamma}}  \left( \left(h \|\grad 
		\tensorbd{F}(\symgrad{\bdd{u}})\|_{2, 
			\Omega}\right)^{\min \left\{ 1, \frac{2}{p'} \right\}} + h
		\|\grad q\|_{p', \Omega} \right).
	\end{align}
\end{corollary}
\noindent In general, obtaining higher rates of convergence in the natural 
distance norm assuming sufficient regularity remains an open problem. 

We also note in passing that $\bdd{W}^{1,p}(\Omega)$ error estimates for the 
velocity can also be obtained using monotone operator theory; e.g. 
\cite[Theorems 3 \& 4]{Chow89} gives 	
\begin{align}
	\label{eq:p-stokes-u-w1p}
	\|\bdd{u} - \bdd{u}_{N}\|_{1, p, \Omega} \leq \begin{cases}
		C(p) \inf\limits_{ \substack{ \bdd{v} \in \bdd{V}^N_{\Gamma} \\ \dive 
				\bdd{v} \equiv 0 }} \|\bdd{u} - \bdd{v}\|_{1, 
			p, \Omega}^{p/2} & \text{if } p \in (1,2], \\
		C(p) \inf\limits_{ \substack{ \bdd{v} \in \bdd{V}^N_{\Gamma} 
				\\ \dive \bdd{v} \equiv 0 }} C(\|\bdd{v}\|_{1,p}) \|\bdd{u} - 
		\bdd{v}\|_{1, p, \Omega}^{2/p} & \text{if } p \in (2,\infty).
	\end{cases}
\end{align}
However, these estimates are known to be suboptimal in some cases 
\cite[Corollary 2]{Chow89}.

\begin{remark}
	The proof of \cref{lem:p-stokes-quasi-opt} only required that 
	$\bdd{V}^N_{\Gamma}$ is a conforming subspace of $\bdd{W}^{1, 
		p}_{\Gamma}(\Omega)$ and $Q^{N-1}_{\Gamma} = \dive \bdd{V}^N_{\Gamma}$ 
	(agnostic to the dimension of the problem), while the proof of 
	\cref{cor:p-stokes-linear-convergence} only required a local Fortin 
	operator satisfying \cite[Assumptions 2.9 \& 2.11]{Belenki12}.
\end{remark}

\subsection{Numerical examples}
\label{sec:p-stokes-numerics}

We now turn to two numerical examples to demonstrate the a priori error 
estimates in 
\cref{lem:p-stokes-quasi-opt}. For each example, $\Omega$ is the unit square 
$\Omega 
= (0, 1)^2$ with initial mesh $\mathcal{T}_0$ depicted in \cref{fig:init-mesh} 
and 
Dirichlet boundary conditions are imposed on $\Gamma_0 = \{x = 0\}$. 
We prescribe an exact solution $\bdd{u}$ and $q$ and choose $L$, and the 
boundary conditions accordingly. All computations are performed using 
Firedrake 
\cite{FiredrakeUserManual}. To solve the discrete nonlinear problem 
\cref{eq:p-stokes-fem}, we use the default PETSc 
\cite{PETScUserManual,PETSc11} 
nonlinear solver and 
tolerances: a Newton linesearch method which terminates when the $\ell_2$ norm 
of the 
residual reduces by a factor of $10^{-8}$ or when successive iterates differ 
by less 
than $10^{-8}$ in the $\ell_2$ norm. To initialize the Newton method, we solve 
the 
following Stokes system: Find $(\bdd{u}_{N}, q_{N-1}) \in \bdd{V}^N_{\Gamma} 
\times 
Q^{N-1}_{\Gamma}$ such that
\begin{alignat*}{2}
	(\symgrad{\bdd{u}_N}, \tensorbd{D} \bdd{v}) - (\dive \bdd{v}, q_{N-1}) &= 
	(\symgrad{\bdd{u}}, \tensorbd{D} \bdd{v}) - (\dive \bdd{v}, q) \qquad & 
	&\forall \bdd{v} \in \bdd{V}^N_{\Gamma}, \\
	-(\dive \bdd{u}_N, r) &= 0, \qquad & &\forall r \in 
	Q^{N-1}_{\Gamma}.
\end{alignat*}
We invert all of the resulting linear systems with the default LU solver in 
PETSc. We consider values of $p$ between 1.333 and 4, as the nonlinear solver 
diverges outside of this range. More sophisticated nonlinear solvers tailored 
to this particular nonlinearity are available \cite{Balci23,Heid22}.

To refine the discretization, we either refine the mesh by subdividing every 
element 
into four congruent triangles or increase the polynomial degree. 
Note that on $\mathcal{T}_0$ and subsequent refinements, there are no singular 
vertices and so $Q^{N-1}_{\Gamma}$ coincides with the usual discontinuous 
piecewise polynomial space $DG^{N-1}(\mathcal{T})$ \cref{eq:dg-definition}.	
The four error 
metrics we consider are 
(i) the $W^{1, p}$ velocity error $\|\bdd{u} - \bdd{u}_N\|_{1, p}$, 
(ii) the $L^{p'}$ stress error $\| 
\tensorbd{S}(\symgrad{\bdd{u}}) - \tensorbd{S}(\symgrad{\bdd{u}_N}) \|_{p'}$, 
(iii) the natural velocity distance error 
$\|\tensorbd{F}(\symgrad{\bdd{u}}) - \tensorbd{F}( \symgrad{\bdd{u}_N} )\|_2$,
and (iv) the $L^{p'}$ pressure error $\|q - q_{N-1}\|_{p'}$.

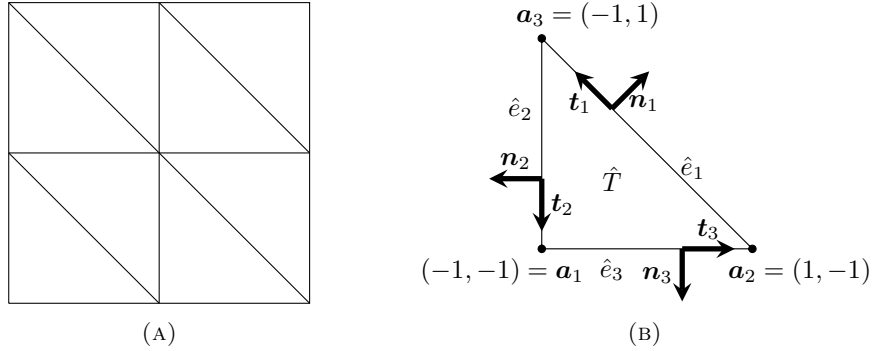
\begin{figure}[htb]
	\centering	
	\begin{subfigure}{0.49\linewidth}
		\centering
		\begin{tikzpicture}[scale=4]
			\coordinate (a1) at (0,0);
			\coordinate (a2) at (1,0);
			\coordinate (a3) at (1,1);
			\coordinate (a4) at (0,1);
			
			\coordinate (b1) at (0.5, 0);
			\coordinate (b2) at (0.5, 0.5);
			\coordinate (b3) at (1, 0.5);
			\coordinate (b4) at (0.5, 1);
			\coordinate (b5) at (0, 0.5);
			
			\draw (a1) -- (a2) -- (a3) -- (a4) -- (a1);
			\draw (b1) -- (b4);
			\draw (b5) -- (b3);
			\draw (a2) -- (a4);
			\draw (b1) -- (b5);
			\draw (b3) -- (b4);				
		\end{tikzpicture}
		\caption{}
		\label{fig:init-mesh}
	\end{subfigure}
	\hfill
	\begin{subfigure}{0.49\linewidth}
		\centering
		\begin{tikzpicture}[scale=0.7]
			\filldraw (0,0) circle (2pt) 
			node[align=center,below]{}
			-- (4,0) circle (2pt)
			node[align=center,below]{}	
			-- (0,4) circle (2pt) 
			node[align=center,above]{}
			-- (0,0);
			
			\draw (-0.75, 0) node[align=center,below]{$(-1,-1) = 
				\vertex{a}_1$};
			\draw (4.9, 0) node[align=center,below]{$\vertex{a}_2 = (1, -1)$};
			\draw (0.875, 4) node[align=center,above]{$\vertex{a}_3 = (-1, 
				1)$};
			
			\coordinate (a1) at (0,0);
			\coordinate (a2) at (4,0);
			\coordinate (a3) at (0,4);
			
			\coordinate (e113) at ($(a2)!1/3!(a3)$);
			\coordinate (e123) at ($(a2)!2/3!(a3)$);
			\coordinate (e1231) at ($(a2)!2/3+1/sqrt(32)!(a3)$);
			
			\coordinate (e213) at ($(a3)!1/3!(a1)$);
			\coordinate (e223) at ($(a3)!2/3!(a1)$);
			\coordinate (e2231) at ($(a3)!2/3+1/4!(a1)$);
			
			\coordinate (e313) at ($(a1)!1/3!(a2)$);
			\coordinate (e323) at ($(a1)!2/3!(a2)$);
			\coordinate (e3231) at ($(a1)!2/3+1/4!(a2)$);
			
			\draw ($(e113)+(0.2,0.2)$) node[align=center]{$\hat{e}_1$};
			\draw ($(e213)+(0,0)$) node[align=center,left]{$\hat{e}_2$};
			\draw ($(e313)+(0,0)$) node[align=center,below]{$\hat{e}_3$};
			
			\draw[line width=2, -stealth] (e123) -- (e1231);
			\draw ($($(e123)!0.5!(e1231)$)+(-0.25,-0.25)$)
			node[align=center]{$\unitvec{t}_1$};
			\draw[line width=2, -stealth] (e123) -- ($(e123)!1!-90:(e1231)$);
			\draw ($($(e123)!0.5!-90:(e1231)$)+(0.25,-0.25)$)
			node[align=center]{$\unitvec{n}_1$};
			\draw[line width=2, -stealth] (e223) -- (e2231);
			\draw ($($(e223)!0.5!(e2231)$)+(0,0)$) 
			node[align=center, right]{$\unitvec{t}_2$};
			\draw[line width=2, -stealth] (e223) -- ($(e223)!1!-90:(e2231)$);
			\draw ($($(e223)!0.5!-90:(e2231)$)+(0,0)$) 
			node[align=center, above]{$\unitvec{n}_2$};
			\draw[line width=2, -stealth] (e323) -- (e3231);
			\draw ($($(e323)!0.5!(e3231)$)+(0,0)$) 
			node[align=center, above]{$\unitvec{t}_3$};
			\draw[line width=2, -stealth] (e323) -- ($(e323)!1!-90:(e3231)$);
			\draw ($($(e323)!0.5!-90:(e3231)$)+(0,0)$) 
			node[align=center, left]{$\unitvec{n}_3$};

			\draw (4/3, 4/3) node(T){$\reftri$};
		\end{tikzpicture}		
		\caption{}
		\label{fig:reference triangle}
	\end{subfigure}
	\caption{Notation for (A) initial mesh $\mathcal{T}_0$ and (B) reference 
		triangle 
		$\reftri$.}
\end{figure}

\subsubsection{Smooth solution}

Consider a smooth solution
\begin{align}
	\label{eq:smooth-exact}
	\bdd{u} = \vcurl ((\mathrm{e}^x -1)^2 \mathrm{e}^y) = (\mathrm{e}^{x} - 1) 
	\mathrm{e}^y \begin{bmatrix}
		-(\mathrm{e}^x - 1) \\
		2\mathrm{e}^x
	\end{bmatrix}  \quad \text{and} \quad q = \mathrm{e}^{x+2y}.
\end{align}
In particular, $\tensorbd{S}(\symgrad{\bdd{u}})$ and 
$\tensorbd{F}(\symgrad{\bdd{u}})$ are smooth in $\Omega$. In 
\cref{fig:smooth-errors}, we display plots of the four error 
metrics and power-law indices $p \in \{1.333, 1.5, 2, 3, 4\}$ for the 
$h$-version method with $N \in \{ 4, 5, 6 \}$ and $p$-version 
method\footnote{Recall that the ``$p$" in ``$p$-version" refers to the 
	polynomial degree and not the index in the power-law.} on 
$\mathcal{T}_0$, where the polynomial degree is varied from 4 to 10. 
\Cref{tab:plot-legends} contains the legend for \cref{fig:smooth-errors}. We 
observe that for the $h$-version method, all of the error metrics converge at 
$\mathcal{O}(h^N)$, while for the $p$-version method, all converge at 
$\mathcal{O}(\exp(-N))$, irrespective of the chosen values of $p$. 
Additional theory is required to explain these rates of convergence in the 
nonlinear setting $p \neq 2$.

\begin{figure}[htb]
	\centering
	\begin{subfigure}[t]{0.48\linewidth}
		\centering
		\includegraphics[width=\linewidth]%
		{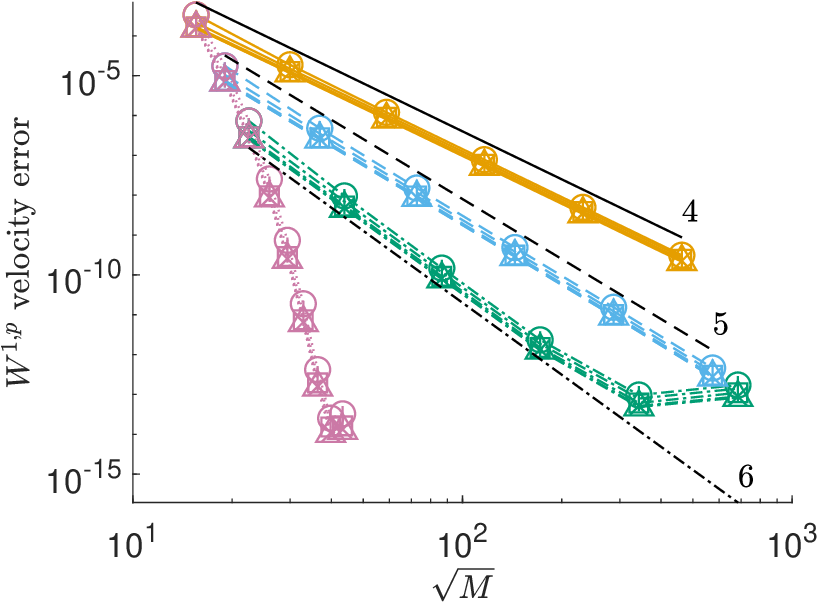}
		\caption{$\|\bdd{u} - \bdd{u}_N\|_{1, p, \Omega}$}
	\end{subfigure}
	\hfill
	\begin{subfigure}[t]{0.48\linewidth}
		\centering
		\includegraphics[width=\linewidth]%
		{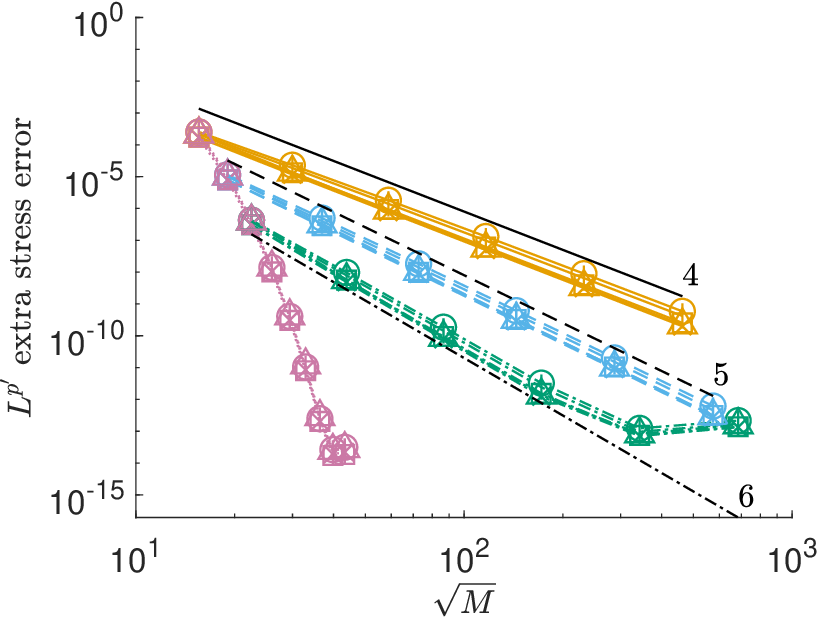}
		\caption{$\| \tensorbd{S}(\symgrad{\bdd{u}}) - 
			\tensorbd{S}(\symgrad{\bdd{u}_N}) \|_{p', \Omega}$}
	\end{subfigure}
	\\
	\begin{subfigure}[t]{0.48\linewidth}
		\centering
		\includegraphics[width=\linewidth]%
		{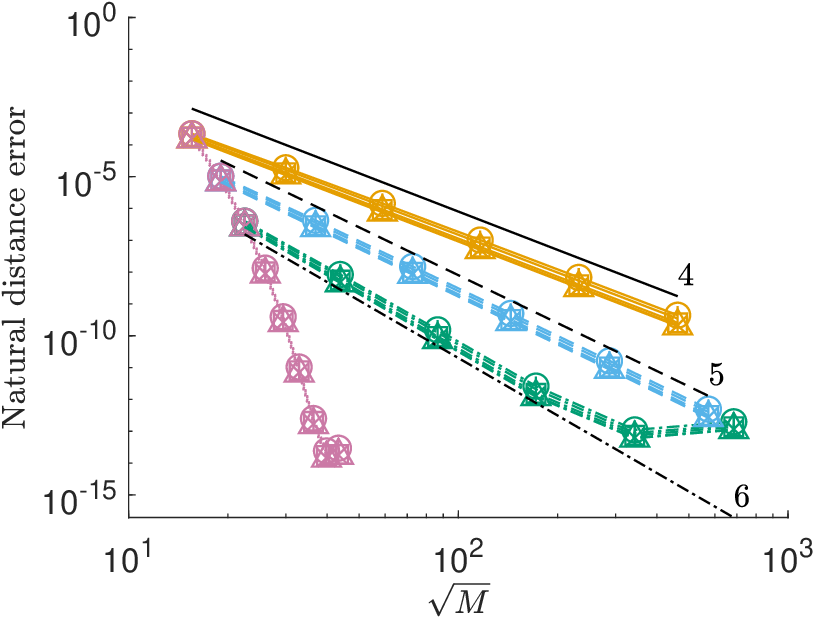}
		\caption{$\|\tensorbd{F}(\symgrad{\bdd{u}}) - 
			\tensorbd{F}( \symgrad{\bdd{u}_N} )\|_{2, \Omega}$}
	\end{subfigure}
	\hfill
	\begin{subfigure}[t]{0.48\linewidth}
		\centering
		\includegraphics[width=\linewidth]%
		{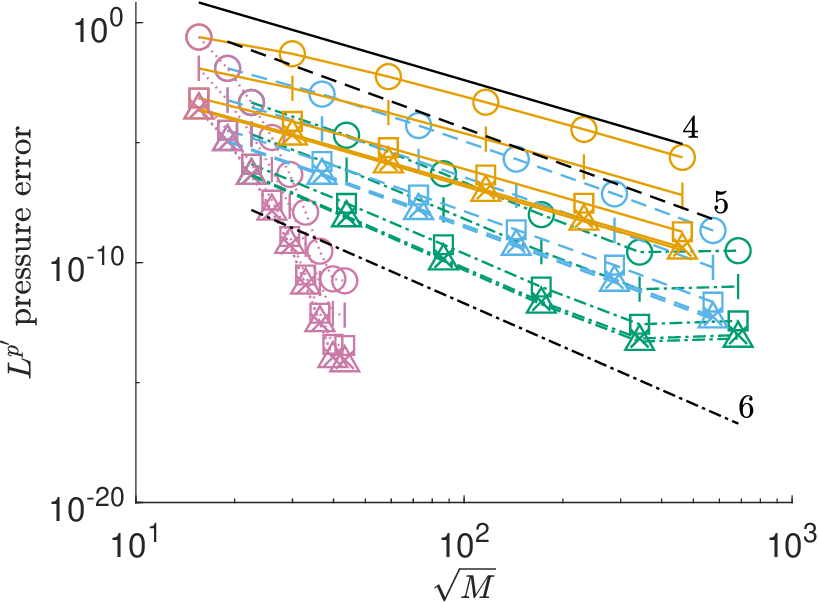}
		\caption{ $\|q - q_{N-1}\|_{p', \Omega}$}
	\end{subfigure}
	\caption{Relative errors for example \cref{eq:smooth-exact} against 
		the total number of degrees of freedom $M = \dim \bdd{V}^N + \dim 
		Q^{N-1}$; 
		see \cref{tab:plot-legends} for the legend for the lines with markers. 
		The text demarcates the slope
		of the marker-less black reference lines, 
		indicating that each of the errors of the 
		$h$-version methods decays like $\mathcal{O}(h^N)$. Note that the 
		error plateau for $N=6$ is well below the nonlinear solver tolerance.}
	\label{fig:smooth-errors}
\end{figure}

\begin{table}[htb]
	\centering
	\begin{subtable}[t]{0.48\linewidth}
		\centering
		\begin{tabular}{c|c}
			Method & Line style \\
			\hline
			$h$-version, $N=4$ & \textcolor{fourth}{solid} \\
			$h$-version, $N=5$ & \textcolor{fifth}{dashed} \\
			$h$-version, $N=6$ & \textcolor{sixth}{dash-dotted} \\
			$p$-version, $\mathcal{T}_0$ & \textcolor{pver}{dotted}
		\end{tabular}
		\caption{}
	\end{subtable}
	\hfill 	
	\begin{subtable}[t]{0.48\linewidth}
		\centering
		\begin{tabular}{c|c}
			$p$ & Marker \\
			\hline
			1.333 & $\triangle$ \\
			1.5 & $\times$ \\
			2 & $\square$ \\
			3 & $|$ \\
			4 & $\circ$ 
		\end{tabular}
		\caption{}	
	\end{subtable}
	\caption{Legend for the plots in \cref{fig:smooth-errors}.}
	\label{tab:plot-legends}
\end{table}

\subsubsection{Solution with limited regularity}

We now consider the example from \cite[eq. (6.2)]{Diening25} (see also 
\cite[eq. (7.1)]{Belenki12}) with limited regularity:
\begin{align}
	\label{eq:nonsmooth-exact-sol}
	\bdd{u} = |\bdd{x}|^{0.01} \begin{bmatrix}
		x_2 \\ -x_1
	\end{bmatrix} \quad \text{and} \quad q = 
	|\bdd{x}|^{\frac{2}{p} - 1 + 0.01} \cdot \begin{cases}
		1 & \text{if } p \geq 2, \\
		0.01 & \text{otherwise}.
	\end{cases} 
\end{align}
In particular, the exponents are chosen so that
\begin{alignat*}{2}
	\bdd{u} &\in \bdd{W}^{1.01 + \frac{2}{p}-\epsilon,p}(\Omega), \qquad &  
	\tensorbd{F}(\symgrad{\bdd{u}}) &\in 
	\tensorbd{W}^{1 + \frac{(0.01)p}{2}-\epsilon,2}(\Omega), \\
	\tensorbd{S}(\symgrad{\bdd{u}}) &\in \tensorbd{W}^{ \frac{2}{p'} + 
		\frac{0.01}{p'-1} -\epsilon, p'}(\Omega), \qquad & q &\in 
	W^{1.01-\epsilon, p'}(\Omega),
\end{alignat*}
for any $\epsilon > 0$. As we consider $p$ between 1.333 and 4, we expect the 
convergence rates for the natural distance and the pressure to behave as if 
$\bdd{u} \in \bdd{W}^{1+\frac{2}{p}, p}(\Omega)$, 
$\tensorbd{S}(\symgrad{\bdd{u}})\in \tensorbd{W}^{ \frac{2}{p'}, p'}$,
$\tensorbd{F}(\symgrad{\bdd{u}}) \in \tensorbd{W}^{1,2}(\Omega)$, and $q \in 
W^{1,p'}(\Omega)$. Although the estimates in 
\cref{lem:p-stokes-quasi-opt} are not directly applicable to the case of 
nonzero boundary conditions, we 
will see that these estimates appear to accurately describe the behavior of 
the method in some cases. 
In \cref{tab:rough-convergence-rates}, we display the convergence 
rates of the four error metrics for two 
methods: the $h$-version method with $N=4$ (the first column) and the 
$p$-version method on the initial mesh with $N \in \{4,\ldots, 16\}$ (the 
second column). Here, the convergence rate is fitted to be the constant 
$\gamma$ such that the error measure decays like $\mathcal{O}(M^{-\gamma/2})$,
where $M = \dim \bdd{V}^N + \dim Q^{N-1}$ is the total number of degrees of 
freedom. We note that the $h$-version method with $N \in \{5,6\}$ produced the 
same convergence rates as for $N=4$, as one would expect for a solution with
limited regularity.

\begin{table}[htb]
	\centering
	\resizebox{\linewidth}{!}{%
		\begin{tabular}{c|cc|cc|cc|cc}
			$p$ & \multicolumn{2}{|c|}{$\|\bdd{u} - \bdd{u}_N\|_{1, p, 
					\Omega}$} 
			& \multicolumn{2}{|c|}{$\| \tensorbd{S}(\symgrad{\bdd{u}}) - 
				\tensorbd{S}(\symgrad{\bdd{u}_N}) \|_{p', \Omega}$} 
			& \multicolumn{2}{|c|}{$\|\tensorbd{F}(\symgrad{\bdd{u}}) - 
				\tensorbd{F}( \symgrad{\bdd{u}_N} )\|_{2, \Omega}$}
			& \multicolumn{2}{|c}{$\|q - q_{N-1}\|_{p', \Omega}$} \\
			\hline
			$1.333$ & $1.54$ & $2.49$ & $0.51$ & $1.02$ & $1.03$ & $1.98$ & 
			$0.51$ & $1.16$  \\ 
			$1.4$ & $1.47$ & $2.43$ & $0.59$ & $1.18$ & $1.03$ & $1.98$ & 
			$0.59$ & $1.30$  \\ 
			$1.5$ & $1.37$ & $2.34$ & $0.68$ & $1.39$ & $1.03$ & $1.98$ & 
			$0.68$ & $1.48$ \\ 
			$1.666$ & $1.24$ & $2.21$ & $0.82$ & $1.65$ & $1.03$ & $1.98$ & 
			$0.82$ & $1.71$  \\ 
			$2$ & $1.03$ & $1.98$ & $1.03$ & $1.93$ & $1.03$ & $1.93$ & 
			$1.03$ & $1.94$  \\ 
			$2.5$ & $0.83$ & $1.64$ & $1.24$ & $2.23$ & $1.03$ & $1.99$ & 
			$1.03$ & $1.89$  \\ 
			$3$ & $0.69$ & $1.40$ & $1.38$ & $2.36$ & $1.04$ & $1.99$ & 
			$1.03$ & $1.72$  \\ 
			$3.5$ & $0.59$ & $1.25$ & $1.48$ & $2.43$ & $1.04$ & $1.99$ & 
			$1.03$ & $1.57$ \\ 
			$4$ & $0.52$ & $1.15$ & $1.56$ & $2.48$ & $1.04$ & $1.99$ & 
			$1.03$ & $1.45$  \\ 
		\end{tabular}
	}
	\caption{Algebraic convergence rates for the 
		nonsmooth solution \cref{eq:nonsmooth-exact-sol}. 
		For each error metric, the first 
		column is the rate of convergence of the 
		$h$-version with $N=4$, while the second column is the rate of 
		convergence of the $p$-version method on one refinement of 
		$\mathcal{T}_0$ 
		with $N \in \{4,\ldots,16\}$.}
	\label{tab:rough-convergence-rates}
\end{table}

\textbf{$h$-version. } We observe the convergence rates predicted by 
\cref{cor:p-stokes-linear-convergence}: For the natural distance error, 
$\|\tensorbd{F}(\symgrad{\bdd{u}}) - 
\tensorbd{F}( \symgrad{\bdd{u}_N} )\|_2 = \mathcal{O}(h)$, while for the 
pressure error, $\|q - q_{N-1}\|_{p'} = \mathcal{O}(h^{\min\{ 1, 2/p' \}})$. 
For the velocity $W^{1,p}$ error and the stress error, we appear to 
observe the best approximation rates
\begin{align*}
	\inf_{\bdd{v} \in \bdd{V}^N} \|\bdd{u} - \bdd{v}\|_{1,p,\Omega} 
	&\leq C h^{\frac{2}{p}} \|\bdd{u}\|_{p, 1+\frac{2}{p}, \Omega}, \\
	\inf_{\tensorbd{T} \in \tensorbd{DG}^{N-1}} 
	\|\tensorbd{S}(\symgrad{\bdd{u}}) - \tensorbd{T}\|_{p', \Omega} &\leq C 
	h^{\frac{2}{p'}} \|\tensorbd{S}(\symgrad{\bdd{u}}) \|_{p', \frac{2}{p'}, 
		\Omega}.
\end{align*}
Consequently, inequality \cref{eq:p-stokes-u-w1p} does not capture the 
behavior for any $p \neq 2$, while 
\cref{eq:proof:p-stokes-extra-stress-by-natural-dist} only captures the 
behavior of the stress error for $p < 2$.

\textbf{$p$-version. } For the four error metrics and values of $p$ 
considered, the $p$-version method converges at a faster rate than the 
$h$-version method. For the natural distance error
$\|\tensorbd{F}(\symgrad{\bdd{u}}) - 
\tensorbd{F}( \symgrad{\bdd{u}_N} )\|_2$, we observe quadratic convergence for 
all values of $p$ considered, which is double the rate of the $h$-version 
method. We also observe a doubling in the rate of convergence in the
stress error $\| \tensorbd{S}(\symgrad{\bdd{u}}) - 
\tensorbd{S}(\symgrad{\bdd{u}_N}) \|_{p', \Omega}$ and the pressure error $\|q 
- q_{N-1}\|_{p'}$ for $p < 2$ and in the velocity error $\|\bdd{u}_{p} 
- \bdd{u}_{p, N}\|_{1,p,\Omega}$ for $p > 2$. The literature for $p$-version 
approximation theory for functions of the form \cref{eq:nonsmooth-exact-sol} 
is not complete. One of the only available results is \cite[Theorem 
12]{Ain99} (see \cite{AinKay00} for $hp$-approximation theory): for $p > 2$ 
and functions of the form $u = |\bdd{x}|^{\lambda} 
\Theta(\bdd{x})$, where $\Theta$ is smooth, $\lambda > 1-2/p$, there holds
\begin{align*}
	\inf_{v \in V^N_{\Gamma}} \|u - v\|_{1, p} \leq C(\epsilon, u) 
	p^{-2\left(\lambda-1+\frac{2}{p}\right)+\epsilon},
\end{align*}
where $V^{N}_{\Gamma}$ is the space of continuous piecewise polynomials of 
degree $N$ vanishing on $\Gamma_0$. Consequently, the convergence rates we 
observe for $\|\bdd{u}_{p} - \bdd{u}_{p, N}\|_{1,p,\Omega}$ when $p > 2$ 
correspond to the optimal approximation rates. There is currently no theory to 
confirm whether or not the remaining convergence rates we observe are optimal.

\begin{remark}
	For $p < 2$, the pressure error and the stress error have about the 
	same convergence rate for the $p$-version. If the error estimate 
	\cref{eq:p-stokes-q-quasi-opt-s} and inf-sup estimate 
	\cref{eq:inf-sup-stability} were sharp, we would expect that 
	the pressure convergence rate is $3\left| \frac{1}{2} - 
	\frac{1}{p}\right|$ less than the stress convergence rate. Since 
	this is not what we observe, our estimate for the inf-sup constant 
	$\beta_{N, p, \Gamma}$ may be suboptimal for $p \in [1.333, 2)$.
\end{remark}

\section{Inverting the divergence operator}	
\label{sec:invert-div}

In this section, we will construct a right-inverse of $\dive : 
\bdd{V}_{\Gamma}^{N} \to Q_{\Gamma}^{N-1}$ whose stability is uniform in the
polynomial degree (\cref{thm:invert-div}). The construction will proceed in 
several steps in the proceeding sections. We first recall the following 
inverse inequality from \cite{Daugavet72} (see also \cite[eq. 
(5.3)]{Ganzburg01}).
\begin{lemma}
	Let $U \subset \mathbb{R}^d$ be open and bounded with Lipschitz boundary 
	and $N \in \mathbb{N}$. For all $1 \leq r \leq p \leq \infty$, there 
	holds
	\begin{align}
		\label{eq:lp-inverse-tri}
		\| q \|_{p, U} \leq C(U, r, p) N^{2d\left( \frac{1}{r} - 
			\frac{1}{p} 
			\right)} \|q\|_{r, U} \qquad \forall q \in \mathcal{P}_{N}(U).
	\end{align}
\end{lemma}

The construction of the right-inverse of the divergence operator in 
\cref{thm:invert-div} follows most of the construction given in 
\cite{AinCP19StokesI} for the case $p=2$. However, at each step, the 
construction requires modification for the case $p \neq 2$.

\subsection{Single element}

We now prove \cref{thm:invert-div} on a single reference element and then a 
physical element $K \in \mathcal{T}$, both with homogeneous Dirichlet 
conditions. Let 
$\reftri \subset \mathbb{R}^2$ be the reference element depicted in 
\cref{fig:reference triangle}. Given $K \in \mathcal{T} \cup \{ \reftri \}$, 
let $\mathcal{V}_K$ and $\mathcal{E}_K$ denote the vertices and edges of $K$, 
and for $N \in \mathbb{N}_0$, define
\begin{align*}
	\bdd{V}^N(K) &:= \mathcal{P}_N(K)^2, \\
	\bdd{V}^N_0(K) &:= \mathcal{P}_N(K)^2 \cap \bdd{W}^{1,1}_0(K), \\
	Q^N_0(K) &:= \{ q \in \mathcal{P}_N(K) \cap L^1_0(K) : q(\vertex{a}) = 0 \ 
	\forall \vertex{a} \in \mathcal{V}_{K} \}.
\end{align*}
Additionally, for $K \in \mathcal{T}$, let $\bdd{F}_K : \reftri \to K$ denote 
a bijective affine mapping.

We begin with an inverse of the trace operator on the reference element that 
provides a divergence-free extension of suitable boundary data. 
\begin{lemma}
	There exists a linear operator
	\begin{align*}
		\hat{\mathcal{L}}_{\partial} : \bigcup_{1 < p < \infty} \{ \bdd{w} \in 
		\bdd{W}^{1,p}(\reftri) : \dive \bdd{w} \in L^p_0(\reftri) \} \to 
		\bdd{W}^{1,1}(\reftri)	
	\end{align*}
	satisfying the following: For all $p \in (1,\infty)$ and $\bdd{v} \in \{ 
	\bdd{w} \in \bdd{W}^{1,p}(\reftri) : \dive \bdd{w} \in L^p_0(\reftri) \}$, 
	there holds $\hat{\mathcal{L}}_{\partial} \bdd{v} \in \bdd{W}^{1, 
		p}(\reftri)$ with 
	\begin{align}
		\label{eq:div-free-extension-cont}
		\hat{\mathcal{L}}_{\partial} \bdd{v} = \bdd{v} \quad \text{on } 
		\partial \reftri, \quad \dive  \hat{\mathcal{L}}_{\partial} \bdd{v} 
		\equiv 0, \quad \text{and} \quad \| \hat{\mathcal{L}}_{\partial} 
		\bdd{v} \|_{1, p, \reftri} \leq C(p) \|\bdd{v}\|_{1-\frac{1}{p}, p, 
			\partial \reftri}.
	\end{align}
	Moreover, for all $N \in \mathbb{N}$, there holds
	\begin{align}
		\label{eq:div-free-extension-poly}
		\hat{\mathcal{L}}_{\partial} : \{ \bdd{w} \in \bdd{V}^N(\reftri) : 
		\dive \bdd{w} \in Q_0^{N-1}(\reftri) \} \to \bdd{V}^{N}(\reftri).
	\end{align}
\end{lemma}
\begin{proof}
	Let $p \in (1,\infty)$ and $\bdd{v} \in \{ \bdd{w} \in 
	\bdd{W}^{1,p}(\reftri) : \dive \bdd{w} \in L^p_0(\reftri) \}$.
	Define $f, g : \partial \reftri \to \mathbb{R}$ by the rules
	\begin{align}
		\label{eq:proof:div-free-f-g}
		f(\bdd{x}) = \int_{\vertex{a}_1}^{\bdd{x}} \bdd{v} \cdot \unitvec{n} 
		\d{s} \quad \text{and} \quad g(\bdd{x}) = -(\bdd{v} \cdot 
		\unitvec{t})(\bdd{x}) \qquad \forall \bdd{x} \in \partial \reftri,
	\end{align}
	where the path integral is oriented in the counter-clockwise direction. 
	Then, these functions satisfy the following conditions:
	\begin{enumerate}
		\item[1.] $f$ is continuous on $\partial \reftri$ since $(\bdd{v} 
		\cdot \unitvec{n}, 1)_{\partial \reftri} = (\dive \bdd{v}, 
		1)_{\reftri} = 0$;
		
		\item[2.] The quantity $\bdd{\sigma}(f, g) := \partial_t f \unitvec{t} 
		+ g \unitvec{n}$, where $\partial_t$ denotes the tangential derivative 
		along $\partial \reftri$, satisfies
		\begin{align*}
			\bdd{\sigma}(f, g) = (\bdd{v} \cdot \unitvec{n}) \unitvec{t} - 
			(\bdd{v} \cdot \unitvec{t}) \unitvec{n} = \bdd{v}^{\perp} \in 
			\bdd{W}^{1-\frac{1}{p}, p}(\partial \reftri).
		\end{align*}
		Recall that $\cdot^{\perp}$ denotes a counterclockwise 
		rotation by $\pi/2$.
	\end{enumerate}
	Thanks to \cite[Theorem 3.2]{Parker23}, there exists $\psi \in W^{2, 
		p}(\reftri)$ satisfying
	\begin{align*}
		\psi|_{\partial \reftri} = f, \quad \partial_n \psi|_{\partial 
			\reftri} = g, \quad \text{and} \quad \|\psi\|_{2, p, \reftri} \leq 
		C(p) \left( \|f\|_{p, \partial \reftri} + \| \bdd{\sigma}(f, g)  
		\|_{1-\frac{1}{p}, p, \partial \reftri} \right).
	\end{align*}
	Moreover, the mapping $(f, g) \mapsto \psi$ is linear and independent of 
	$p$. The function $\bdd{w} := \vcurl \psi$ then satisfies $\dive \bdd{w} = 
	0$,
	\begin{align*}
		\bdd{w} = \vcurl \psi = \partial_t f \unitvec{n} - g \unitvec{t} = 
		\bdd{v} \qquad \text{on } \partial \reftri,
	\end{align*}
	and
	\begin{align*}
		\|\bdd{w}\|_{1, p} &\leq C(p) \left( \|f\|_{p, \partial \reftri} + \| 
		\bdd{\sigma}(f, g) \|_{1-\frac{1}{p}, p, \partial \reftri} \right) \\
		&\leq C(p) \left( \| \bdd{v} \cdot \unitvec{n} \|_{p, \partial 
			\reftri} + \|\bdd{v}^{\perp} \|_{1-\frac{1}{p}, p, \partial \reftri} 
		\right) 
		\leq C(p) \|\bdd{v}\|_{1, p, \reftri},
	\end{align*}
	where we used H\"{o}lder's inequality and the trace theorem. Defining 
	$\hat{\mathcal{L}}_{\partial} \bdd{v} := \bdd{w}$ then satisfies 
	\cref{eq:div-free-extension-cont}.
	
	It remains to show that $\hat{\mathcal{L}}_{\partial}$ satisfies 
	\cref{eq:div-free-extension-poly}. Let $N \in \mathbb{N}$ and $\bdd{v} 
	\in \{ \bdd{w} \in \bdd{V}^N(\reftri) : \dive \bdd{w} \in 
	Q_0^{N-1}(\reftri) \}$. The proof of \cite[Lemma 3.6]{AinCP19StokesI} 
	shows that $f$ and $g$ given by \cref{eq:proof:div-free-f-g} satisfy the 
	following additional conditions:
	\begin{enumerate}
		
		\item[3.] $\bdd{\sigma}(f, g)$ is continuous;
		
		\item[4.] $\unitvec{t}_{i+1} \cdot \bdd{\sigma}(f, 
		g)|_{\hat{e}_i}(\vertex{a}_{i+2}) = \unitvec{t}_i \cdot 
		\bdd{\sigma}(f, g)|_{\hat{e}_{i+1}}(\vertex{a}_{i+2})$, $1 \leq i \leq 
		3$, where the indices are understood modulo 3;
		
		\item[5.] $f|_{\hat{e}_i} \in \mathcal{P}_{N+1}(\hat{e}_i)$ and 
		$g|_{\hat{e}_i} \in \mathcal{P}_{N}(\hat{e}_i)$, $1 \leq i \leq 3$.
	\end{enumerate}
	Then, \cite[Theorem 3.2]{Parker23} shows that the mapping $(f, g) \mapsto 
	\psi$ above also satisfies $\psi \in \mathcal{P}_{N+1}(\reftri)$, and so 
	$\bdd{w} = \vcurl \psi \in \bdd{V}^{N}(\reftri)$, which completes the 
	proof.
\end{proof}

The next step completes the construction an inverse for the divergence 
operator on the reference element.
\begin{lemma}
	\label{lem:invert-div-ref-element}
	There exists a linear operator
	\begin{align*}
		\hat{\mathcal{L}}_{\dive} : \bigcup_{1 < p < \infty} L^p_0(\reftri) 
		\to \bdd{W}^{1, 1}_0(\reftri)
	\end{align*}
	satisfying the following: For all $p \in (1,\infty)$, and $q \in 
	L^p_0(\reftri)$, there holds $\hat{\mathcal{L}}_{\dive} q \in \bdd{W}^{1, 
		p}_0(\reftri)$ with
	\begin{align}
		\label{eq:invert-div-ref-cont}
		\dive \hat{\mathcal{L}}_{\dive} q = q \quad \text{and} \quad \| 
		\hat{\mathcal{L}}_{\dive} q \|_{1, p, \reftri} \leq C(p) \|q\|_{p, 
			\reftri}.
	\end{align}
	Moreover, for all $N \in \mathbb{N}$, there holds
	\begin{align}
		\label{eq:invert-div-ref-poly}
		\hat{\mathcal{L}}_{\dive} : Q_0^{N-1}(\reftri) \to 
		\bdd{V}^N_0(\reftri),
	\end{align}
	and for every $\epsilon \in (0, 1)$, there holds
	\begin{align}
		\label{eq:invert-div-ref-cont-l1-linf}
		\| \hat{\mathcal{L}}_{\dive} q \|_{1, p, \reftri} \leq C(\epsilon) 
		N^{\epsilon} \|q\|_{p, \reftri}, \qquad p \in \{1, \infty\}, \ \forall 
		q \in Q_0^{N-1}(\reftri).
	\end{align}
\end{lemma}
\begin{proof}
	Let $\theta \in C^{\infty}_c(\reftri)$ and define the operator 
	$\mathcal{R}$ by the rule
	\begin{align*}
		\mathcal{R}(q)(\bdd{x}) := -\int_{\reftri} \theta(\bdd{y}) (\bdd{x}- 
		\bdd{y}) \int_{0}^{1} t q(\bdd{y} + t(\bdd{x}-\bdd{y})) \d{t} 
		\d{\bdd{y}}, \qquad \bdd{x} \in \reftri.
	\end{align*}
	Then,  \cite[Remark 3.5]{Costabel10} shows that $\mathcal{R}$ maps 
	$L^p(\reftri)$ boundedly into $\bdd{W}^{1, p}(\reftri)$ for all $p \in 
	[1,\infty]$ with 
	\begin{align*}
		\dive \mathcal{R}(q) = q \quad \text{and} \quad \| \mathcal{R}(q) 
		\|_{1, p, \reftri} \leq C(p) \| q \|_{p, \reftri} \qquad \forall q \in 
		L^p(\reftri),
	\end{align*}
	and $\mathcal{R}$ maps $\mathcal{P}_{N-1}(\reftri) \to 
	\mathcal{P}_{N}(\reftri)^2$ for all $N \in \mathbb{N}$. 
	
	We now define $\hat{\mathcal{L}}_{\dive} := (I - 
	\hat{\mathcal{L}}_{\partial})\mathcal{R}$. Let $p \in (1,\infty)$ and $q 
	\in L^p_0(\reftri)$. First note that $\hat{\mathcal{L}}_{\dive}$ is 
	well-defined since $\dive \mathcal{R} = I$. Moreover, that 
	$\hat{\mathcal{L}}_{\dive} q|_{\partial \reftri} = \bdd{0}$ and the 
	properties 
	in \cref{eq:invert-div-ref-cont} follow from 
	\cref{eq:div-free-extension-cont}, and so $\hat{\mathcal{L}}_{\dive}$ maps 
	$L^p_0(\reftri)$ boundedly into $\bdd{W}^{1, p}_0(\reftri)$.
	
	Assume further that $N \in \mathbb{N}$ and $q \in Q_0^{N-1}(\reftri)$. 
	Then, $\mathcal{R}(q) \in \bdd{V}^N(\reftri)$ with $\dive \mathcal{R}(q)  
	= q \in Q_0^{N-1}(\reftri)$, and so \cref{eq:div-free-extension-poly} 
	shows that $\hat{\mathcal{L}}_{\dive} q \in \bdd{V}^N_0(\reftri)$. 
	Moreover, the inverse inequality \cref{eq:lp-inverse-tri} gives
	\begin{align*}
		\| \hat{\mathcal{L}}_{\dive} q  \|_{1, 1, \reftri} &\leq C	\| 
		\hat{\mathcal{L}}_{\dive} q  \|_{1, 1+\epsilon, \reftri} \leq 
		C(\epsilon)  \|q\|_{1+\epsilon, \reftri} \leq C(\epsilon) 
		N^{4\epsilon} \|q\|_{1, \reftri}, \\
		\| \hat{\mathcal{L}}_{\dive} q  \|_{1, \infty, \reftri} &\leq 
		C(\epsilon) N^{4\epsilon} \| \hat{\mathcal{L}}_{\dive} q  \|_{1, 
			\frac{1}{\epsilon}, \reftri} \leq C(\epsilon) N^{4\epsilon} 
		\|q\|_{\frac{1}{\epsilon}, \reftri} \leq C(\epsilon) N^{4\epsilon} 
		\|q\|_{\infty, \reftri}.
	\end{align*}
\end{proof}

Using a variant of the Piola transformation and standard scaling arguments, we 
obtain an inverse of the divergence operator on a physical element.
\begin{corollary}
	\label{cor:invert-div-element-phys}
	For each $K \in \mathcal{T}$, there exists a linear operator 
	\begin{align*}
		\mathcal{L}_{\dive, K} : \bigcup_{1 < p < \infty} L^p_0(K) \to 
		\bdd{W}^{1,1}_0(K)	
	\end{align*}
	satisfying the following for all $p \in (1,\infty)$, $N \in \mathbb{N}$, 
	and $q \in L_0^{p}(K)$:
	\begin{align}
		\label{eq:invert-div-physical-interp}
		\mathcal{L}_{\dive, K} q \in \bdd{V}_0^{N}(K) \text{ if } q \in 
		Q_0^{N-1}(K), \quad \mathcal{L}_{\dive, K} q \in \bdd{W}^{1,p}_0(K), 
		\text{  and  } \dive \mathcal{L}_{\dive, K} q  = q.
	\end{align}
	Moreover, for all $p \in [1,\infty]$ and $\epsilon \in (0, 1)$, there holds
	\begin{multline}
		\label{eq:invert-div-physical-cont}
		h_K^{-1} \| \mathcal{L}_{\dive, K} q \|_{p, K} + |\mathcal{L}_{\dive, 
			K} q|_{1, p, K} \\ \leq \begin{cases}
			C(p) \|q\|_{p, K} & \text{if } p \in (1,\infty), \\
			C(\epsilon) N^{\epsilon} \|q\|_{p, K} & \text{if } p \in \{1, 
			\infty\}, \ q \in Q_0^{N-1}(K),
		\end{cases} 
	\end{multline}
	where $C(p)$ and $C(\epsilon)$ also depend on shape regularity.

\end{corollary}
\begin{proof}
	Let $K \in \mathcal{T}$ and for $q \in L^p_0(K)$, define 
	$\mathcal{L}_{\dive, K} q := \nabla \bdd{F}_K \hat{\mathcal{L}}_{\dive}(q 
	\circ 
	\bdd{F}_K) \circ \bdd{F}_K^{-1}$,  where  $\nabla \bdd{F}_K$ is the 
	Jacobian of $\bdd{F}_K$. Note that $\mathcal{L}_{\dive, K}$ is 
	well-defined since $q 
	\circ \bdd{F}_K \in L^p_0(\reftri)$. Moreover, the chain rule and 
	\cref{lem:invert-div-ref-element} give
	\begin{align*}
		\dive \mathcal{L}_{\dive, K} q = (\dive \hat{\mathcal{L}}_{\dive}(q 
		\circ \bdd{F}_K)) \circ \bdd{F}_K^{-1} = q,
	\end{align*}
	while a standard scaling argument and \cref{eq:invert-div-ref-cont} give
	\begin{multline*}
		h_K^{-1} \| \mathcal{L}_{\dive, K} q \|_{p, K} + | 
		\mathcal{L}_{\dive, K} q |_{1, p, K} \leq C(p) h_K^{\frac{2}{p}} \| 
		\hat{\mathcal{L}}_{\dive}(q \circ \bdd{F}_K) \|_{1, p, \reftri} \\
		\leq C(p) h_K^{\frac{2}{p}} \|q \circ \bdd{F}_K \|_{p, \reftri} 
		\leq C(p) \|q\|_{p, K}.
	\end{multline*}
	The same scaling argument for $q \in Q_0^{N-1}(K)$ and $p \in 
	\{1,\infty\}$ combined with \cref{eq:invert-div-ref-cont-l1-linf} 
	completes the proof of \cref{eq:invert-div-physical-cont}. The remainder 
	of \cref{eq:invert-div-physical-interp} follows from 
	\cref{eq:invert-div-ref-cont,eq:invert-div-ref-poly}.
\end{proof}

\subsection{Multi-element mesh}	

We now extend \cref{cor:invert-div-element-phys} to a multi-element mesh, 
which requires some additional notation. Let $\mathcal{E}$ denote the sets of 
edges of the mesh, partitioned into interior edges $\mathcal{E}_I$ laying on 
the interior of $\Omega$ and boundary edges $\mathcal{E}_B$ laying on 
$\partial \Omega$.  

\subsubsection{Inverting element averages}

The first result is a partial inverse of the divergence operator that matches 
element averages and is a slight generalization of the stability of the 
$\bdd{V}^2 \times DG^0(\mathcal{T})$ element.
\begin{lemma}
	\label{lem:interp-avg-div}
	
	There exists a linear operator 
	\begin{align*}
		\mathcal{L}_{\mathrm{avg}}: \bigcup_{1 < p < \infty} 
		L^p_{\Gamma}(\Omega) \to \bdd{V}^2_{\Gamma}
	\end{align*}
	satisfying the following for all $p \in (1,\infty)$ and $q \in 
	L^p_{\Gamma}(\Omega)$:
	\begin{align}
		\label{eq:interp-avg-div}
		\int_{K} \dive \mathcal{L}_{\mathrm{avg}} q \d{\bdd{x}} = \int_{K} q 
		\d{\bdd{x}} \quad \forall K \in \mathcal{T} \quad \text{and} \quad 
		\|\mathcal{L}_{\mathrm{avg}} q \|_{1,p,\Omega} \leq C(p) 
		\|q\|_{p,\Omega},
	\end{align}
	where $C$ also depends on $\Omega$, $\Gamma_0$, and shape regularity.
\end{lemma}
\begin{proof}		
	Let $p \in (1,\infty)$ and $q \in L^p_{\Gamma}(\Omega)$ be given. Let 
	$\bdd{v} := \mathcal{L}_{\dive, \Gamma_0} q \in W^{1,p}_{\Gamma}(\Omega)$, 
	where $\mathcal{L}_{\dive, \Gamma_0}$ is given by 
	\cref{lem:invert-div-cont}. Let $\bdd{v}_1 \in \bdd{V}^1_{\Gamma}(\Omega)$ 
	be the Scott-Zhang interpolant \cite{Scott90} of $\bdd{v}$ satisfying
	\begin{align*}
		h_K^{\frac{1}{p}-1} \| \bdd{v}_1 - \bdd{v} \|_{p, \partial K} \leq 
		C(p) \|\bdd{v}\|_{1, p, \omega_K} \text{  and  } 
		\|\bdd{v}_1\|_{1,p,\Omega} \leq C(p) \|\bdd{v}\|_{1, p, \Omega}, \quad 
		1 < p < \infty.
	\end{align*}
	Define $\bdd{v}_{\mathcal{E}} \in \bdd{V}^2$ by assigning the following 
	degrees of freedom:
	\begin{align*}
		\int_{e} \bdd{v}_{\mathcal{E}} \d{s} = \int_{e} (\bdd{v} - \bdd{v}_1) 
		\d{s} \quad \forall e \in \mathcal{E} \quad \text{and} \quad 
		\bdd{v}(\vertex{a}) = 0 \quad \forall \vertex{a} \in \mathcal{V}.
	\end{align*}
	Since $\bdd{v}, \bdd{v}_1 \in \bdd{W}^{1, p}_{\Gamma}(\Omega)$, 
	$\bdd{v}_{\mathcal{E}} \in \bdd{V}^2_{\Gamma}(\Omega)$. Let $p \in 
	(1,\infty)$ be given. A standard equivalence of norm and scaling argument 
	shows that
	\begin{align*}
		h_K^{-1} \| \bdd{v}_{\mathcal{E}} \|_{p, K} + 
		|\bdd{v}_{\mathcal{E}}|_{1, p, K} \leq C(p) h_K^{\frac{2}{p}-2} 
		\sum_{e \in \mathcal{E}_K} \left| \int_{e} \bdd{v}_{\mathcal{E}} \d{s} 
		\right| 
		&\leq C(p) h_K^{\frac{1}{p}-1} \| \bdd{v} - \bdd{v}_1 \|_{p, \partial 
			K} \\
		&\leq C(p) \|\bdd{v}\|_{1,p,\omega_K}.
	\end{align*}
	
	Define $\mathcal{L}_{\mathrm{avg}} q := \bdd{v}_1 + 
	\bdd{v}_{\mathcal{E}}$. Then, for all $K \in \mathcal{T}$, there holds
	\begin{align*}
		\int_{K} \dive \mathcal{L}_{\mathrm{avg}} q \d{\bdd{x}} = \sum_{e \in 
			\mathcal{E}_K} \int_{e} (\mathcal{L}_{\mathrm{avg}} q) \cdot 
		\unitvec{n} \d{s} = \sum_{e \in \mathcal{E}_K} \int_{e} \bdd{v} \cdot 
		\unitvec{n} \d{s} &= \int_{K} \dive \bdd{v} \d{\bdd{x}} \\
		&= \int_{K} q \d{\bdd{x}}.
	\end{align*}
	Collecting results gives $\|\mathcal{L}_{\mathrm{avg}} q \|_{1, p, 
		\Omega} 
	\leq C(p) \|\bdd{v}\|_{1,p, \Omega}$, and so \cref{eq:interp-avg-div} 
	follows from \cref{eq:invert-div-cont}.
\end{proof}

\subsubsection{Inverting vertex values}

The next result is a partial inverse of the divergence operator that matches 
values at vertices. We first need a scaled version of the polynomial lifting 
operator in \cite{Parker23}.

\begin{lemma}
	\label{lem:poly-extension-element}		
	Let $K \in \mathcal{T}$ and $N \in \mathbb{N}_0$, and let $f \in 
	C(\partial K)$ satisfying $f|_{e} 
	\in \mathcal{P}_{N}(e)$ for all $e \in \mathcal{E}_K$ and $f(\vertex{a}) = 
	0$ for all $\vertex{a} \in \mathcal{V}_K$ be given. Then, there exists 
	$u \in \mathcal{P}_{N}(K)$ such that $u|_{\partial K} = f$ and for all $p 
	\in (1,\infty)$, there holds
	\begin{align}
		\label{eq:poly-extension-element}
		h_K^{-1} \|u\|_{p, K} + |u|_{1, p, K} \leq C(p) \left( \sum_{e \in 
			\mathcal{E}_K} h_K^{\frac{1}{p}-1} \|f\|_{p, e} +  
		\iseminorm{00}{f}{1-\frac{1}{p}, p, e} \right),
	\end{align}
	where $C(p)$ also depends on shape regularity and
	\begin{align*}
		\iseminormsup{00}{f}{1-\frac{1}{p}, p, e}{p} := |f|_{1-\frac{1}{p}, p, 
			e}^p + \begin{cases}
			\int_{e} \frac{|f(\bdd{y})|^2}{\mathrm{dist}(\bdd{y}, \partial e)} 
			\d{s} & \text{if } p = 2, \\
			0 & \text{otherwise}.
		\end{cases}
	\end{align*}
\end{lemma}
\begin{proof}
	Let $K \in \mathcal{T}$ and $f \in C(\partial K)$ be as in the statement 
	of the lemma. Define $\hat{f} = f \circ \bdd{F}_K$. \cite[Theorem 
	3.2]{Parker23} shows that there exists $\hat{u} \in 
	\mathcal{P}_N(\reftri)$ satisfying $\hat{u}|_{\partial \reftri} = \hat{f}$ 
	and
	\begin{align*}
		\|\hat{u}\|_{1, p, \reftri}^p \leq C(p) \sum_{\hat{\gamma} \in 
			\mathcal{E}_{\reftri}} \|\hat{f} \|_{1-\frac{1}{p}, p, \hat{\gamma}}^p 
		+ C(p) \sum_{i=1}^{3} \begin{cases}
			0 & \text{if } p \neq 2, \\
			\mathcal{I}_{i}(\hat{f}_{i+1}, \hat{f}_{i+2}) & \text{if } p=2,
		\end{cases}
	\end{align*}
	where $\hat{f}_j := \hat{f}|_{\hat{\gamma}_j}$, indices are understood 
	modulo 3, and
	\begin{align*}
		\mathcal{I}_{i}(\hat{f}_{i+1}, \hat{f}_{i+2}) = \int_{0}^{1} s^{-1} | 
		\hat{f}_{i+1}( \vertex{a}_i - s \unitvec{t}_{i+1}) - 
		f_{i+2}(\vertex{a}_i + s \unitvec{t}_{i+2})|^2 \d{s}.
	\end{align*}
	Since $\hat{f}(\vertex{a}_i) = 0$, there holds
	\begin{align*}
		\sum_{i=1}^{3} \mathcal{I}_{i}(\hat{f}_{i+1}, \hat{f}_{i+2}) \leq C 
		\sum_{\hat{\gamma} \in \mathcal{E}_{\reftri} } 
		\iseminormsup{00}{\hat{f}}{\frac{1}{2}, 2, \hat{\gamma}}{2},
	\end{align*}
	and so $u = \hat{u} \circ \bdd{F}_K^{-1}$ then satisfies $u|_{\partial T} 
	= f$ and \cref{eq:poly-extension-element} follows from a standard scaling 
	argument.	
\end{proof}

\begin{lemma}
	\label{lem:interp-div-vertices}
	For $N \geq 4$ and $\vertex{a} \in \mathcal{V}$, there exists a linear 
	operator
	\begin{align*}
		\mathcal{L}_{N, \vertex{a}} : Q^{N-1}_{\Gamma} \to \bdd{V}^{N}_{\Gamma}
	\end{align*}
	satisfying the following for $q \in Q^{N-1}_{\Gamma}$:
	\begin{enumerate}
		\item[(i)] $\mathcal{L}_{N, \vertex{a}} q = \bdd{0}$ on $\Omega 
		\setminus  
		\omega_{\vertex{a}}$;
		\item[(ii)] $\dive \mathcal{L}_{N, \vertex{a}} q |_{K}(\vertex{b}) = 
		\delta_{\vertex{a} \vertex{b}} q|_{K}(\vertex{a})$ for all $\vertex{b} 
		\in \mathcal{V}_{K}$ and all $K \in \mathcal{T}_{\vertex{a}}$;
		\item[(iii)] $\int_{K} \dive  \mathcal{L}_{N, \vertex{a}} q  
		\d{\bdd{x}} = 0$ for all $K \in \mathcal{T}$;
	\end{enumerate}
	and for all $p \in [1,\infty]$, $\epsilon \in (0, 1)$, and $K \in 
	\mathcal{T}$, there holds
	\begin{align}
		\label{eq:div-interp-vertex-local}
		h_K^{-1} \| \mathcal{L}_{N, \vertex{a}} q \|_{p, K} +  
		|\mathcal{L}_{N, \vertex{a}} q|_{1,p,K} \leq 
		\tilde{\xi}_{\vertex{a}}^{-1} \|q\|_{p, \omega_{\vertex{a}}} \cdot 
		\begin{cases}
			C(p)  & \text{if } p \in (1,\infty), \\
			C(\epsilon) N^{\epsilon}  & \text{if }  p \in \{1,\infty\},
		\end{cases} 
	\end{align}
	where $C(p)$ and $C(\epsilon)$ also depend on shape regularity and
	\begin{align}
		\tilde{\xi}_{\vertex{a}} := \begin{cases}
			\xi_{\vertex{a}} & \text{if } \vertex{a} \in (\mathcal{V}_I \cup 
			\mathcal{V}_{B, 0}) \setminus \mathcal{V}_{S}, \\
			1 & \text{otherwise}.
		\end{cases}
	\end{align}
\end{lemma}
\begin{proof}
	Let $N \geq 4$, $\vertex{a} \in \mathcal{V}$, and $q \in Q_{\Gamma}^{N-1}$ 
	be given.
	
	\noindent \textbf{Step 1: Interior vertex. } Assume first that $\vertex{a} 
	\in \mathcal{V}_I$. Label the elements as in \cref{fig:internal schema} 
	with $m = |\mathcal{T}_{\vertex{a}}|$ and let $e_i$, $1 \leq i \leq m$, 
	denote the common edge between $K_{i}$ and $K_{i-1}$ with unit tangent 
	$\unitvec{t}_i$ pointing towards $\vertex{a}$. \\
	
	\noindent \textbf{Part (a): Algebraic system. } Thanks to \cite[Lemma 
	4.2]{AinCP21LE}, there exist vectors $\bdd{d}_1, \bdd{d}_2, \ldots, 
	\bdd{d}_m \in \mathbb{R}^2$ satisfying
	\begin{subequations}
		\label{eq:proof:div-interp-solution}
		\begin{alignat}{2}
			%
			(\sin \theta_i) q|_{K_i}(\vertex{a})  &= \bdd{d}_i \cdot 
			\unitvec{n}_{i+1} - \bdd{d}_{i+1} \cdot \unitvec{n}_{i}, \qquad & 
			&1 \leq i \leq m, \\
			\max_{1 \leq i \leq m} |\bdd{d}_i| &\leq C 
			\tilde{\xi}_{\vertex{a}}^{-1} \max_{1 \leq i \leq m} 
			|q|_{K_i}(\vertex{a})|, \qquad & &
		\end{alignat}
	\end{subequations}
	where the indices are understood modulo 3, and $\unitvec{n}_i = 
	-\unitvec{t}_i^{\perp}$. \\
	
	\noindent \textbf{Part (b): Extension from the boundary. } For each 
	$i \in 
	\{1,2,\ldots, m\}$, define boundary data $\bdd{f}_i : \gamma_i \to 
	\mathbb{R}^2$ by
	\begin{align}
		\label{eq:proof:div-interp-trace-def}
		\bdd{f}_i(\bdd{y}) :=  \frac{|\gamma_i|}{2}\tilde{\Upsilon}_N \left( 
		\frac{2|\vertex{a} - \bdd{y}|}{|\gamma_i|} -1 \right) \bdd{d}_i, 
		\qquad \bdd{y} \in e_i,
	\end{align}
	where $\tilde{\Upsilon}_N$ is given by \cref{eq:1d-interp-poly-noavg-def} 
	with any fixed $\alpha > 7/2$, say $\alpha = 4$. Note that 
	$\bdd{f}_i|_{\partial e_i} = \bdd{0}$, and so 
	\cref{lem:poly-extension-element} shows that there exists $\bdd{\psi}_i 
	\in \mathcal{P}_N(K_i)^2$, $1 \leq i \leq m$, satisfying
	\begin{align*}
		\bdd{\psi}_i|_{e_i} = \bdd{f}_i, \quad \bdd{\psi}_i|_{e_{i+1}} = 
		\bdd{f}_{i+1}, \quad \text{and} \quad \bdd{\psi}_i = \bdd{0} \text{ on 
		} \partial K \setminus \{ e_i \cup e_{i+1} \}.
	\end{align*}
	Moreover, for $p \in (1,\infty)$, there holds
	\begin{align*}
		h_K^{-\frac{2}{p}} \|\bdd{\psi}_i\|_{p, K} + h_K^{1-\frac{2}{p}} 
		|\bdd{\psi}_i|_{1, p, K} \leq C(p)  \left( \sum_{j=i}^{i+1} 
		h_K^{-\frac{1}{p}} \|\bdd{f}_j\|_{p, e_j} + h_K^{1-\frac{2}{p}} 
		\iseminorm{00}{\bdd{f}_j}{1-\frac{1}{p}, p, e_j} \right)
	\end{align*}
	along with the bound
	\begin{multline*}
		\sum_{j=i}^{i+1} h_K^{-\frac{1}{p}} \|\bdd{f}_j\|_{p, e_j} + 
		h_K^{1-\frac{2}{p}} 
		\iseminorm{00}{\bdd{f}_j}{1-\frac{1}{p}, p, e_j} \\ 
		\leq C(p) h_K \left( \max_{j=i, i+1} |\bdd{d}_j| \right) \cdot 
		\begin{cases}
			\| \tilde{\Upsilon}_N \|_{1 - \frac{1}{p}, p, I} & \text{if } p 
			\neq 2, \\
			\| \tilde{\Upsilon}_N \|_{2, I} + 
			\iseminorm{00}{\tilde{\Upsilon}_N}{\frac{1}{2}, 2, I} & \text{if } 
			p = 2,
		\end{cases}
	\end{multline*} 	
	If $p \in (1,\infty) \setminus \{2\}$, we apply 
	\cref{eq:1d-interp-poly-sobolev} to conclude that
	\begin{align}
		\label{eq:proof:psi-i-decay}
		h_K^{-\frac{2}{p}} \|\bdd{\psi}_i\|_{p, K} + h_K^{1-\frac{2}{p}} 
		|\bdd{\psi}_i|_{1, p, K} \leq C(p) \tilde{\xi}_{\vertex{a}}^{-1} h_K 
		N^{-\frac{4}{p}} \|q\|_{\infty, \omega_{\vertex{a}}},
	\end{align}
	while for $p=2$, we use \cref{eq:1d-interp-poly-sobolev} and function 
	space interpolation 
	to obtain 
	\begin{align*}
		\| \tilde{\Upsilon}_N \|_{2, I} + 
		\iseminorm{00}{\tilde{\Upsilon}_N}{\frac{1}{2}, 2, \gamma} \leq C \| 
		\tilde{\Upsilon}_N \|_{2, I}^{\frac{1}{2}} \| \tilde{\Upsilon}_N 
		\|_{1, 2, I}^{\frac{1}{2}} \leq C N^{-2},
	\end{align*}
	and so \cref{eq:proof:psi-i-decay} holds for all $p \in (1,\infty)$. For 
	$p \in \{1,\infty\}$, we apply analogous arguments to the proof of 
	\cref{lem:invert-div-ref-element} to conclude that for all $\epsilon \in 
	(0, 1)$ there holds
	\begin{align*}
		h_K^{-\frac{2}{p}} \|\bdd{\psi}_i\|_{p, K} + h_K^{1-\frac{2}{p}} 
		|\bdd{\psi}_i|_{1, p, K} \leq C(e) \tilde{\xi}_{\vertex{a}}^{-1} h_K 
		N^{-\frac{4}{p} + \epsilon} \|q\|_{\infty, \omega_{\vertex{a}}} \quad 
		\forall p \in \{1,\infty\}.
	\end{align*}
	Applying \cref{eq:lp-inverse-tri} and a standard scaling argument gives
	\begin{align*}
		\|q\|_{\infty, \omega_{\vertex{a}}} \leq C(p) h_K^{-\frac{2}{p}} 
		N^{\frac{4}{p}} \|q\|_{p, \omega_{\vertex{a}}} \qquad \forall p \in 
		[1,\infty],
	\end{align*}
	and so for all $p \in [1,\infty]$ and $\epsilon \in (0, 1)$ there holds
	\begin{align*}
		h_K^{-\frac{2}{p}} \|\bdd{\psi}_i\|_{p, K} + h_K^{1-\frac{2}{p}} 
		|\bdd{\psi}_i|_{1, p, K} \leq \tilde{\xi}_{\vertex{a}}^{-1} 
		h_K^{1-\frac{2}{p}} \|q\|_{p, \omega_{\vertex{a}}} \begin{cases}
			C(p)  & \text{if } p \in (1,\infty), \\
			C(\epsilon)  N^{\epsilon} & \text{if } p \in \{1, \infty\}.
		\end{cases}
	\end{align*}
	
	\textbf{Part (c): Verifying interpolation. } Also note that $\grad 
	\bdd{\psi}_i(\vertex{b}) = \bdd{0}$ for $\vertex{b} \in \mathcal{V}_{K_i} 
	\setminus \{ \vertex{a} \}$, and so there holds
	\begin{align*}
		( \sin \theta_i )\dive \bdd{\psi}_i(\vertex{a}) &= \frac{\partial 
			\bdd{\psi}_i}{\partial \unitvec{t}_i} (\vertex{a}) \cdot 
		\unitvec{n}_{i+1} -  \frac{\partial \bdd{\psi}_i}{\partial 
			\unitvec{t}_{i+1}}(\vertex{a}) \cdot \unitvec{n}_{i} \\ 
		&=  \frac{\partial \bdd{f}_i}{\partial \unitvec{t}_i} (\vertex{a}) 
		\cdot \unitvec{n}_{i+1} -  \frac{\partial \bdd{f}_{i+1}}{\partial 
			\unitvec{t}_{i+1}}(\vertex{a}) \cdot \unitvec{n}_{i} \\
		&= \Upsilon_N'(-1) \left( \bdd{d}_i \cdot \unitvec{n}_{i+1} - 
		\bdd{d}_{i+1} \cdot \unitvec{n}_i \right) \\
		&= (\sin \theta_i) q|_{K_i}(\vertex{a}) ,
	\end{align*}
	while 
	\begin{align*}
		\int_{K} \dive \bdd{\psi}_i(\bdd{x}) \d{\bdd{x}} = \int_{\partial K} 
		\bdd{\psi}_i \cdot \unitvec{n} \d{s} = \sum_{j=i}^{i+1} \int_{e_j} 
		\bdd{f}_j \cdot \unitvec{n}_j \d{s} = 0
	\end{align*}
	since $\int_{-1}^{1} \tilde{\Upsilon}_N(t) \d{t} = 0$. We then define 
	$\mathcal{L}_{N, \vertex{a}} q|_{K_i} = \bdd{\psi}_i$, $1 \leq i \leq m$, 
	and $\mathcal{L}_{N, \vertex{a}} = \bdd{0}$ on $\Omega \setminus 
	\omega_{\vertex{a}}$. Note that $\mathcal{L}_{N, \vertex{a}} q \in 
	\bdd{V}_{\Gamma}^{N}$ and satisfies (i-iii) and 
	\cref{eq:div-interp-vertex-local} by construction.	\\

	\noindent \textbf{Step 2: Boundary vertex. } Now let $\vertex{a} \in 
	\mathcal{V}_B$ and label the elements as in \cref{fig:boundary schema} 
	with $m = |\mathcal{T}_{\vertex{a}}|$. We label the edges in 
	$\mathcal{E}_{\vertex{a}}$ analogously with $e_1, e_{m+1} \in 
	\mathcal{E}_B$ Thanks to \cite[Lemma 4.3]{AinCP21LE}, there exist 
	$\bdd{d}_1, \bdd{d}_2, \ldots, \bdd{d}_{m+1} \in \mathbb{R}^2$ satisfying 
	\cref{eq:proof:div-interp-solution} and additionally $\bdd{d}_j = \bdd{0}$ 
	if $e_j \in \Gamma_{0}$ for $j \in \{1, m+1\}$. With these coefficients in 
	hand, we may again define $\bdd{f}_i$ according to 
	\cref{eq:proof:div-interp-trace-def} for $m \in \{1,2,\ldots, m+1\}$ and 
	construct $\mathcal{L}_{N, \vertex{a}} q$ as above.
\end{proof}

\subsubsection{Proof of \cref{thm:invert-div}}
\label{sec:proof-invert-div}

Let $N \geq 4$ and $q \in Q_{\Gamma}^{N-1}$ be given. Let
\begin{align*}
	\bdd{\psi} := \mathcal{L}_{\mathrm{avg}} q + \sum_{\vertex{a} \in 
		\mathcal{V}} \mathcal{L}_{N, \vertex{a}} (I - \dive 
	\mathcal{L}_{\mathrm{avg}})q \quad \text{and} \quad r := q - \dive 
	\bdd{\psi}.
\end{align*}
Thanks to \cref{lem:interp-avg-div,lem:interp-div-vertices}, there holds
\begin{align*}
	\|\bdd{\psi}\|_{1, p, \Omega} \leq C(p) \|q\|_{p, \Omega} \qquad \forall p 
	\in (1,\infty) \quad \text{and} \quad r|_{K} \in Q_0^{N-1}(K) \qquad 
	\forall K \in \mathcal{T}. 
\end{align*}
We then define $\mathcal{L}_{\dive, N} q$ element-by-element by
\begin{align*}
	\mathcal{L}_{\dive, N} q|_{K} := \bdd{\psi}|_{K} + \mathcal{L}_{\dive, K}( 
	r|_{K}) \qquad \text{on } K,
\end{align*}
where $\mathcal{L}_{\dive, K}$ is given by \cref{cor:invert-div-element-phys}. 
$\mathcal{L}_{\dive, N} q \in \bdd{V}_{\Gamma}^{N}$ by construction and the
properties \cref{eq:invert-div-fem} now follow from 
\cref{cor:invert-div-element-phys}. \hfill \qedsymbol

\section{$L^p$ stability of the $L^2$-projection}
\label{sec:lp-stability-l2-projection}

For $N \in \mathbb{N}_0$ and a collection of elements $\mathcal{U} 
\subseteq \mathcal{T} \cup \{ \reftri \}$, let 
$\mathcal{V}_{\mathcal{U}}$ denote the 
collection of mesh vertices of $\mathcal{U}$ and  
$\mathcal{V}_{S, \mathcal{U}} := \{ \vertex{a} \in \mathcal{V}_S \cap 
\mathcal{V}_{\mathcal{U}} : \mathcal{T}_{\vertex{a}} \subset \mathcal{U} 
\}$ and define
\begin{align*}
	\widetilde{DG}_{\Gamma}^N(\mathcal{U}) &:= \left\{ q \in 
	DG^N(\mathcal{U}) : \mathcal{A}_{\vertex{a}}(q) = 0 \ \forall 
	\vertex{a} \in \mathcal{V}_{S, \mathcal{U}} \right\},
\end{align*}
where $\mathcal{A}_{\vertex{a}}$ is defined in 
\cref{eq:alternating-sum-condition}. Note that $Q_{\Gamma}^{N} = 
\widetilde{DG}_{\Gamma}^N(\mathcal{T}) \cap L^1_{\Gamma}(\Omega)$. Let 
$\omega_{\mathcal{U}} := \mathrm{int}(\cup_{K \in \mathcal{U}} \bar{K})$ 
and let $\tilde{\mathbb{P}}_{N, \mathcal{U}} : L^1(\omega_{\mathcal{U}}) 
\to \widetilde{DG}^N_{\Gamma}(\mathcal{U})$ be the 
$L^2(\omega_{\mathcal{U}})$-projection operator onto $ 
\widetilde{DG}^N_{\Gamma}(\mathcal{U})$:
\begin{align*}
	( \tilde{\mathbb{P}}_{N, \mathcal{U}} q, r )_{\omega_{\mathcal{U}}} = 
	(q, r)_{\omega_{\mathcal{U}}} \qquad \forall r \in 
	\widetilde{DG}^N_{\Gamma}(\mathcal{U}), \ \forall q \in 
	L^1(\omega_{\mathcal{U}}).
\end{align*}
The first result shows that the global projection $\tilde{\mathbb{P}}_{N, 
	\mathcal{T}}$ is local to elements or vertex patches provided that the 
mesh satisfies \ref{cond:mesh-disjoint-sing}.

\begin{lemma}
	\label{lem:l2-projection-local}
	Suppose that the mesh satisfies \ref{cond:mesh-disjoint-sing} and let 
	$N \in \mathbb{N}_0$. Then, for all $q \in L^1(\Omega)$, there holds
	\begin{subequations}
		\begin{alignat}{2}
			\label{eq:l2-projection-element-equiv}
			(\tilde{\mathbb{P}}_{N, \mathcal{T}} q)|_{K} &= 
			\tilde{\mathbb{P}}_{N, K} (q|_{K}) \qquad & &\forall K \in 
			\mathcal{T} :  \mathcal{V}_K \cap \mathcal{V}_S = \emptyset, 
			\\
			\label{eq:l2-projection-sing-patch-equiv}
			(\tilde{\mathbb{P}}_{N, \mathcal{T}} 
			q)|_{\omega_{\vertex{a}}} &= \tilde{\mathbb{P}}_{N, 
				\mathcal{T}_{\vertex{a}}} (q|_{\omega_{\vertex{a}}}) \qquad & 
			&\forall \vertex{a} \in \mathcal{V}_S.
		\end{alignat}
	\end{subequations}
\end{lemma}
\begin{proof}
	If $K \in \mathcal{T}$ satisfies $\mathcal{V}_K \cap \mathcal{V}_S = 
	\emptyset$, then the zero extension of $r \in \mathcal{P}_N(K)$ 
	satisfies $r \in \widetilde{DG}^N_{\Gamma}$, and so 
	\cref{eq:l2-projection-element-equiv} follows. Similarly, if 
	$\vertex{a} \in \mathcal{V}_S$, then \ref{cond:mesh-disjoint-sing} 
	ensures that $\vertex{a}$ is the only singular vertex of elements in 
	$\mathcal{T}_{\vertex{a}}$, and so the zero extension of $r \in 
	\widetilde{DG}^N_{\Gamma}(\mathcal{T}_{\vertex{a}})$ belongs to 
	$\widetilde{DG}^N_{\Gamma}$.
\end{proof}

The next result provides $L^p$ stability bounds of $\tilde{\mathbb{P}}_{N, 
	\mathcal{U}}$ explicit in the polynomial degree.	
\begin{lemma}
	\label{lem:l2-proj-lp-stability-local}
	Let $N \in \mathbb{N}$ and $\mathcal{U}$ be either a single element 
	$K \in \mathcal{T}$ or a patch around vertex 
	$\mathcal{T}_{\vertex{a}}$ for some $\vertex{a} \in \mathcal{V}_S$. 
	For all  $p \in [1,\infty]$, there holds
	\begin{align}
		\label{eq:l2-proj-lp-stability-local}
		\| \tilde{\mathbb{P}}_{N, \mathcal{U}} q \|_{p, 
			\omega_{\mathcal{U}}} \leq C N^{3\left| \frac{1}{2} - 
			\frac{1}{p} \right|} \|q\|_{p, \omega_{\mathcal{U}}} \qquad 
		\forall q \in L^p(\omega_{\mathcal{U}}),
	\end{align}
	where $C$ depends only on shape regularity. Consequently, if the 
	mesh satisfies \ref{cond:mesh-disjoint-sing}, then there holds
	\begin{align}
		\label{eq:l2-proj-lp-stability-dg}
		\| \tilde{\mathbb{P}}_{N, \mathcal{T}} q \|_{p, \Omega} \leq C 
		N^{3\left| \frac{1}{2} - \frac{1}{p} \right|} \|q\|_{p, \Omega} 
		\qquad \forall q \in L^p(\Omega), \ \forall p \in [1,\infty].
	\end{align}
\end{lemma}
\begin{proof}
	\textbf{Step 1: Single element. } Taking $d=2$ and $\kappa_i = 1/2$ 
	for $i \in \{1,2,3\}$,in \cite[Corollary 2.9]{DaiXu09} gives
	\begin{align*}
		\| \tilde{\mathbb{P}}_{N, \reftri} q \|_{p, \reftri} \leq C 
		N^{\frac{3}{2}} \|q\|_{p, \reftri} \qquad \forall q \in 
		L^p(\reftri), \ p \in \{1, \infty\}. 
	\end{align*}
	Interpolating with $p=2$ and applying a standard scaling argument, we 
	obtain
	\begin{align*}
		\| \tilde{\mathbb{P}}_{N, K} q \|_{p, K} \leq C 
		N^{3\left| \frac{1}{2} - \frac{1}{p} \right|} \|q\|_{p, K} 
		\qquad \forall q \in L^p(K), \ \forall p \in [1, \infty], \ 
		\forall K \in \mathcal{T}.
	\end{align*}
	
	\noindent \textbf{Step 2: Patch around singular vertex. } Let 
	$\vertex{a} \in 
	\mathcal{V}_{S}$ and $\mathcal{U} = \mathcal{T}_{\vertex{a}}$. Let
	\begin{align*}
		b_{\vertex{a}} := \sum_{j=1}^{|\mathcal{T}_{\vertex{a}}|} 
		\frac{(-1)^{j+k}}{|K_j|} P_{N}^{(0, 2)}(1-2\lambda_{K_j, 
			\vertex{a}}) \mathds{1}_{K_j},
	\end{align*}
	where $P_{N}^{(0, 2)} : [-1, 1] \to \mathbb{R}$ is usual $(0, 
	2)$-Jacobi polynomial \cite{Szego75}, normalized as in 	
	\cref{eq:jacobi-normalization} below, $\lambda_{K_j, \vertex{a}}$ is 
	the barycentric coordinate associated to $\vertex{a}$ on $K_j \in 
	\mathcal{T}_{\vertex{a}}$, and $\mathds{1}_{K_j}$ is the indicator 
	function of $K_j$. Thanks to \cite[Lemma 6, 3.]{Grable24}, we have 
	the following orthogonal decomposition:
	\begin{align*}
		DG^N(\mathcal{U}) = \widetilde{DG}_{\Gamma}^{N}(\mathcal{U}) 
		\oplus^{\perp} \spann \left\{ b_{\vertex{a}} \right\},
	\end{align*}
	and so for $q \in L^1(\omega_{\mathcal{U}})$, there holds
	\begin{align*}
		\tilde{\mathbb{P}}_{N, \mathcal{U}} q = \mathbb{P}_{N, 
			\mathcal{U}} q 
		- \frac{(b_{\vertex{a}}, q)_{\omega_{\mathcal{U}}} }{ 
			\|b_{\vertex{a}}\|_{\omega_{\mathcal{U}}}^2 } b_{\vertex{a}}, 
	\end{align*}
	where $\mathbb{P}_{N, \mathcal{U}} :L^1(\omega_{\mathcal{U}}) \to 
	DG^N(\mathcal{U})$ is the $L^2$-projection operator. Consequently,
	\begin{align*}
		\| \tilde{\mathbb{P}}_{N, \mathcal{U}} q \|_{p, 
			\omega_{\mathcal{U}}} \leq \|  \mathbb{P}_{N, \mathcal{U}} q 
		\|_{p, \omega_{\mathcal{U}}} + \frac{ \|b_{\vertex{a}} \|_{p, 
				\omega_{\mathcal{U}}} \|b_{\vertex{a}}\|_{p', 
				\omega_{\mathcal{U}}} }{ \| b_{\vertex{a}}\|_{2, 
				\omega_{\mathcal{U}}}^2 } \|q\|_{p, \omega_{\mathcal{U}}}. 
	\end{align*}
	
	\noindent \textbf{Part (a): $L^1$, $L^2$, and $L^{\infty}$ bounds for 
		$b_{\vertex{a}}$. } We first have 
	that
	\begin{align*}
		\|b_{\vertex{a}}\|_{1, \omega_{\mathcal{U}}} = 
		\sum_{j=1}^{|\mathcal{T}_{\vertex{a}}|} \frac{1}{|K_j|} 
		\int_{K_j} 
		|P_N^{(0, 2)}(1 - 2\lambda_{K_j, \vertex{a}})| \d{\bdd{x}} &= 
		|\mathcal{T}_{\vertex{a}}|  \int_{-1}^{1} (1+t) 
		|P_N^{(0, 2)}(t)| \d{t} \\
		&\leq C N^{-\frac{1}{2}}
	\end{align*}
	since applying \cite[Theorem 7.34]{Szego75} gives
	\begin{align*}
		\int_{-1}^{1} (1+t) |P_N^{(0, 2)}(t)| \d{t} \leq 2 \int_{0}^{1} 
		|P_N^{(0, 2)}(t)| \d{t} + \int_{0}^{1} (1-t) 
		|P_N^{(2, 0)}(t)| \leq C N^{-\frac{1}{2}}.
	\end{align*}
	Additionally, \cite[Lemma 6, 1. \& 2.]{Grable24} show that
	\begin{align*}
		\|b_{\vertex{a}}\|_{2, \omega_{\mathcal{U}}}^2 = 
		\sum_{j=1}^{|\mathcal{T}_{\vertex{a}}|} \frac{1}{|K_j|} 
		\int_{K_j}   \|b_{\vertex{a}}\|_{2, K_j}^2 = 
		\sum_{i=1}^{|\mathcal{T}_{\vertex{a}}|} \frac{1}{|K_j} \leq C 
		\mathrm{diam}(\omega_{\vertex{a}})^{-2}.
	\end{align*}
	An $L^{\infty}$ bound follows immediately from \cite[Theorem 
	7.32.1]{Szego75}, as
	\begin{align*}
		\|b_{\vertex{a}}\|_{\infty, \omega_{\mathcal{U}}} \leq C 
		\mathrm{diam}(\omega_{\vertex{a}})^{-2} \|P_N^{(0, 2)}\|_{\infty, 
			(-1, 1)} \leq C \mathrm{diam}(\omega_{\vertex{a}})^{-2} N^{2}.
	\end{align*}
	
	\noindent \textbf{Part (b): $L^p$ bound for the $L^2$ projection. } 
	Collecting 
	results, we have
	\begin{alignat*}{2}
		\| \tilde{\mathbb{P}}_{N, K} q \|_{1, \omega_{\mathcal{U}}} &\leq 
		C N^{\frac{3}{2}} \|q\|_{1, \omega_{\mathcal{U}}} \qquad & 
		&\forall q \in L^1(\omega_{\mathcal{U}}), \\
		\| \tilde{\mathbb{P}}_{N, K} q \|_{\infty, 
			\omega_{\mathcal{U}}} &\leq C N^{\frac{3}{2}} \|q\|_{\infty, 
			\omega_{\mathcal{U}}} \qquad & &\forall q \in 
		L^{\infty}(\omega_{\mathcal{U}}).
	\end{alignat*}
	Inequality \cref{eq:l2-proj-lp-stability-local} now follows from 
	interpolation with $p=2$.

	\noindent \textbf{Step 3: \cref{eq:l2-proj-lp-stability-dg}. } If the 
	mesh satisfies \ref{cond:mesh-disjoint-sing}, then 
	\cref{lem:l2-projection-local} gives
	\begin{align*}
		\| \tilde{\mathbb{P}}_{N, \mathcal{T}} q \|_{p, \Omega}^p = 
		\sum_{\vertex{a} \in \mathcal{V}_S}  \| \tilde{\mathbb{P}}_{N, 
			\mathcal{T}_{\vertex{a}}} q \|_{p, \omega_{\vertex{a}}}^p + 
		\sum_{\substack{K \in \mathcal{T} \\ \mathcal{V}_K \cap 
				\mathcal{V}_S = \emptyset }} \| \tilde{\mathbb{P}}_{N, K} q 
		\|_{p, K}^p,
	\end{align*}
	and so \cref{eq:l2-proj-lp-stability-dg} follows from 
	\cref{eq:l2-proj-lp-stability-local}.
\end{proof}

\subsection{Proof of \cref{thm:inf-sup-stability}}
\label{sec:proof-inf-sup}
If $|\Gamma_{1}| \neq 0$, then $Q_{\Gamma}^{N} =  
\widetilde{DG}_{\Gamma}^N(\mathcal{T})$ and so 
\cref{eq:lp-norm-l2-projection} is a restatement of 
\cref{eq:l2-proj-lp-stability-dg}. Now suppose that $|\Gamma_1| = 0$ 
so $Q_{\Gamma}^{N} = \widetilde{DG}_{\Gamma}^N(\mathcal{T}) \cap 
L^1_0(\Omega)$. Let $q \in L^p(\Omega)$ with $\bar{q} = |\Omega|^{-1} 
\int_{\Omega} q \d{\bdd{x}}$. Then, condition 
\ref{cond:mesh-no-odd-sing} ensures that $\tilde{\mathbb{P}}_{N, 
	\mathcal{T}} q - \bar{q} \in Q_{\Gamma}^{N}$ and we have
\begin{align*}
	( \tilde{\mathbb{P}}_{N, \mathcal{T}} q - \bar{q},  r) = ( 
	\tilde{\mathbb{P}}_{N, \mathcal{T}} q,  r) = (q, r) \qquad 
	\forall r \in Q_{\Gamma}^{N}.
\end{align*}
Consequently, $\mathbb{P}_{N}^{Q} q = \tilde{\mathbb{P}}_{N, 
	\mathcal{T}} q - \bar{q}$, and so 
\cref{eq:lp-norm-l2-projection} follows from 
\cref{eq:l2-proj-lp-stability-dg} and H\"{o}lder's inequality.

Inequality \cref{eq:inf-sup-stability} follows immediately from 
\cref{eq:lp-norm-l2-projection,eq:inf-sup-by-discrete-norm,%
	eq:discrete-norm-by-projection}.
\hfill \qedsymbol

\section{Constructing local Fortin operators}
\label{sec:local-fortin}

We now modify the construction of $\mathcal{L}_{\dive, N}$ from 
\cref{thm:invert-div} to produce a local Fortin operator.

\subsection{Auxiliary interpolants}

We begin with a particular set of degrees of freedom that induces an 
equivalent norm that will be useful in the foregoing analysis. To this end, 
given an edge $e \in \mathcal{E}$, let $\partial_e$ denote the tangential 
derivative along $e$ and let $\mathcal{V}_e$ denote the vertices of $e$.

\begin{lemma}
	\label{lem:cg-dofs}
	For $N \geq 0$, let $\bdd{G}^N_0(K) := \{ \bdd{z} \in \mathcal{P}_N(K)^2 : 
	\bdd{z}|_{\partial K} = \bdd{0} \text{ and }  \dive \bdd{z} \equiv 0  \}$. 
	For $N \geq 4$, the following degrees of freedom are unisolvent for 
	$\bdd{v} \in \bdd{V}^N$:
	\begin{subequations}
		\label{eq:dofs-v-alt}
		\begin{alignat}{3}
			\label{eq:dofs-v-alt-1}
			&\bdd{v}(\vertex{a})  \qquad & & & &\forall \vertex{a} \in 
			\mathcal{V}, \\
			\label{eq:dofs-v-alt-2}
			&\partial_e \bdd{v}|_{e}(\vertex{a}) \qquad & &\forall \bdd{a} \in 
			\mathcal{V}_{e}, \ & &\forall e \in \mathcal{E}, \\
			\label{eq:dofs-v-alt-3}
			&\int_{e} \bdd{v} \cdot \bdd{r}  \d{s} \qquad & &\forall \bdd{r} 
			\in \mathcal{P}_{N-4}(e)^2, \ & &\forall e \in \mathcal{E}, \\
			\label{eq:dofs-v-alt-4}
			&\int_{K} \nabla \bdd{v} :  \nabla \bdd{z} \d{\bdd{x}} \qquad & 
			&\forall \bdd{z} \in \bdd{G}_{0}^N(K), \ & &\forall K \in 
			\mathcal{T}, \\
			\label{eq:dofs-v-alt-5}
			&\int_{K} (\dive \bdd{v}) q  \d{\bdd{x}} \qquad & &\forall q \in 
			Q_0^{N-1}(K), \ & &\forall K \in \mathcal{T}.
		\end{alignat}
	\end{subequations}
	Consequently, for all $K \in \mathcal{T}$, $\bdd{v} \in 
	\mathcal{P}_{N}(K)^2$, and $p \in [1,\infty]$, the norm
	\begin{align*}
		\begin{aligned}
			&\sum_{ \vertex{a} \in \mathcal{V}_K  }   \left( 
			|\bdd{v}(\vertex{a})| + h_K \sum_{e \in \mathcal{E}_{\vertex{a}} 
				\cap \mathcal{E}_K } |\partial_e \bdd{v}(\vertex{a})| \right)  + 
			h_K^{-\frac{1}{p}} \sum_{e \in \mathcal{E}_K} \sup_{ \substack{ 
					\bdd{r} \in \mathcal{P}_{N-4}(e)^2 \\ \|\bdd{r}\|_{p', e} = 1 
					} } 
			\int_{e}  \left| \bdd{v} \cdot \bdd{r}  \right| \d{s} \\
			&\quad \qquad + h_K^{1-\frac{2}{p}} \sup_{\substack{ \bdd{z} \in 
					\bdd{G}_0^{N}(K) \\ \| \nabla \bdd{z}\|_{p', K} = 1 } }  
					\int_{K} 
			\left| \nabla \bdd{v} : \nabla \bdd{z} \right| \d{\bdd{x}}  
			+ h_K^{1-\frac{2}{p}} \sup_{\substack{ q \in Q_0^{N-1}(K) \\ 
					\|q\|_{p',K} = 1 } }  \int_{K} \left| (\dive \bdd{v})q \right| 
			\d{\bdd{x}}. 
		\end{aligned}
	\end{align*}
	is equivalent to $h_K^{-\frac{2}{p}} \| \bdd{v} \|_{p, K} + 
	h_K^{1-\frac{2}{p}} 
	|\bdd{v}|_{1,p,K}$, where the equivalence constants depend only on 
	$p$, 
	$N$, 
	and shape regularity. 
	Additionally, for all $K \in \mathcal{T}$, $q \in \mathcal{P}_{N}(K)$, and 
	$p \in [1,\infty]$ the norm
	\begin{align*}
		\sum_{ \vertex{a} \in \mathcal{V}_K } | q(\vertex{a})| + h_K^{-2} 
		\left| \int_{K} q \d{\bdd{x}} \right|  + h_K^{-\frac{2}{p}} 
		\sup_{\substack{ r \in Q_0^{N-1}(K) \\ \|r\|_{p',K} = 1 } }   \left| 
		\int_{K} q r  \d{\bdd{x}}\right|
	\end{align*}
	is equivalent to $h_K^{-\frac{2}{p}} \| q \|_{p, K}$, where the 
	equivalence constants again depend only on $p$, $N$, and shape regularity. 
\end{lemma}
\begin{proof}
	The degrees of freedom in \cref{eq:dofs-v-alt} are linearly independent 
	and thanks to \cite[Theorem 3.3]{AinCP19StokesIII}, $\dim \bdd{G}_0^{N}(K) 
	= N^2/2 - 7N/2  + 6$, and so they total $\dim \bdd{V}^N$. Now suppose that 
	$\bdd{v} \in \bdd{V}^N$ and the degrees of freedom in \cref{eq:dofs-v-alt} 
	vanish. Let $K \in \mathcal{T}$. Since 
	\cref{eq:dofs-v-alt-1,eq:dofs-v-alt-2,eq:dofs-v-alt-3} vanish, 
	$\bdd{v}|_{K} \in \bdd{V}^N_0(K)$, while \cref{eq:dofs-v-alt-5} then gives 
	$\dive \bdd{v}|_{K} \equiv 0$. \Cref{eq:dofs-v-alt-4} then shows that 
	$\bdd{v}|_{K} = \bdd{0}$, and so the degrees of freedom 
	\cref{eq:dofs-v-alt} are unisolvent.
	
	The norm equivalences follow from standard scaling arguments. For brevity, 
	we only sketch the proof of the norm equivalence on $\mathcal{P}_N(K)^2$. 
	Let $\bdd{v} \in 
	\mathcal{P}_{N}(K)^2$ and define $\hat{\bdd{v}} := \bdd{v} \circ 
	\bdd{F}_K$. An equivalence of norms argument shows that
	\begin{align*}
		&\sum_{ \vertex{a} \in \mathcal{V}_{\hat{T}}  }   \left( 
		|\hat{\bdd{v}}(\vertex{a})| + \sum_{e \in \mathcal{E}_{\vertex{a}}} 
		|\partial_e \hat{\bdd{v}}(\vertex{a})| \right)  + \sum_{e \in 
			\mathcal{E}_{\hat{T}}} \sup_{ \substack{ \hat{\bdd{r}} \in 
				\mathcal{P}_{N-4}(e)^2 \\ \| \hat{\bdd{r}} \|_{p', e} = 1 } } 
		\int_{e}  \left| \hat{\bdd{v}} \cdot \hat{\bdd{r}}  \right| \d{s} \\
		&\qquad + \sup_{\substack{ \hat{\bdd{z}} \in \bdd{G}_0^{N}(\hat{T}) \\ 
				\| \nabla \hat{\bdd{z}} \|_{p', \hat{T}} = 1 } }  \int_{\hat{T}} 
		\left| \nabla \hat{\bdd{v}} : \nabla \hat{\bdd{z}} \right| 
		\d{\bdd{x}}  + \sup_{\substack{ \hat{q} \in Q_0^{N-1}(\hat{T}) \\ 
				\|\hat{q}\|_{p', \hat{T}} = 1 } }  \int_{\hat{T}} \left| (\dive 
		\hat{\bdd{v}})\hat{q} \right| \d{\bdd{x}} 
	\end{align*}
	is equivalent to $\| \hat{\bdd{v}} \|_{1, p, \hat{T}}$, where the 
	equivalence constants depend only on $p$ and $N$. Then, we obtain
	\begin{align*}
		\sup_{\substack{ \hat{\bdd{z}} \in \bdd{G}_0^{N}(\hat{T}) \\ \| \nabla 
				\hat{\bdd{z}}\|_{p', \hat{T}} = 1 } }  \int_{\hat{T}} \left| 
				\nabla 
		\hat{\bdd{v}} : \nabla \hat{\bdd{z}} \right| \d{\bdd{x}} 
		&= \sup_{\hat{\bdd{z}} \in \bdd{G}_0^{N}(\hat{T}) }  \frac{ 
			\int_{\hat{T}} \left| \nabla \hat{\bdd{v}} : \nabla \hat{\bdd{z}} 
			\right| \d{\bdd{x}} }{ \| \nabla \hat{\bdd{z}}\|_{p', \hat{T}} } \\
		&\leq C(p) \sup_{\hat{\bdd{z}} \in \bdd{G}_0^{N}(\hat{T}) }  \frac{ 
			\int_{\hat{T}} \left| \nabla \hat{\bdd{v}} : \nabla (\nabla \bdd{F}_K )
			\hat{\bdd{z}}\right| \d{\bdd{x}} }{ \| \nabla  D\bdd{F}_K 
			\hat{\bdd{z}}\|_{p', \hat{T}} } \\
		&\leq C(p) h_K^{\frac{2}{p'}-1} \sup_{\hat{\bdd{z}} \in 
			\bdd{G}_0^{N}(\hat{T}) }  \frac{ \int_{K} \left| \nabla \bdd{v} : 
			\nabla (\nabla \bdd{F}_K \hat{\bdd{z}} \circ \bdd{F}_K^{-1} )\right| 
			\d{\bdd{x}} }{ \| \nabla  (D\bdd{F}_K \hat{\bdd{z}} \circ 
			\bdd{F}_K^{-1} )\|_{p', K} } \\
		&\leq C(p)  h_K^{1-\frac{2}{p}}  \sup_{\substack{ \bdd{z} \in 
				\bdd{G}_0^{N}(K) \\ \| \nabla \bdd{z}\|_{p', K} = 1 } }  \int_{K} 
		\left| \nabla \bdd{v} : \nabla \bdd{z} \right| \d{\bdd{x}},
	\end{align*}
	where we used that the Piola transform of $\hat{\bdd{z}}$ satisfies 
	$\nabla \bdd{F}_K \hat{\bdd{z}} \circ \bdd{F}_K^{-1} \in \bdd{G}^N_0(K)$. 
	A similar use of the Piola transform shows that
	\begin{align*}
		\sup_{\substack{ \hat{q} \in Q_0^{N-1}(\hat{T}) \\ \|\hat{q}\|_{p', 
					\hat{T}} = 1 } }  \int_{\hat{T}} \left| (\dive 
					\hat{\bdd{v}})\hat{q} 
		\right| \d{\bdd{x}} \leq C(p) h_K^{1-\frac{2}{p}} \sup_{\substack{ q 
				\in Q_0^{N-1}(K) \\ \|q\|_{p',K} = 1 } }  \int_{K} \left| (\dive 
		\bdd{v})q \right| \d{\bdd{x}}.
	\end{align*}
	The other terms may be treated with an analogous scaling argument to show 
	that ``$\leq$" direction of norm equivalence. The ``$\geq$" direction is 
	similar.
\end{proof}

The next result concerns the existence and stability of a Scott-Zhang 
\cite{Scott90} operator that utilizes the degrees of freedom in 
\cref{lem:cg-dofs}.
\begin{lemma}
	Let $N \geq 4$. There exists a linear projection operator  
	$\mathcal{I}_N : 
	\bdd{W}^{1,1}(\Omega) \to \bdd{V}^N$ satisfying the following for $\bdd{v} 
	\in \bdd{W}^{1,1}(\Omega)$:
	\begin{subequations}
		\label{eq:degree-N-coarse-interp}
		\begin{alignat}{2}
			\label{eq:degree-N-coarse-interp-1}	
			\int_{e} (\bdd{v} - \mathcal{I}_N \bdd{v}) \cdot \bdd{r} \d{s} &= 
			0 \qquad & &\forall \bdd{r} \in \mathcal{P}_{N-4}(e)^2, \ \forall 
			e \in \mathcal{E}, \\
			\label{eq:degree-N-coarse-interp-2}
			\int_{K} \dive (\bdd{v} - \mathcal{I}_N \bdd{v}) q \d{\bdd{x}} &= 
			0 \qquad & &\forall q \in Q_0^{N-1}(K), \ \forall K \in 
			\mathcal{T}, \\
			\label{eq:degree-N-coarse-interp-zero-trace}
			\mathcal{I}_N \bdd{v}|_{\Gamma_0} &= \bdd{0} \qquad & &\text{if } 
			\bdd{v} \in \bdd{W}^{1,1}_{\Gamma}(\Omega).
		\end{alignat}	
	\end{subequations}
	Moreover, for all $p \in [1, \infty]$ and $\bdd{v} \in 
	\bdd{W}^{1,p}(\Omega)$, there holds
	\begin{subequations}
		\begin{alignat}{2}
			\label{eq:degree-N-coarse-interp-lp-stable}
			\| \mathcal{I}_N \bdd{v} \|_{p, K} &\leq C(N, p) \left( \|  
			\bdd{v} \|_{p, \omega_K} + h_K |  \bdd{v} |_{1, p, \omega_K} 
			\right) \qquad & &\forall K \in \mathcal{T}, \\
			\label{eq:degree-N-coarse-interp-w1p-stable}
			| \mathcal{I}_N \bdd{v} |_{1, p, K} &\leq C(N, p)  |  \bdd{v} 
			|_{1, p, \omega_K}  \qquad & &\forall K \in \mathcal{T}, \\
			\label{eq:coarse-interp-div-stab}
			\| \dive   \mathcal{I}_N \bdd{v} \|_{p, K} &\leq C(N, p) | \bdd{v} 
			|_{1, p, K} \qquad & &\forall K \in \mathcal{T},
		\end{alignat}
	\end{subequations}
	where $C(N, p)$ also depends on $\Omega$, $\Gamma_0$, and shape regularity.
\end{lemma}
\begin{proof}
	For each $\vertex{a} \in \mathcal{V}$, let $e_{\bdd{a}} \in 
	\mathcal{E}_{\vertex{a}}$ be chosen so that $e_{\bdd{a}} \subset 
	\Gamma_0$  if $\vertex{a} \in \mathcal{V}_{B, 0}$ and let $\bdd{\Pi}_{e} : 
	\bdd{L}^{1}(e) \to \mathcal{P}_N(e)^2$ denote the $L^2(e)$ projection 
	operator
	\begin{align*}
		( \bdd{\Pi}_{e} \bdd{w}, \bdd{r} )_{e} = (\bdd{w}, \bdd{r} )_{e} 
		\qquad \forall \bdd{r} \in \mathcal{P}_N(e)^2, \ \forall \bdd{w} \in 
		\bdd{L}^{1}(e).
	\end{align*}
	Similar arguments to those in the proof of 
	\cref{lem:l2-proj-lp-stability-local} show that for all $p \in 
	[1,\infty]$, there holds
	\begin{align}
		\label{eq:proof:l2-proj-edge-stable}
		\| \bdd{\Pi}_{e} \bdd{w} \|_{p, e} \leq C(N, p) \| \bdd{w} \|_{p, e} 
		\qquad \forall \bdd{w} \in \bdd{L}^p(e).  
	\end{align}
	
	Let $\bdd{v} \in \bdd{W}^{1,1}(\Omega)$ be given and define 
	$\mathcal{I}_{N} \bdd{v} \in \bdd{V}^N$ by assigning the degrees of 
	freedom in \cref{eq:dofs-v-alt} as follows:
	\begin{subequations}
		\begin{alignat}{3}
			\mathcal{I}_{N} \bdd{v}(\vertex{a}) &= (\bdd{\Pi}_{e_{\vertex{a}}} 
			\bdd{v})(\bdd{a})  \qquad & & & &\forall \vertex{a} \in 
			\mathcal{V}, \\
			\partial_e \mathcal{I}_{N} \bdd{v}|_{e}(\vertex{a}) &= \partial_e 
			(\bdd{\Pi}_{e} \bdd{v}|_{e} )(\vertex{a}) \qquad & &\forall 
			\bdd{a} \in \mathcal{V}_{e}, \ & &\forall e \in \mathcal{E}, \\
			\int_{e} \mathcal{I}_{N} \bdd{v} \cdot \bdd{r}  \d{s} &= \int_{e} 
			\bdd{v} \cdot \bdd{r}  \d{s}  \qquad & &\forall \bdd{r} \in 
			\mathcal{P}_{N-4}(e)^2, \ & &\forall e \in \mathcal{E}, \\
			\int_{K} \nabla \mathcal{I}_{N} \bdd{v} :  \nabla \bdd{z} 
			\d{\bdd{x}} &= \int_{K} \nabla \bdd{v} :  \nabla \bdd{z} 
			\d{\bdd{x}} \qquad & &\forall \bdd{z} \in \bdd{G}_{0}^N(K), \ & 
			&\forall K \in \mathcal{T}, \\
			\int_{K} (\dive \mathcal{I}_{N} \bdd{v}) q  \d{\bdd{x}} &= 
			\int_{K} (\dive \bdd{v}) q  \d{\bdd{x}} \qquad & &\forall q \in 
			Q_0^{N-1}(K), \ & &\forall K \in \mathcal{T}.
		\end{alignat}
	\end{subequations}
	Properties \cref{eq:degree-N-coarse-interp-1,eq:degree-N-coarse-interp-2} 
	follow immediately by construction. If $\bdd{v} \in 
	\bdd{W}^{1,1}_{\Gamma}(\Omega)$, then $\bdd{\Pi}_{e_{\vertex{a}}} \bdd{v} 
	= \bdd{0}$ if $a \in \mathcal{V}_{B, 0}$ since $e_{\vertex{a}} \subset 
	\Gamma_0$. Moreover, $\bdd{\Pi}_e \bdd{v}|_{e} = \bdd{0}$ for $e \subset 
	\Gamma_0$, and so for every $e \in \mathcal{E}$ with $e \subset \Gamma_0$, 
	there holds
	\begin{align*}
		\partial_e^{j} \mathcal{I}_N \bdd{v}(\vertex{a}) = \bdd{0}  \quad 
		\forall \vertex{a} \in \mathcal{V}_e, \ j \in \{0, 1\}, \quad 
		\text{and} \quad \int_{e} \mathcal{I}_N \bdd{v} \cdot \bdd{r} \d{s} = 
		0 \quad \forall \bdd{r} \in \mathcal{P}_{N-4}(e)^2.
	\end{align*}
	Thus, $\mathcal{I}_N \bdd{v}|_{e} = \bdd{0}$, and 
	\cref{eq:degree-N-coarse-interp-zero-trace} follows.
	
	We now turn to 
	\cref{eq:degree-N-coarse-interp-lp-stable,eq:degree-N-coarse-interp-w1p-stable}.
	Standard inverse estimates, \cref{eq:proof:l2-proj-edge-stable}, and the 
	trace theorem give
	\begin{align*}
		&\sum_{ \vertex{a} \in \mathcal{V}_K  }   \left( 
		|(\bdd{\Pi}_{e_{\bdd{a}}} \bdd{v})(\vertex{a})| + h_K \sum_{e \in 
			\mathcal{E}_{\vertex{a}} \cap \mathcal{E}_K } |\partial_e 
		(\bdd{\Pi}_{e} \bdd{v})(\vertex{a})| \right) \\
		&\qquad \leq C(N, p) h_K^{-\frac{1}{p}} \sum_{ \vertex{a} \in 
			\mathcal{V}_K  } \left( \| \bdd{\Pi}_{e_{\vertex{a}}} \bdd{v} \|_{p, 
			e_{\vertex{a}}} + \sum_{e \in \mathcal{E}_{\vertex{a}} \cap 
			\mathcal{E}_K } \| \bdd{\Pi}_{e} \bdd{v} \|_{p, e} \right) \\
		&\qquad \leq C(N, p) h_K^{-\frac{1}{p}} \sum_{ \vertex{a} \in 
			\mathcal{V}_K  } \left( \|  \bdd{v} \|_{p, e_{\vertex{a}}} + \sum_{e 
			\in \mathcal{E}_{\vertex{a}} \cap \mathcal{E}_K } \| \bdd{v} \|_{p, e} 
		\right) \\
		&\qquad \leq C(N, p) \left( h_K^{-\frac{2}{p}} \| \bdd{v} \|_{p, 
			\omega_K} + h_K^{1-\frac{2}{p}} |\bdd{v}|_{1,p,\omega_{K}} \right)
	\end{align*}
	and $h_K^{-\frac{1}{p}} \| \bdd{v}\|_{p, \partial K} \leq C(p) 
	(h_K^{-\frac{2}{p}} \| \bdd{v} \|_{p, K} + h_K^{1-\frac{2}{p}} 
	|\bdd{v}|_{1,p,K})$. Thus, we obtain
	\begin{align*} 
		h_K^{-\frac{2}{p}} \|\mathcal{I}_{N} \bdd{v} \|_{p, K} + 
		h_K^{1-\frac{2}{p}} |\mathcal{I}_{N} \bdd{v} |_{1,p,K}  &\leq C(N, p) 
		\left(  h_K^{-\frac{2}{p}} \| \bdd{v} \|_{p, \omega_K} + 
		h_K^{1-\frac{2}{p}} |\bdd{v}|_{1,p,\omega_{K}} \right)
	\end{align*}
	thanks to \cref{lem:cg-dofs}. Inequality 
	\cref{eq:degree-N-coarse-interp-lp-stable} now follows. Since 
	$\mathcal{I}_N \bdd{c} = \bdd{c}$ for $\bdd{c} \in \mathbb{R}^2$, 
	inequality \cref{eq:degree-N-coarse-interp-w1p-stable} follows from 
	Poincar\'{e}'s inequality:
	\begin{align*}
		|\mathcal{I}_{N} \bdd{v} |_{1,p,K} = \newinf_{ \bdd{c} \in 
			\mathbb{R}^2}   |\mathcal{I}_{N} (\bdd{v} - \bdd{c}) |_{1,p,K} &\leq 
		C(N, p) \left( h_K^{-1} \inf_{\bdd{c} \in \mathbb{R}^2} \| \bdd{v} - 
		\bdd{c} \|_{p, \omega_K} + 
		|\bdd{v}|_{1,p,\omega_{K}} \right) \\
		&\leq C(N, p) |\bdd{v}|_{1,p,\omega_{K}}. 
	\end{align*}

	Finally, we turn to \cref{eq:coarse-interp-div-stab}. Note that
	\begin{align*}
		h_K \sum_{ \vertex{a} \in \mathcal{V}_K } |\nabla \mathcal{I}_{N} 
		\bdd{v}(\vertex{a})| 
		&\leq C 
		h_K \sum_{ \vertex{a} \in \mathcal{V}_K } \sum_{e \in 
			\mathcal{E}_{\vertex{a}}} |\partial_{e} \mathcal{I}_{N} 
		\bdd{v}(\vertex{a})| \\
		&\leq C(N, p) \left( 
		h_K^{-\frac{2}{p}} \| \bdd{v} \|_{p,K} + h_K^{1-\frac{2}{p}} 
		|\bdd{v}|_{1,p,K} \right),
	\end{align*}
	and a similar argument using Poincar\'{e}'s inequality gives
	\begin{align*}
		h_K \sum_{ \vertex{a} \in \mathcal{V}_K } |\nabla \mathcal{I}_{N} 
		\bdd{v}(\vertex{a})|  \leq C(N, p) h_K^{1-\frac{2}{p}} 
		|\bdd{v}|_{1,p,K}.
	\end{align*}
	Moreover, \cref{eq:degree-N-coarse-interp-1} and the divergence theorem 
	show that
	\begin{align*}
		\int_{K} \dive \mathcal{I}_{N} \bdd{v} \d{\bdd{x}} = \sum_{e \in 
			\mathcal{E}_K} \int_{e} \mathcal{I}_{N} \bdd{v} \cdot \unitvec{n} 
		\d{s} = \sum_{e \in \mathcal{E}_K} \int_{e} \bdd{v} \cdot \unitvec{n} 
		\d{s} =  \int_{K} \dive \bdd{v} \d{\bdd{x}}.
	\end{align*}
	Since $\dive \mathcal{I}_N \bdd{v}|_{K} \in \mathcal{P}_{N-1}(K)$, we 
	apply \cref{lem:cg-dofs,eq:degree-N-coarse-interp-2}  to obtain
	\begin{align*}
		&\| \dive \mathcal{I}_N \bdd{v} \|_{p, K} \\
		&\qquad \leq C(N, p) \left( h_K^{\frac{2}{p}}  \sum_{ \vertex{a} \in 
			\mathcal{V}_K } |\nabla \mathcal{I}_{N} \bdd{v}(\vertex{a})| + 
		h_K^{\frac{2}{p}-2 } \left| \int_{K} \dive \bdd{v} \d{\bdd{x}} \right| 
		+ \|\dive \bdd{v} \|_{p, K} \right) \\
		&\qquad \leq C(N, p) \left( |\bdd{v}|_{1,p,K} + \|\dive \bdd{v} \|_{p, K} 
		\right), 
	\end{align*}
	which completes the proof of \cref{eq:coarse-interp-div-stab}.
\end{proof}

\subsection{Proof of \cref{thm:fortin-nonprojection}}

Let $\bdd{v} \in \bdd{W}^{1,1}(\Omega)$ be given.

\noindent \textbf{Step 1: Interpolate the average divergence. } Let 
$q_0 \in 
DG^{N-1}(\mathcal{T})$ be defined by $q_0 := \tilde{\mathbb{P}}_{N-1} 
(\dive \bdd{v}) - \dive \mathcal{I}_{4} \bdd{v}$. Note that the mesh 
condition \ref{cond:mesh-no-sing} means that $\tilde{\mathbb{P}}_{N-1}$ is 
the $L^2$-projection onto $DG^{N-1}(\mathcal{T})$. Then, $q_0 \in Q^{N-1} 
\cap L^1_0(\Omega) \subseteq Q^{N-1}_{\Gamma}$ since
\begin{align}
	\label{eq:proof:fortin-step1-interp-avg-div}
	\int_{K} q_0 \d{\bdd{x}} = \int_{K} \left( \dive \bdd{v} - \dive 
	\mathcal{I}_{4} \bdd{v} \right) \d{\bdd{x}} = \int_{\partial K} \left( 
	\bdd{v} - \mathcal{I}_{4} \bdd{v} \right) \cdot \unitvec{n} \d{s} = 0 
	\qquad \forall K \in \mathcal{T},
\end{align}
where we used \cref{eq:degree-N-coarse-interp-1}.

\noindent \textbf{Step 2: Interpolate the divergence at the vertices. 
} Thanks to  
\cref{lem:interp-div-vertices}, the vector field  $\bdd{\psi} \in 
\bdd{V}^N_{\Gamma}$ given by $\bdd{\psi} := \sum_{\vertex{a} \in 
	\mathcal{V}} \mathcal{L}_{N, \vertex{a}} q_0$ is well-defined since $q_0 
\in Q_{\Gamma}^{N-1}$, and $\bdd{\psi}$ satisfies the following for all $K 
\in \mathcal{T}$:
\begin{alignat*}{2}
	\dive \bdd{\psi}|_{K}(\vertex{a}) &= q_0|_{K}(\vertex{a}) \qquad & 
	&\forall \vertex{a} \in \mathcal{V}_K, \\
	(\dive \bdd{\psi}, 1)_{K} &= 0, \qquad & &\\
	h_K^{-1} \| \bdd{\psi} \|_{p, K} + |\bdd{\psi}|_{1,p,K} &\leq 
	\begin{cases}
		C(p)  \xi_{\mathcal{T}}^{-1} \|q_0\|_{p, \omega_K} & \text{if } p 
		\in (1,\infty), \\
		C(\epsilon) N^{\epsilon} \xi_{\mathcal{T}}^{-1} \|q_0\|_{p, 
			\omega_K} & \text{if } p \in \{ 1, \infty \}.
	\end{cases}
\end{alignat*}

\noindent \textbf{Step 3: Correct the interior divergence. } Let $q_1 
:= q_0 - 
\dive \bdd{\psi}$. Then, $\dive q_1(\vertex{a}) = \bdd{0}$ for all 
$\vertex{a} \in \mathcal{V}$ and 
\cref{eq:proof:fortin-step1-interp-avg-div} shows that $q_1|_{K} \in 
Q_0^{N-1}(K)$. Appealing to \cref{cor:invert-div-element-phys}, the 
vector field $\bdd{v}_K := \mathcal{L}_{\dive, K} q_1|_{K} \in 
\bdd{V}_0^N(K)$ is well-defined for all $K \in \mathcal{T}$ and satisfies $\dive 
\bdd{v}_K = q_1|_{K}$ and
\begin{align*}
	 h_K^{-1} \| \bdd{v}_K 
	\|_{p, K} + |\bdd{v}_K|_{1,p,K} \leq \begin{cases}
		C(p)  \xi_{\mathcal{T}}^{-1} \|q_1\|_{p, K} & \text{if } p \in 
		(1,\infty), \\
		C(\epsilon) N^{\epsilon} \xi_{\mathcal{T}}^{-1} \|q_1\|_{p, K} & 
		\text{if } p \in \{ 1, \infty \}.
	\end{cases}
\end{align*}
We then define
\begin{align*}
	\mathcal{L}_{F, N} \bdd{v} := \mathcal{I}_{4} \bdd{v} + \bdd{\psi} + 
	\bdd{v}_K  \qquad \text{on } K.
\end{align*}
Note that $\bdd{v} \in \bdd{V}^N$ and if $\bdd{v}|_{\Gamma_0} = \bdd{0}$, 
then $\mathcal{L}_{F, N} \bdd{v}|_{\Gamma_0} = \bdd{0}$ thanks to 
\cref{eq:degree-N-coarse-interp-zero-trace}. Additionally, if $\bdd{v} \in 
\bdd{V}^4$, then $\mathcal{L}_{F, N} \bdd{v} = \bdd{v}$ since $q_0 \equiv 
0$ and so $\bdd{\psi} \equiv \bdd{0}$ and $\bdd{v}_K \equiv \bdd{0}$ for 
all $K \in \mathcal{T}$. Moreover, on $K \in \mathcal{T}$, there holds
\begin{align*}
	\dive \mathcal{L}_{F, N} \bdd{v}|_{K} = \dive (\mathcal{I}_{4} \bdd{v} 
	+ \bdd{\psi} + \bdd{v}_K)|_{K} = \tilde{\mathbb{P}}_{N-1, K} (\dive 
	\bdd{v}) = (\tilde{\mathbb{P}}_{N-1} \dive \bdd{v})|_{K},
\end{align*} 
and \cref{eq:fortin-nonprojection-l2-projection} follows. 

\noindent \textbf{Step 4: Continuity. } Suppose now that $\bdd{v} \in 
\bdd{W}^{1,p}(\Omega)$ for some $p \in [1,\infty]$. If $p \in (1,\infty)$, 
then \cref{eq:coarse-interp-div-stab,eq:l2-proj-lp-stability-local} give 
the following on $K \in \mathcal{T}$ and for $p \in (1,\infty)$:
\begin{align*}
	\| \mathcal{L}_{F, N} \bdd{v} \|_{p, K} &\leq \| \mathcal{I}_{4} 
	\bdd{v}\|_{p, K} + \| \bdd{\psi}\|_{p, K} + \|\bdd{v}_K\|_{p, K} \\
	&\leq C(p) \left( \| \mathcal{I}_{4} \bdd{v}\|_{p, K} + h_K 
	\xi_{\mathcal{T}}^{-1} \|q_0\|_{p, \omega_{K}} \right) \\
	&\leq C(p) \left\{  \| \mathcal{I}_{4} \bdd{v}\|_{p, K} +  h_K 
	\xi_{\mathcal{T}}^{-1} \left(  N^{3\left| \frac{1}{2} - \frac{1}{p} 
		\right|} \| \dive \bdd{v} \|_{p, \omega_{K}} + \| \dive   
	\mathcal{I}_{4} \bdd{v}\|_{p, \omega_K} \right) \right\} \\
	&\leq C(p) \left( \|\bdd{v}\|_{p, \omega_K} + h_K 
	\xi_{\mathcal{T}}^{-1} N^{3\left| \frac{1}{2} - \frac{1}{p} \right|} 
	|v|_{1,p,\omega_K} \right),
\end{align*}
and so \cref{eq:fortin-nonprojection-lp-local} follows. The case $p \in 
\{1,\infty\}$ and the remaining bound 
\cref{eq:fortin-nonprojection-w1p-local} for $|\mathcal{L}_{F, N} 
\bdd{v}|_{1, p, K}$ follow from similar arguments. \hfill \qedsymbol

\subsection{Construction of $\mathcal{I}_{N, F}$}

The construction is the same as that of $\mathcal{L}_{F, N}$ above, 
replacing the 
use of $\mathcal{I}_4$ with $\mathcal{I}_N$. The rest of the proof is 
identical. \hfill \qedsymbol

\appendix

\section{Inverting the divergence operator with mixed boundary conditions}

\begin{lemma}
	\label{lem:invert-div-cont}
	There exists a linear operator
	\begin{align*}
		\mathcal{L}_{\dive, \Gamma_0} : \bigcup_{1 < p < \infty} 
		L_{\Gamma}^p(\Omega) \to W^{1,1}_{\Gamma}(\Omega)
	\end{align*}	
	satisfying the following: For all $p \in (1,\infty)$ and $q \in 
	L^p_{\Gamma}(\Omega)$, there holds $\mathcal{L}_{\dive} q \in 
	\bdd{W}^{1,p}_{\Gamma}(\Omega)$ with
	\begin{align}
		\label{eq:invert-div-cont}
		\dive \mathcal{L}_{\dive, \Gamma_0} q = q \quad \text{and} \quad \| 
		\mathcal{L}_{\dive, \Gamma_0} q \|_{1, p, \Omega} \leq C(\Omega, 
		\Gamma_0, p) \|q\|_{p, \Omega}.
	\end{align}
\end{lemma}
\begin{proof}
	If $|\Gamma| = |\Gamma_0|$, then the result is a restatement of 
	\cite[Theorem 3.1]{Arn88}, so suppose that $|\Gamma_1| > 0$. Let $\gamma 
	\subset \Gamma_1$ be an open segment in $\Gamma_1$ and let $\phi \in 
	C^{\infty}_c(\gamma)$ be chosen so that $\int_{\gamma} \phi \d{s} = 1$. 
	Let $\tilde{\phi}$ be the extension by zero of $\phi$ to all of $\partial 
	\Omega$. By \cite[Theorem 6.1]{Arn88}, there exists $\bdd{\psi} \in 
	\cup_{1 < p < \infty} \bdd{W}^{1,p}(\Omega)$ satisfying
	\begin{align*}
		\bdd{\psi}|_{\gamma} = \phi \unitvec{n}, \quad \bdd{\psi}|_{\partial 
			\Omega \setminus \gamma} = \bdd{0}, \quad \text{and} \quad 
		\|\bdd{\psi}\|_{1,p} \leq C(\Omega, p) \| \tilde{\phi} 
		\|_{1-\frac{1}{p}, p, \partial \Omega} \quad \forall p \in (1,\infty).
	\end{align*}
	For $p \in (1,\infty)$ and $q \in L^p_{\Gamma}(\Omega)$, let $\bar{q} = 
	|\Omega|^{-1} \int_{\Omega} q \d{\bdd{x}}$ and define
	\begin{align*}
		\mathcal{L}_{\dive, \Gamma_0} q := \bar{q} \bdd{\psi} + 
		\mathcal{L}_{\dive, \partial \Omega} \left( q -  \bar{q} \dive 
		\bdd{\psi} \right).
	\end{align*}
	Note that $\mathcal{L}_{\dive, \Gamma_0} q$ is well-defined since 
	$\int_{\Omega} \dive \bdd{\psi} \d{\bdd{x}} = \int_{\gamma} \phi \d{s} = 
	1$ and so $q - \bar{q} \dive \bdd{\psi} \in L^p_0(\Omega)$. Properties 
	\cref{eq:invert-div-cont} now follows from the properties of 
	$\mathcal{L}_{\dive, \partial \Omega}$ and H\"{o}lder's inequality.
\end{proof}

\section{Rapidly decaying 1D polynomials}

Let $I := (-1, 1)$. For $N \in \mathbb{N}_0$, $\alpha, \beta > -1$, let 
$P_N^{(\alpha, \beta)}$ denote the usual Jacobi polynomial \cite{Szego75}, 
normalized so that
\begin{align}
	\label{eq:jacobi-normalization}
	P_N^{(\alpha, \beta)}(1) = \binom{N + \alpha}{N} = 
	\frac{\Gamma(N+\alpha+1)}{\Gamma(\alpha + 1) N!}
\end{align}
The first result concerns the existence of 1D vertex functions that 
interpolate derivative values at an endpoint and decay rapidly in all $L^p(I)$ 
spaces, improving upon \cite[Lemma B.1]{AinCP19Precon} which only concerned 
decay in $L^2(I)$.
\begin{lemma}
	\label{lem:1d-interp-poly}
	Let $m \in \mathbb{N}_0$, $\alpha > 2m + \frac{3}{2}$, and $N \geq 2m+1$. 
	Then, the polynomial
	\begin{align}
		\label{eq:1d-interp-poly-def}
		\zeta_{m, N}(t) := \frac{1}{2^{m+1} P^{(\alpha, \alpha)}_{N-2m-1}(-1)} 
		(1-t)^{m+1} (1+t)^{m} P^{(\alpha, \alpha)}_{N-2m-1}(t), \qquad t \in 
		\bar{I},
	\end{align}
	satisfies $\partial_t^{n} \zeta_{m, N}(-1) = \delta_{mn}$ and 
	$\partial_t^n \zeta_{m, N}(1) = 0$ for $n \in \{0,1,\ldots, m\}$. 
	Additionally, for $p \in [1,\infty]$ and $s \in [0, \infty)$, there holds
	\begin{align}
		\label{eq:1d-interp-poly-sobolev}
		\| \zeta_{m, N} \|_{s, p, I} \leq C(m, \alpha, s, p) 
		N^{2(s-m-\frac{1}{p})}.
	\end{align}
\end{lemma}
\begin{proof}
	The endpoint properties of $\zeta_{m, N}$ follow by construction. Thanks 
	to the symmetry of the Jacobi polynomials, there holds
	\begin{align*}
		\int_{-1}^{1} (1-t)^{m+1} (1+t)^{m} \left| P^{(\alpha, 
			\alpha)}_{N-2m-1}(t) \right| \d{t} &\leq 2 \int_{-1}^{1} (1-t^2)^m 
		\left| P^{(\alpha, \alpha)}_{N-2m-1}(t) \right| \d{t} \\
		&\leq 2^{m+1} \int_{0}^{1} (1-t)^m  \left| P^{(\alpha, 
			\alpha)}_{N-2m-1}(t) \right| \d{t}.
	\end{align*}
	Theorem 7.34 of \cite{Szego75} then gives
	\begin{align*}
		\int_{0}^{1} (1-t)^{m}  \left| P^{(\alpha, \alpha)}_{N-2m-1}(t) 
		\right| \d{t} \leq C(\alpha) (N-2m-1)^{\alpha-2(m+1)}.
	\end{align*}
	Applying Sterling's formula, we obtain
	\begin{align*}
		\| \zeta_{m, N} \|_{1, I} &\leq C(\alpha) \frac{\Gamma(\alpha + 1) 
			(N-2m-1)!}{\Gamma(N-2m+\alpha)} (N-2(m+1))^{\alpha-2(m+1)} \\
		&\leq C(m, \alpha)  N^{-2(m+1)}. 
	\end{align*}
	Inequality \cref{eq:1d-interp-poly-sobolev} then follows from polynomial 
	inverse inequalities; see e.g. \cite[Theorem 5.1]{Bern07}.	 
\end{proof}

We now modify the above construction for $m=1$ to also have zero mean.
\begin{corollary}
	Let $\alpha > 7/2$ and $N \geq 4$. Then, the function
	\begin{align}
		\label{eq:1d-interp-poly-noavg-def}
		\tilde{\Upsilon}_N(t) := \zeta_{1, N}(t) - \frac{15}{16} (1-t^2)^2 
		\int_{-1}^{1} \zeta_{1, N}(s) \d{s}, \qquad t \in \bar{I},
	\end{align}
	satisfies 
	\begin{align}
		\label{eq:1d-interp-poly-noavg}
		\partial_t^{m} \tilde{\Upsilon}_N(-1) = \delta_{m1}, \quad 
		\partial_t^m \tilde{\Upsilon}_N(1) = 0, \quad m \in \{0, 1\}, \quad 
		\text{and} \quad \int_{-1}^{1} \tilde{\Upsilon}_N(t) \d{t} = 0.
	\end{align}
	Additionally, for $p \in [1,\infty]$ and $s \in [0,\infty)$, there holds
	\begin{align}
		\label{eq:1d-interp-poly-noavg-sobolev}
		\| \tilde{\Upsilon}_N \|_{s, p, I} \leq C(\alpha, s, p) 
		N^{2(s-1-\frac{1}{p})}.
	\end{align}
\end{corollary}
\begin{proof}
	Properties \cref{eq:1d-interp-poly-noavg} follow by construction from 
	\cref{lem:1d-interp-poly}, while \cref{eq:1d-interp-poly-noavg-sobolev} 
	follows from \cref{eq:1d-interp-poly-sobolev} and the bound $\| 
	\tilde{\Upsilon}_N \|_{s, p, I} \leq C(\alpha, s, p) (\|\zeta_{1, N}\|_{s, 
		p, I} + \| \zeta_{1, N} \|_{1, I})$.  
\end{proof}

\providecommand{\bysame}{\leavevmode\hbox to3em{\hrulefill}\thinspace}
\providecommand{\MR}{\relax\ifhmode\unskip\space\fi MR }
\providecommand{\MRhref}[2]{%
	\href{http://www.ams.org/mathscinet-getitem?mr=#1}{#2}
}
\providecommand{\href}[2]{#2}

\end{document}